\documentclass[letterpaper,reqno,11pt,final]{amsart}

\usepackage{amsmath,amsthm,amsfonts,amssymb}
\usepackage{enumitem,comment}
\usepackage[usenames,dvipsnames]{color}
\usepackage{xstring}
\usepackage{mathtools}
\usepackage[extension=pdf]{hyperref}
\usepackage[totalheight=9.6in,totalwidth=5.95in,centering]{geometry}
\usepackage{graphics,graphicx}
\usepackage[color,notref,notcite]{showkeys}
\usepackage[style=alphabetic,sorting=nty,natbib=true,maxnames=99,isbn=false,doi=false,url=false,firstinits=true,hyperref=auto,arxiv=abs,backend=bibtex]{biblatex}
\usepackage[color]{showkeys}%,notref,notcite

\DeclareFieldFormat[article,inbook,incollection,inproceedings,patent,thesis,unpublished]{title}{#1\isdot}
\renewbibmacro{in:}{%
  \ifentrytype{article}{}{%
    \printtext{\bibstring{in}\intitlepunct}}} 
\defbibheading{apa}[\refname]{\section*{#1}}
\DeclareFieldFormat{sentencecase}{\MakeSentenceCase{#1}} \renewbibmacro*{title}{%
  \ifthenelse{\iffieldundef{title}\AND\iffieldundef{subtitle}}{}
  {\ifthenelse{\ifentrytype{article}\OR\ifentrytype{inbook}%
      \OR\ifentrytype{incollection}\OR\ifentrytype{inproceedings}%
      \OR\ifentrytype{inreference}} {\printtext[title]{%
        \printfield[sentencecase]{title}%
        \setunit{\subtitlepunct}%
        \printfield[sentencecase]{subtitle}}}%
    {\printtext[title]{%
        \printfield[titlecase]{title}%
        \setunit{\subtitlepunct}%
        \printfield[titlecase]{subtitle}}}%
    \newunit}%
  \printfield{titleaddon}}
\AtEveryBibitem{\clearlist{language}}

\definecolor{darkblue}{rgb}{0.13,0.13,0.39}%{0.03,0.03,0.265}
\hypersetup{colorlinks=true,urlcolor=darkblue,citecolor=darkblue,linkcolor=darkblue,%
  pdftitle=Argmax of B.M.}

\mathtoolsset{showmanualtags,showonlyrefs}

\newtheorem{thm}{Theorem}[section] 
\newtheorem{lem}[thm]{Lemma}
\newtheorem{prop}[thm]{Proposition} \newtheorem{cor}[thm]{Corollary}
 
\theoremstyle{definition}  \newtheorem*{rem*}{Remark}
 
 \newcounter{assum}

\newcommand{\I}{{\rm i}} \newcommand{\pp}{\mathbb{P}} 
  \newcommand{\rr}{\mathbb{R}}
\newcommand{\nn}{\mathbb{N}} \newcommand{\zz}{\mathbb{Z}} \newcommand{\aip}{\mathcal{A}_2}

  \newcommand{\ct}{\mathcal{T}}

\newcommand{\cm}{\mathcal{M}}

\newcommand{\p}{\partial}
\newcommand{\uno}[1]{\mathbf{1}_{#1}}
\newcommand{\ep}{\varepsilon}
\newcommand{\vs}{\vspace{6pt}}
\newcommand{\wt}{\widetilde}

\newcommand{\K}{{\sf K}_{\Ai}}  
\newcommand{\qand}{\quad\text{and}\quad}
\newcommand{\qqand}{\qquad\text{and}\qquad}
\newcommand{\ts}{\hspace{0.1em}}
\newcommand{\tts}{\hspace{0.05em}}
\newcommand{\tsm}{\hspace{-0.1em}}
\newcommand{\ttsm}{\hspace{-0.05em}}
\newcommand{\sI}{{\sf I}}

\DeclareMathOperator{\Ai}{Ai}
\DeclareMathOperator{\tr}{tr}
\DeclareMathOperator{\sech}{sech}

\newcommand{\KGN}{{\sf K}_{N}^{\rm bb}}
\newcommand{\wtKGN}{{\wt {\sf K}}_{N}^{\rm bb}}
\newcommand{\sKGN}{{\sf K}_{N}^{\star}}
\newcommand{\oKGN}{{\sf K}^{\rm be}_{N}}

\newcommand{\eKGN}{{\sf K}^{\rm rbb}_{N}}
\newcommand{\oeKGN}{{\sf K}^{{\rm be}/{\rm rbb}}_{N}}

\DeclareMathOperator*{\argmax}{argmax}

\makeatletter
\newcommand{\subalign}[1]{%
  \vcenter{%
    \Let@ \restore@math@cr \default@tag
    \baselineskip\fontdimen10 \scriptfont\tw@
    \advance\baselineskip\fontdimen12 \scriptfont\tw@
    \lineskip\thr@@\fontdimen8 \scriptfont\thr@@
    \lineskiplimit\lineskip
    \ialign{\hfil$\m@th\scriptstyle##$&$\m@th\scriptstyle{}##$\crcr
      #1\crcr
    }%
  }
}
\makeatother

\def\dash---{\kern.16667em---\penalty\exhyphenpenalty\hskip.16667em\relax}

\numberwithin{equation}{section}

\let\oldmarginpar\marginpar
\renewcommand\marginpar[1]{\-\oldmarginpar[\raggedleft\footnotesize #1]%
  {\raggedright{\small\textsf{#1}}}}

\allowdisplaybreaks[2]

\begin{document}

\title{Extreme statistics of non-intersecting Brownian paths}

\author[G.~B.~Nguyen]{Gia Bao Nguyen} 
\address{G.~B.~Nguyen\\
  Centro de Modelamiento Matem\'atico\\
  Universidad de Chile\\
  Av. Beauchef 851, Torre Norte\\
  Santiago\\
  Chile} \email{bnguyen@dim.uchile.cl}

\author[D.~Remenik]{Daniel Remenik}
\address{D.~Remenik\\
  Departamento de Ingenier\'ia Matem\'atica and Centro de Modelamiento Matem\'atico\\
  Universidad de Chile\\
  Av. Beauchef 851, Torre Norte\\
  Santiago\\
  Chile} \email{dremenik@dim.uchile.cl}

\begin{abstract}

We consider finite collections of $N$ non-intersecting Brownian paths on the line and on the half-line with both absorbing and reflecting boundary conditions (corresponding to Brownian excursions and reflected Brownian motions) and compute in each case the joint distribution of the maximal height of the top path and the location at which this maximum is attained.
The resulting formulas are analogous to the ones obtained in \cite{mqr} for the joint distribution of $\cm=\max_{x\in\rr}\!\big\{\aip(x)-x^2\}$ and $\ct=\argmax_{x\in\rr}\!\big\{\aip(x)-x^2\}$, where $\aip$ is the Airy$_2$ process, and we use them to show that in the three cases the joint distribution converges, as $N\to\infty$, to the joint distribution of $\cm$ and $\ct$.
In the case of non-intersecting Brownian bridges on the line, we also establish small deviation inequalities for the argmax which match the tail behavior of $\ct$.
Our proofs are based on the method introduced in \cite{cqr,bcr} for obtaining formulas for the probability that the top line of these line ensembles stays below a given curve, which are given in terms
of the Fredholm determinant of certain ``path-integral'' kernels.
\end{abstract}

\maketitle

\section{Introduction and main results}\label{sec:intro}

Consider a collection of $N$ Brownian bridges $(B_1(t),\dotsc,B_N(t))_{t\in[0,1]}$, starting and ending at the origin, and condition them (in the sense of Doob) to not intersect in the region $t\in(0,1)$.
We will refer to this model as \emph{non-intersecting Brownian bridges}, and we will always assume that the paths are ordered so that $B_1(t)<\dotsm<B_N(t)$ for $t\in(0,1)$.
This model together with its many variants have been studied intensively in the last decade or so, both in the probability and statistical physics literatures (see for instance \cite{tracyWidomDysonBM,adlerVanMoerbeke-PDEs,tracyWidom-Pearcey,kobIzumKat,SMCR,delvauxKuijlaarsZhang,ferrariVeto-Tacnode,forrester,liechty-nibmdope,schehr,johansson-BMtacnode,ferrariVetoHardEdge,liechtyWang} among many others).
Most of the recent interest in the study of systems of non-intersecting paths stems from their relation with random matrix theory (RMT) and the Kardar-Parisi-Zhang (KPZ) universality class.
For an overview of this relation in the context of this paper we refer the reader to the introduction of \cite{nibm-loe}; for a more general overview of the other aspects of the KPZ universality class which are relevant to our discussion we mention \cite{quastelRem-review,borodinPetrovReview,quastelSpohn}.

This paper is a continuation of \cite{nibm-loe}, where we studied the distribution of the random variable
\begin{equation}\label{eq:cmN}
\cm_N^{\rm bb}=\max_{t\in[0,1]}B_N(t),
\end{equation}
the maximal height attained by the top path in our collection of non-intersecting Brownian bridges.
The main result of \cite{nibm-loe} is that, for fixed $N$, $(\cm_N^{\rm bb})^2$ is distributed as the largest eigenvalue of a certain random matrix model, known as the Laguerre Orthogonal Ensemble.
Our goal now is twofold: first, to study the location at which the maximum in \eqref{eq:cmN} is attained, and, second, to extend our results to the case of non-intersecting Brownian motions on the half-line.
Before stating our results we will briefly explain the motivation behind the result obtained in \cite{nibm-loe} and discuss the context in which the study of the location of the maximum is natural.

\subsection{Last passage percolation and the Airy process}\label{sec:LPP}

In \emph{(geometric) last passage percolation (LPP)} one considers a family $\big\{w(i,j)\}_{i,j\in\zz^+}$ of independent geometric random variables with parameter
$q$ (i.e. $\pp(w(i,j)=k)=q(1-q)^{k}$ for $k\geq0$) and lets $\Pi_N$ be the collection of up-right paths of length $N$, that is, paths $\pi=(\pi_0,\dotsc,\pi_n)$ such that $\pi_i-\pi_{i-1}\in\{(1,0),(0,1)\}$.
The \emph{point-to-point last passage time} is defined, for $M,N\in\zz^+$, by 
\[L^{\rm point}(M,N)=\max_{\pi\in\Pi_{N+M}:\ts(0,0)\to(M,N)}\sum_{i=0}^{M+N}w(\pi_i),\]
where the maximum is taken over all up-right paths connecting the origin to $(M,N)$.
\citet{johanssonShape} proved that there are explicit constants $c_1$ and $c_2$, depending only on $q$, such that 
\[\pp\big(L^{\rm point}(N,N)\leq c_1N+c_2N^{1/3}r\big)\longrightarrow F_{\rm GUE}(r)\] as $N\to\infty$, with $F_{\rm GUE}$ the Tracy-Widom GUE distribution from random matrix theory, that is, the distribution of the asymptotic fluctuations of the largest eigenvalue of a random matrix drawn from the \emph{Gaussian Unitary Ensemble} \cite{tracyWidom} (an Hermitian matrix with complex Gaussian entries).
The above convergence still holds if one considers $L^{\rm point}(N+k,N-k)$ for any fixed $k$ instead of $L^{\rm point}(N,N)$.
\citet{prahoferSpohn} turned next to the study of the asymptotic fluctuations of the process $k\mapsto L^{\rm point}(N+k,N-k)$.
Consider the process $t\mapsto H_N(t)$ defined by linearly interpolating the values given by scaling $L^{\rm point}(N,M)$ through the relation
\begin{equation}\label{eq:lpp-pt-scaling}
L^{\rm point}(N+k,N-k)=c_1N+c_2N^{1/3}H_N(c_3N^{-2/3}k),
\end{equation}
where $c_3$ is another explicit constant which depends only on $q$.
Then
\begin{equation}\label{eq:johAiry2}
  H_N(t) \longrightarrow \aip(t)-t^2
\end{equation}
in distribution, in the topology of uniform convergence on compact sets, where $\aip$ is the \emph{Airy$_2$ process} \cite{prahoferSpohn,johansson}.
The Airy$_2$ process is a stationary, non-Markovian process, with marginals given by the Tracy-Widom GUE distribution and with finite-dimensional distributions given by an explicit Fredholm determinant formula, and is believed to describe the asymptotic spatial fluctuations for all models in the KPZ universality class with curved initial data.
On the other hand one can define the \emph{point-to-line last passage time} by
\begin{equation}
L^{\rm line}(N)=\max_{k=-N,\dots,N}L^{\rm point}(N+k,N-k).\label{eq:lpplptline}
\end{equation}
From the definition of $H_N$, \cite{johansson} showed based on \eqref{eq:johAiry2} that
\begin{equation}\label{eq:baikrains}
c_2^{-1}N^{-1/3}[L^{\rm line}(N)-c_1N]\longrightarrow\sup_{t\in\rr}\{\aip(t)-t^2\}
\end{equation}
in distribution.
But it was known separately \cite{baikRains} that the distribution of the quantity on the left converges to $F_{\rm GOE}$, the \emph{Tracy-Widom GOE distribution} \cite{tracyWidom2}, which is the analog of $F_{\rm GUE}$ in the case of real symmetric Gaussian random matrices.
From this, \citet{johansson} deduced the remarkable fact that
\begin{equation}\label{eq:johGOE}
  \pp\!\left(\max_{t\in\rr}\,(\aip(t)-t^2)\leq m\right)=F_{\rm GOE}(4^{1/3}m).
\end{equation}

Since it will play an important role in the sequel, let us stop for a moment to define $F_{\rm GOE}$ more precisely.
We say that an $N\times N$ random matrix $A$ is drawn from the \emph{Gaussian Orthogonal Ensemble (GOE)} if $A_{ij}=\mathcal{N}\hspace{-0.1em}(0,1)$ for $i>j$ and $A_{ii}=\mathcal{N}\hspace{-0.1em}(0,2)$, where $\mathcal{N}(a,b)$ denotes a Gaussian random variable with mean $a$ and variance $b$ and all the Gaussian variables are independent (subject to the symmetry condition).
By the Wigner semicircle law \cite{wigner} the largest eigenvalue $\lambda_{\rm GOE}(N)$ of $A$ is expected to lie near $2\sqrt{N}$.
The Tracy-Widom GOE distribution describes the fluctuations of $\lambda_{\rm GOE}(N)$ around $2\sqrt{N}$:
\begin{equation}\label{eq:GOElim}
F_{\rm GOE}(m)=\lim_{N\to\infty}\pp\big(\lambda_{\rm GOE}(N)\leq2\sqrt{N}+N^{-1/6}m\big).
\end{equation}
It is given explicitly by
\begin{equation}\label{eq:GOE}
F_{\rm GOE}(m)=\det({\sf I}-{\sf P}_0{\sf B}_m{\sf P}_0)_{L^2(\rr)},
\end{equation}
where ${\sf P}_m$ denotes the projection onto the interval $(m,\infty)$ (i.e. ${\sf P}_mf(x)=f(x)\uno{x>m}$ for $f\in L^2(\rr)$), ${\sf B}_m$ is the integral operator acting on $L^2(\rr)$ with kernel
\begin{equation}
  {\sf B}_m(x,y)=\Ai(x+y+m),\label{eq:defB0}
\end{equation}
and $\Ai$ denotes the Airy function.
The determinant in \eqref{eq:GOE} means the Fredholm determinant on the Hilbert space $L^2(\rr)$.
For the definition, properties and some background on Fredholm determinants, which can be thought of as the natural generalization of the usual determinant to infinite dimensional Hilbert spaces, we refer the reader to \cite[Section 2]{quastelRem-review}.

A direct proof of \eqref{eq:johGOE} was provided in \cite{cqr}.
The proof is based in first obtaining a Fredholm determinant formula for probabilities of the form $\pp\big(\aip(t)\leq g(t),\,\forall~t\in[\ell,r])$, and then choosing $g(t)=t^2+m$ and computing the limit as $\ell\to-\infty$ and $r\to\infty$.
This method has been applied to obtain several other results about the Airy$_2$ and related processes (see \cite{bcr} and the review \cite{quastelRem-review}), and is the basis of our arguments in \cite{nibm-loe} and in this paper.
Arguably the most important of those applications has been the computation of the distribution of the location at which the maximum in \eqref{eq:johGOE} is obtained.
To understand the interest in this distribution, consider the random variable
\[\ct_N^{\rm lpp}=\argmax_{k=-N,\dotsc,N}L^{\rm point}_N(N+k,N-k)\] 
(the location $k$ which solves the maximization problem need not be unique, so for simplicity we take the argmax to mean the leftmost point at which the maximum is attained).
The random variable $\ct_N^{\rm lpp}$ corresponds to the location of the endpoint of the maximizing path in point-to-line LPP.
\citet{mezardParisi} derived non-rigorously the scaling relation $|\ct_N^{\rm lpp}|\sim N^{2/3}$.
In view of this we define the rescaled endpoint
$\wt\ct_N^{\rm lpp}=c_3N^{-2/3}\ct_N^{\rm lpp}$,
so that
\[\wt\ct_N^{\rm lpp}=\argmax_{|t|\leq c_3^{-1}N^{2/3}}H_N(t).\]
Since $H_N(t)$ converges to $\aip(t)-t^2$, one expects that $\wt\ct_N^{\rm lpp}$ converges in distribution to 
\begin{equation}\label{eq:Tdef}
  \ct:=\argmax_{t\in\rr}\big\{\aip(t)-t^2\},
\end{equation}
provided of course that this last argmax is unique.
Johansson proved in \cite{johansson} that, under the assumption of uniqueness of this argmax, which was proved several years later independently in \cite{corwinHammond} and \cite{mqr} (and slightly later, in much greater generality, in \cite{pimentelStatMax}), one indeed has
\begin{equation}\label{eq:LPPcv}
  \wt{\ct}_N^{\rm lpp}\xrightarrow[N\to\infty]{}\ct
\end{equation}
in distribution.
By KPZ universality, it is expected that $\ct$ should appear through similar considerations for many other models in the KPZ class.
In particular, $\ct$ should describe the asymptotic distribution of the endpoint for a very broad class of point-to-line directed random polymers (of which LPP should be thought of as a zero-temperature limit).
While the computation of the \emph{polymer endpoint distribution} has interested statistical physicists since at least the mid 90's \cite{halpZhang}, its identification with $\ct$ dates back only to \cite{johansson}.
After several (non-rigorous) attempts in the physics literature at computing the distribution of $\ct$ which yielded only partial progress, the answer came in \cite{mqr}, who used the method introduced in \cite{cqr} to derive a formula for the joint density of $\cm$ and $\ct$, with
\begin{equation}\label{eq:Mdef}
  \cm=\max_{t\in\rr}\big\{\aip(t)-t^2\}
\end{equation}
(see \eqref{eq:densairy2} below for the explicit formula for this density).

\subsection{Non-intersecting Brownian bridges and LOE}

As we mentioned above, the Airy$_2$ process is expected to arise in the description of the asymptotic spatial fluctuations of a wide class of models in the KPZ universality class.
While this conjecture, in its full generality, remains one of the central open problems in the field, it is known to hold for a wide class of models, among them non-intersecting Brownian bridges.
More precisely, it holds that the top curve in a system of $N$ non-intersecting Brownian bridges converges to the Airy$_2$ process minus a parabola:
\begin{equation}\label{eq:bbAiry2}
  2N^{1/6}\Big(B_N\big(\tfrac12(1+N^{-1/3}t)\big)-\sqrt{N}\Big)\longrightarrow\aip(t)-t^2
\end{equation}
in the sense of convergence in distribution in the topology of uniform convergence on compact sets.
This result has long been well-known in the sense of convergence of finite-dimensional distributions; the stronger convergence stated here was proved in \cite{corwinHammond}.
In view of \eqref{eq:bbAiry2}, a similar argument as the one leading to \eqref{eq:johGOE} shows that
\begin{equation}\label{eq:bbGOE}
    \lim_{N\to\infty}\pp\Big(2N^{1/6}\big(\cm_N^{\rm bb}-\sqrt{N}\big)\leq m\Big)=F_{\rm GOE}(4^{1/3}m)
\end{equation}
(where, we recall, $\cm_N^{\rm bb}$ was defined in \eqref{eq:cmN}).

The question we wanted to answer in our previous article \cite{nibm-loe} was whether there is a finite $N$ version of this result.
Suprisingly, the answer turned out to be positive, connecting $\cm^{\rm bb}_N$ with another random matrix ensemble.
Let $X$ be an $n\times N$ matrix whose entries are i.i.d. $\mathcal{N}\hspace{-0.1em}(0,1)$, where we assume $n\geq N$.
Then the random $N\times N$ matrix $M=X^{\sf T}\tsm X$ is said to be drawn from the \emph{Laguerre Orthogonal Ensemble} (and is often referred to also as a \emph{Wishart matrix}, as it can be thought of essentially as the sample covariance matrix of $n$ independent samples of an $N$-variate Gaussian population).
By the Mar\v{c}enko-Pastur law \cite{marcenkoPastur} the largest eigenvalue $\lambda_{\rm LOE}(N)$ of $M$ lies at $(4+o(1))N$.
Assuming that $n=N+p$ for some fixed $p$, the fluctuations of $\lambda_{\rm LOE}(N)$ around $4N$ are of order $N^{1/3}$, and the limiting law is again Tracy-Widom GOE:
\begin{equation}\label{eq:LOElim}
  \lim_{N\to\infty}\pp\big(\lambda_{\rm LOE}(N)\leq4N+2^{4/3}N^{1/3}m\big)=F_{\rm GOE}(m).
\end{equation}
In all that follows we will assume that $n=N+1$.
For this choice we let
\begin{equation}\label{eq:FLOE}
F_{{\rm LOE},N}(m)=\pp(\lambda_{\rm LOE}(N)\leq m).
\end{equation}
We introduce also the \emph{Hermite kernel}\footnote{This is just the standard Hermite kernel which appears elsewhere in the literature (and, in particular, in \cite{nibm-loe}); the superscript bb in our notation stands for Brownian bridges, and is included to distinguish the kernel from similar ones which will be introduced below in the case of Brownian bridges on the half-line.}
\begin{equation}\label{eq:defKN}
  \KGN(x,y)=\sum_{n=0}^{N-1}\varphi_n(x)\varphi_n(y),
\end{equation}
where the $\varphi_n$'s are the \emph{harmonic oscillator functions} (which we will refer to as \emph{Hermite functions}), defined as $\varphi_n(x)=e^{-x^2/2}p_n(x)$ with $p_n$ the $n$-th normalized Hermite polynomial.
We introduce also the \emph{reflection operator} $\varrho_m$, given by 
\begin{equation}
\varrho_mf(x)=f(2m-x).\label{eq:refl}
\end{equation}

\begin{thm}[\cite{nibm-loe}]\label{thm:nibm-loe}
For every fixed $N$ we have
\[\pp\!\left(\sqrt{2}\cm_N^{\rm bb}\leq m\right)=\det\!\left(\sI - \KGN\varrho_m\KGN\right)_{L^2(\rr)}=F_{{\rm LOE},N}(2m^2).\]  
In particular $4\cm_N^2$ is distributed as the largest eigenvalue of the LOE matrix $M$ introduced above.
\end{thm}

The proof of the first equality is similar to that of \eqref{eq:johGOE} in \cite{cqr} and will be described in Section \ref{sec:bm_halfspace} in the context of Brownian motions on a half-line.
The second equality was also proved in \cite{nibm-loe} (through an independent argument).
Theorem \ref{thm:nibm-loe} can be recast in terms of the probability that (GUE) Dyson Brownian motion hits an hyperbolic cosine barrier (see \cite[Prop. 1.4]{nibm-loe}), but we will not adopt that perspective in this paper.

\subsection{Location of the maximum}

Our first result provides a formula for the distribution of 
\begin{equation}\label{eq:defTNbb}
  \ct_N^{\rm bb}:=\argmax_{t\in[0,1]}B_N(t),
\end{equation}
the location at which the maximum height of the top line in the system of $N$ non-intersecting Brownian bridges is attained (note that, since the top path is obviously absolutely continuous with respect to a Brownian bridge, the argmax in this case is easily seen to be almost surely unique).
Analogously to the result of \cite{mqr}, we will provide in fact an explicit formula for the joint density of the max and the argmax.

For $m\geq0$ and $t\in(0,1)$ let
\begin{equation}\label{eq:at}
g(t)=\frac{1}{\sqrt{2t(1-t)}}
\end{equation}
and define the function
\begin{equation}\label{eq:psibb}
\psi^{\rm bb}_{m,t}(n)=\sqrt{2}\tts g(t)^{3/2}\,\Bigl(\frac{t}{1-t}\Bigr)^{-\frac{n}{2}}\bigl[\varphi_n'(m\tts g(t))+(2t-1)m\tts g(t)\varphi_n(m\tts  g(t))\bigr]
\end{equation}
and the rank one kernel
\begin{equation}\label{eq:Psibb}
\Psi^{\rm bb}_{N,m,t}(x,y)=\left(\sum_{n=0}^{N-1}\varphi_n(x)\psi^{\rm bb}_{m,t}(n)\right)\!\left(\sum_{n=0}^{N-1}\varphi_n(y)\psi^{\rm bb}_{m,1-t}(n)\right).
\end{equation}
We note also that, by the second equality in Theorem \ref{thm:nibm-loe} and the fact that $F_{\rm LOE}(m)>0$ for all $m>0$, ${\sf I}-\KGN\varrho_m\KGN$ is invertible for all such $m$.

\begin{thm}\label{thm:densBB}
Let $f_N^{\rm bb}(m,t)$ denote the joint density of $\cm_N^{\rm bb}$ and $\ct_N^{\rm bb}$.
Then for all $m>0$ and all $t\in(0,1)$,
\begin{equation}
\begin{split}
\label{eq:densBB}
f_N^{\rm bb}(m,t)&=\tr\!\left[({\sf I}-\KGN\varrho_{\sqrt{2}m}\KGN)^{-1}\Psi^{\rm bb}_{N,m,t}\right]F_{{\rm LOE},N}(4m^2)\\
&=\det\!\left({\sf I}-\KGN \varrho_{\sqrt{2}m} \KGN+\Psi^{\rm bb}_{N,m,t}\right)-F_{{\rm LOE},N}(4m^2).
\end{split}
\end{equation}
\end{thm}

\begin{figure}[!ht]
	\centering
		\includegraphics[width=0.7\textwidth]{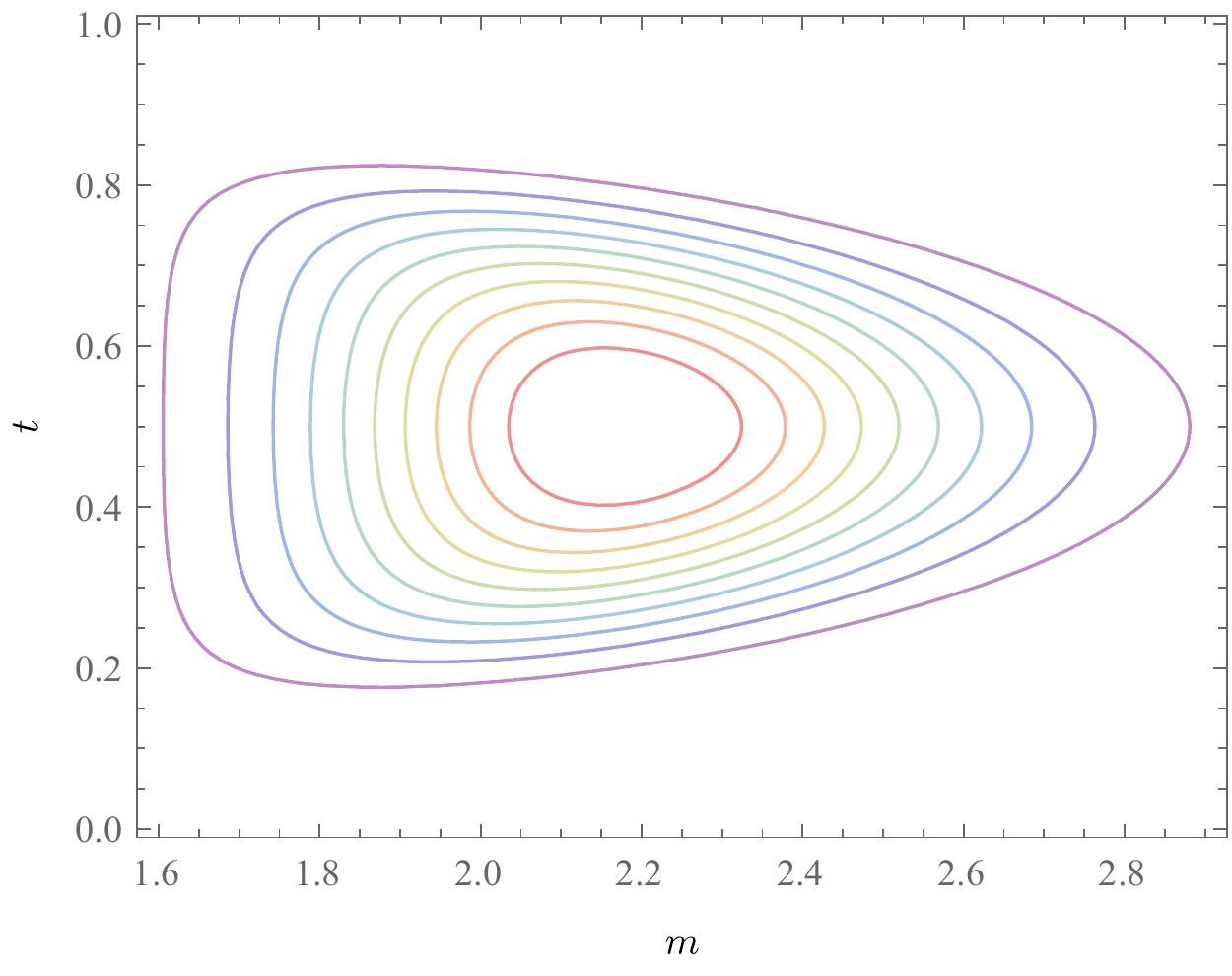}
	\caption{Contour plot of the joint density of $\cm_N^{\rm bb}$ and $\ct_N^{\rm bb}$ with $N=6$.}
	\label{fig:jdensity}
\end{figure}

The second equality in \eqref{eq:densBB} follows directly from the fact that $\Psi^{\rm bb}_{N,m,t}$ is rank one.
Note, moreover, that all the operators appearing above are finite-rank, and thus the formulas can be easily expressed in terms of the determinant and trace of finite matrices (see e.g. \cite[eqn. (3.6)]{nibm-loe}).
In particular, the numerical computation of $f^{\rm bb}_N$ is completely straightforward (see Figure \ref{fig:jdensity} for a contour plot).
We remark that this formula is entirely analogous to the one derived in \cite{mqr} for the joint density of $\cm$ and $\ct$ (see \eqref{eq:densairy2} below).
Furthermore, we obtain as a consequence a direct proof of the convergence of the rescaled argmax of $B_N(t)$ to that of $\aip(t)-t^2$:

\begin{cor}\label{cor:bbCvgce}
Let
\[\wt\cm_N^{\rm bb}=2N^{1/6}(\cm_N^{\rm bb}-\sqrt{N})\qqand\wt\ct_N^{\rm bb}=2N^{1/3}\big(\ct_N^{\rm bb}-\tfrac12\big).\]
Then we have the convergence in distribution
\[(\wt\cm_N^{\rm bb},\wt\ct_N^{\rm bb})\xrightarrow[N\to\infty]{}(\cm,\ct).\]
\end{cor}

This result can also be derived from \eqref{eq:bbAiry2}, similarly to how the analogous result for last passage percolation (which is \eqref{eq:baikrains} together with \eqref{eq:johGOE}) was derived (under the assumption of uniqueness) in \cite{johansson}; here \eqref{eq:bbAiry2} shows that if one restricts the maximizer to lie in an interval of the form $[\frac12(1-N^{-1/3}K),\frac12(1+N^{-1/3}K)]$ for any fixed $K>0$ then the scaled restricted maximizer converges in distribution to $\argmax_{t\in[-K,K]}\{\aip(t)-t^2\}$, and from this the result can be obtained by using arguments in \cite{corwinHammond} to show that the probability that the maximizer lies outside such an interval can be made arbitrarily small as $N\to\infty$ provided that one chooses a large enough $K$.

Recall that, as we mentioned in Section \ref{sec:LPP}, the random variable $\ct$, which is the limit of $\wt\ct^{\rm lpp}_N$, is expected to appear similarly in a wide class of models in the KPZ class.
As far as we are aware, Corollary \ref{cor:bbCvgce}, together with Corollary \ref{cor:berbbCvgce} below, constitute the first rigorous proofs in this direction after the LPP case \eqref{eq:LPPcv}. 

We turn now to studying the rate at which of $\ct_N^{\rm bb}$ concentrates around its expected location $\frac12$.
For the case of the polymer endpoint distribution (which is centered around the origin), it was conjectured in \cite{halpZhang} that $\pp(|\ct|>t)\sim e^{-ct^3}$ for some $c>0$.
The upper bound was first proved in \cite{corwinHammond} with an unknown $c$.
Later on, \cite{quastelRemTails} gave a different proof of the upper bound with $c=\frac43$ together with a lower bound with a different constant.
The same paper conjectured that the right constant is in fact $\frac43$, which was then proved through a combination of the arguments of \cite{schehr} and \cite{baikLiechtySchehr}.

In our case note that, by Corollary \ref{cor:bbCvgce}, as $N$ gets large the location of the argmax $\ct_N^{\rm bb}$ concentrates around $\frac12$ at a scale of $N^{-1/3}$.
In view of the tail behavior of $\ct$ and the scaling in Corollary \ref{cor:bbCvgce}, one expects then that, optimally, the probability $\pp(|\ct_N^{\rm bb}-\frac12|>\ep)$ should decay like $e^{-\frac{32}3N\ep^3}$ for small $\ep$.
This can be thought of as the \emph{small deviation} regime for the concentration of $\ct_N^{\rm bb}$, and is the content of our next result, where we get the predicted upper bound as well as a lower bound with a different constant.

\begin{thm}\label{thm:TNtails}
  Let $\ep_1=\frac12\frac{e^{2/3}-1}{e^{2/3}+1}\approx0.16$ and $\ep_2>\frac12\frac{e^{2}-1}{e^{2}+1}\approx0.38$.
  There are $c_1,c_2,c_3,n_0>0$ such that
  \begin{equation}
    c_1\ts e^{-c_2N\ep^3}\leq\pp\big(|\ct_N^{\rm bb}-\tfrac12|>\ep\big)\leq c_3\ts e^{-\frac{32}3N\ep^3+\mathcal{O}(N^{2/3})},\label{eq:MNtails}
  \end{equation}
  with the upper bound holding uniformly in $N\in\nn$ and $\ep\in(0,\ep_2)$ satisfying $N\ep^3\geq n_0$ and the lower bound holding uniformly in $N\in\nn$ and $\ep\in(0,\ep_1)$ satisfying $N\ep^3\geq n_0$.
\end{thm}

The proof is based on the arguments of \cite{quastelRemTails} together with the small deviation estimates for $F_{{\rm LOE},N}(m)$ established in \cite{ledouxRider} (these estimates can be used to obtain similar tail bounds for $\cm^{\rm bb}_N$, see also \eqref{eq:iniUpperBound} and \eqref{eq:smalldevLOE} below).

\subsection{Non-intersecting Brownian motions on the half-line}

We consider now systems of non-intersecting Brownian motions which are restricted to stay in the positive half-line.
There are two standard ways to enforce this condition.
The first one is to put an absorbing boundary at the origin, which corresponds to conditioning the Brownian bridges to stay positive and leads to the process known as a \emph{Brownian excursion}.
In this case we will denote the $N$ paths by $B^{\rm be}_1(t)<\dotsm<B^{\rm be}_N(t)$ (that is, we consider $N$ independent Brownian excursions starting from and ending at the origin and condition them, in the sense of Doob, to not intersect).
The second possibility is to put a reflecting wall at the origin, which corresponds to considering \emph{reflected Brownian bridges}.
In this second case we will use the notation $B^{\rm rbb}_1(t)<\dotsm<B^{\rm rbb}_N(t)$.

In \cite{tracyWidomNIBrExc} Fredholm determinant formulas for the finite-dimensional distribution of both systems were derived.
The resulting formulas are analogous to those for non-intersecting Brownian bridges, and using the general result of \cite{bcr} this will allow us to derive formulas for the hitting probabilities of the top path of these systems.
Based on these we will derive an explicit Fredholm determinant formula for the maximal height of these systems.

As we will see, the resulting formulas have the same structure as the Fredholm determinant formula given for the Brownian bridge case in Theorem \ref{thm:nibm-loe}.
In fact, all that changes is that the Hermite kernel $\KGN$ gets replaced by
\begin{equation}\label{eq:defoddKN}
  \oKGN(x,y)=\sum_{n=0}^{N-1}\varphi_{2n+1}(x)\varphi_{2n+1}(y)
\end{equation}
in the absorbing case and by
\begin{equation}\label{eq:defevenKN}
  \eKGN(x,y)=\sum_{n=0}^{N-1}\varphi_{2n}(x)\varphi_{2n}(y)
\end{equation}
in the reflecting case, while the reflection operator $\varrho_m$ gets replaced by more complicated operators composed of a infinite sum of reflections,
\begin{equation}\label{eq:def-infinitevarrho}
  \varrho^{\rm be}_mf(x)=2\sum_{k=1}^{\infty}f(2km-x)\qquad\text{and}\qquad\varrho^{\rm rbb}_mf(x)=2\sum_{k=1}^{\infty}(-1)^{k+1} f(2km-x).
\end{equation}
We note that the Hermite functions with odd and even indices appearing in \eqref{eq:defoddKN} and \eqref{eq:defevenKN} are, respectively, odd and even.
This will be important in the proof of our formulas.

\begin{thm}\label{thm:be-formula}
  Let 
  \begin{equation}\label{eq:cmhsN}
	\cm_N^{\rm be}=\max_{t\in[0,1]}B^{\rm be}_N(t)\qqand\cm_N^{\rm rbb}=\max_{t\in[0,1]}B^{\rm rbb}_N(t).
  \end{equation}
  Then for any $m\geq 0$, and with $\star$ standing for either {\upshape be} or {\upshape rbb}, we have
  \begin{equation}\label{eq:be-formula}
    \pp\!\left(\sqrt{2}\cm^\star_N\leq m\right) = \det\!\left({\sf I}-\sKGN \varrho^{\star}_m \sKGN\right)_{L^2(\rr)}.
  \end{equation}
\end{thm}

It is natural to wonder whether these two probability distributions have an interpretation in terms of RMT, as in the case of $\cm^{\rm bb}_N$, but we are not aware of any.

We turn now to the distribution of the argmax for the top path of non-intersecting Brownian excursions and reflected Brownian bridges.
To that end we introduce, for $m\geq0$ and $t\in(0,1)$, the functions
\begin{multline}\label{eq:psibe}
\psi^{\rm be}_{m,t}(n)=2\psi^{\rm bb}_{m,t}(n)+2\sqrt{2}\tts g(t)^{3/2}\,\Bigl(\frac{t}{1-t}\Bigr)^{-\frac{n}{2}}\sum_{k=1}^{\infty}e^{2k(k+1)(2t-1)m^2\tts g(t)^2}\\
\times\bigl[\varphi_n'\bigl((2k+1)m\tts g(t)\bigr)+(2k+1)(2t-1)m\tts g(t)\varphi_n\bigl((2k+1)m\tts g(t)\bigr)\bigr],
\end{multline}
and
\begin{multline}\label{eq:psirbb}
\psi^{\rm rbb}_{m,t}(n)=2\psi^{\rm bb}_{m,t}(n)+2\sqrt{2}\tts g(t)^{3/2}\,\Bigl(\frac{t}{1-t}\Bigr)^{-\frac{n}{2}}\sum_{k=1}^{\infty}(-1)^ke^{2k(k+1)(2t-1)m^2\tts g(t)^2}\\
\times\bigl[\varphi_n'\bigl((2k+1)m\tts g(t)\bigr)+(2k+1)(2t-1)m\tts g(t)\varphi_n\bigl((2k+1)m\tts g(t)\bigr)\bigr],
\end{multline}
where $\psi^{\rm bb}_{m,t}$ and $g(t)$ were defined in \eqref{eq:at} and \eqref{eq:psibb}, and the rank one kernels
\begin{equation}\label{eq:Psiberbb}
\begin{split}
	\Psi^{\rm be}_{m,t}(x,y)&=\Bigl(\sum_{n=0}^{N-1}\varphi_{2n+1}(x)\psi^{\rm be}_{m,t}(2n+1)\Bigr)\Bigl(\sum_{n=0}^{N-1}\varphi_{2n+1}(x)\psi^{\rm be}_{m,1-t}(2n+1)\Bigr),\\
	\Psi^{\rm rbb}_{m,t}(x,y)&=\Bigl(\sum_{n=0}^{N-1}\varphi_{2n}(x)\psi^{\rm rbb}_{m,t}(2n)\Bigr)\Bigl(\sum_{n=0}^{N-1}\varphi_{2n}(x)\psi^{\rm rbb}_{m,1-t}(2n)\Bigr).
\end{split}
\end{equation}

\begin{thm}\label{thm:half-line-argmax}
  Let 
  \begin{equation}\label{eq:defTNberbb}
    \ct_N^{\rm be}:=\argmax_{t\in[0,1]}B^{\rm be}_N(t) \qqand \ct_N^{\rm rbb}:=\argmax_{t\in[0,1]}B^{\rm rbb}_N(t)
  \end{equation}
	and let $f^{\rm be}_N(m,t)$ and $f^{\rm rbb}_N(m,t)$ be the joint densities of $(\cm^{\rm be}_N,\ct^{\rm be}_N)$ and $(\cm^{\rm rbb}_N,\ct^{\rm rbb}_N)$, respectively.
	Then for any $t\in(0,1)$ and any $m>0$, and with $\star$ standing for either $\rm be$ or $\rm rbb$, we have
	\begin{align}\label{eq:densBERBB}
	f^\star_N(m,t)&=\tr\!\left[({\sf I}-\sKGN\varrho^{\star}_{\sqrt{2}m}\sKGN)^{-1}\Psi^{\star}_{m,t}\right]\det\!\left({\sf I}-\sKGN \varrho^{\star}_{\sqrt{2}m} \sKGN\right)\\
	&=\det\!\left({\sf I}-\sKGN \varrho^{\star}_{\sqrt{2}m} \sKGN+\Psi^{\star}_{m,t}\right)-\det\!\left({\sf I}-\sKGN \varrho^{\star}_{\sqrt{2}m} \sKGN\right).
	\end{align}
\end{thm}

Note that, in principle, Theorem \ref{thm:be-formula} can be obtained as a corollary of this result by integrating the joint densities with respect to $t$ (see \cite[Sec. 3]{mqr}, where this is done in the case of the Airy$_2$ process).
 
By KPZ universality, it is expected that both $B^{\rm be}_N$ and $B^{\rm rbb}_N$, suitably rescaled, converge to the Airy$_2$ process (in fact, in the sense of finite-dimensional distributions, this can be proved just as for the case of Brownian bridges, based on the formulas appearing in \cite{tracyWidomNIBrExc}, see \eqref{eq:beextkernel}--\eqref{eq:defKbeext} below).
Hence one should expect the analog of Corollary \ref{cor:bbCvgce} to hold.
This is the content of our last result, which we prove based on our formulas for $f_N^{\rm be}$ and $f_N^{\rm rbb}$.

\begin{cor}\label{cor:berbbCvgce}
  Let
  \begin{gather}
	\wt\cm_N^{\rm be}=2^{7/6}N^{1/6}(\cm_N^{\rm be}-\sqrt{2N}),\qquad\wt\ct_N^{\rm be}=2^{4/3}N^{1/3}\big(\ct_N^{\rm be}-\tfrac12\big),\\
    \wt\cm_N^{\rm rbb}=2^{7/6}N^{1/6}(\cm_N^{\rm rbb}-\sqrt{2N}),\qquad\wt\ct_N^{\rm rbb}=2^{4/3}N^{1/3}\big(\ct_N^{\rm rbb}-\tfrac12\big).
  \end{gather}
Then we have, in distribution,
\[(\wt\cm_N^{\rm be},\wt\ct_N^{\rm be})\xrightarrow[N\to\infty]{}(\cm,\ct)\qqand(\wt\cm_N^{\rm rbb},\wt\ct_N^{\rm rbb})\xrightarrow[N\to\infty]{}(\cm,\ct).\]	
\end{cor}

We remark that the convergence of $\wt\cm_N^{\rm be}$ and of $\wt\cm_N^{\rm rbb}$ to $\cm$ has in fact already been proved by \citet{liechty-nibmdope}, who used discrete orthogonal polynomials and Riemann-Hilbert techniques.

In view of the corollary, and analogously to Theorem \ref{thm:TNtails}, we conjecture that the tails of $\ct_N^{\rm be}$ and $\ct_N^{\rm rbb}$ should satisfy
\begin{equation}\label{eq:tails-conj}
  \ct_N^{\rm be}\sim c\tts e^{-\frac{64}3N\ep^3},\qquad\ct_N^{\rm rbb}\sim c\tts e^{-\frac{64}3N\ep^3}.
\end{equation}
The proof of Theorem \ref{thm:TNtails} should in principle also be applicable to these cases.
However, the estimates needed to get the analogous result become much more involved due to the more complicated expressions for $f^{\rm be}$ and $f^{\rm rbb}$.
In addition, in order to obtain \eqref{eq:tails-conj} from these arguments one would need a replacement for the tail estimate obtained in \cite{ledouxRider} for the small deviations of the Laguerre Orthogonal Ensemble (see \eqref{eq:riderledoux} below), which was obtained using random matrix arguments which most likely would not apply to this case.
For these reasons, we opted not to pursue this any further in this paper.

\subsection*{Outline} 
The rest of the paper is devoted to proofs.
Section \ref{sec:bm_halfspace} contains some preliminaries and the continuum statistics formulas which we will use, as well as the proof of Theorem \ref{thm:be-formula}.
Section \ref{sec:derivation} is devoted to the derivation of the joint densities for the three models (Theorems \ref{thm:densBB} and \ref{thm:half-line-argmax}).
The proof of the tail estimate, Theorem \ref{thm:TNtails}, is contained in Section \ref{sec:bounds}.
Appendix \ref{sec:estimates} is devoted to the proof a precise small deviation estimate for the largest eigenvalue in an $N\times N$ GUE matrix.

\section{Path-integral kernels and continuum statistics formulas for non-intersecting Brownian paths}
\label{sec:bm_halfspace}

The basic tool we will use in the proof of all of our results is a ``continuum statistics'' formula for the probability that the top line of a system of non-intersecting Brownian paths (in each of the three cases which we consider) stays below a given function on an interval $[a,b]$ with $0<a<b<1$.
Such a formula was derived in \cite{bcr} for non-intersecting Brownian bridges, and was the basis of our arguments in \cite{nibm-loe}.
In this section we will recall this formula, and derive the analogous formula for the case of non-intersecting Brownian excursions and non-intersecting reflected Brownian motions.

In everything that follows we will use the abbreviations BB, BE and RBB in the text to refer to the models of non-intersecting Brownian bridges, Brownian excursions and reflected Brownian bridges.
Similarly, to each notation we will use a superscript $\star$ in objects like $\cm^\star_N$ when we write formulas which are valid for either of the three models.
So, for instance, $\cm^\star_N$ refers to $\cm^{\rm bb}_N$, $\cm^{\rm be}_N$, or $\cm^{\rm rbb}_N$, and correspondingly $B^\star_N(t)$ refers to $B_N(t)$, $B^{\rm be}_N(t)$, or $B^{\rm rbb}_N(t)$.
We will sometimes also use the superscript be/rbb when dealing with formulas which are relevant only in those two cases.

The finite-dimensional distributions of a system of $N$ non-intersecting BB/BE/RBB can be written \cite{tracyWidomNIBrExc} in terms of a Fredholm determinant as follows:
\begin{equation}\label{eq:beextkernel}
\pp\Bigl(\sqrt{2}B^{\star}_N\bigl(\tfrac{e^{2t_j}}{1+e^{2t_j}}\bigr)\leq r_j\sech(t_j),\,j=1,\dotsc,n\Bigr)=\det\!\left({\sf I}-{\rm f}{\sf K}^{\star}_{{\rm ext},N}{\rm f}\right)_{L^2(\{t_1,\ldots,t_n\}\times\rr)}
\end{equation}
with ${\rm f}(t_j,x)=\uno{x\in(r_j,\infty)}$ and where the extended kernels ${\sf K}^{\star}_{{\rm ext},N}$ are defined as \noeqref{eq:defKrbbext}
\begin{align}\label{eq:defKbbext}
  {\sf K}^{\rm bb}_{{\rm ext},N}(s,x;t,y)&=\begin{cases}
    \sum_{n=0}^{N-1}e^{n(s-t)}\varphi_{n}(x)\varphi_{n}(y) &\mbox{if } s\geq t,\\
    -\sum_{n=N}^{\infty}e^{n(s-t)}\varphi_{n}(x)\varphi_{n}(y) &\mbox{if } s<t,
  \end{cases}\\
  \label{eq:defKbeext}
  {\sf K}^{\rm be}_{{\rm ext},N}(s,x;t,y)&=\begin{cases}
    2\sum_{n=0}^{N-1}e^{(2n+1)(s-t)}\varphi_{2n+1}(x)\varphi_{2n+1}(y) &\mbox{if } s\geq t,\\
    -2\sum_{n=N}^{\infty}e^{(2n+1)(s-t)}\varphi_{2n+1}(x)\varphi_{2n+1}(y) &\mbox{if } s<t,
  \end{cases}\\
  \label{eq:defKrbbext}
  {\sf K}^{\rm rbb}_{{\rm ext},N}(s,x;t,y)&=\begin{cases}
    2\sum_{n=0}^{N-1}e^{2n(s-t)}\varphi_{2n}(x)\varphi_{2n}(y) &\mbox{if } s\geq t,\\
    -2\sum_{n=N}^{\infty}e^{2n(s-t)}\varphi_{2n}(x)\varphi_{2n}(y) &\mbox{if } s<t,
  \end{cases}
\end{align}
and where, we recall, the Hermite functions $\varphi_n$ were defined after \eqref{eq:defKN}.
We note that the value of the determinants in \eqref{eq:beextkernel} in the BE and RBB cases do not change if we replace the corresponding kernels ${\sf K}^{{\rm be}/{\rm rbb}}_{{\rm ext},N}$ by $\frac12{\sf K}^{{\rm be}/{\rm rbb}}_{{\rm ext},N}$ and the projection ${\rm f}$ by $\bar{\rm f}(t_j,x)=\uno{x\in(-\infty,-r_j)\cup(r_j,\infty)}$.
This can be seen at the level of the series expansion of the Fredholm determinant, by using the fact that $\varphi_{2n}$ is even and $\varphi_{2n+1}$ is odd to show that the value of $\det\!\left[{\sf K}^{{\rm be}/{\rm rbb}}_{{\rm ext},N}(t_i,x_i;t_j,x_j)\right]_{i,j=1}^k$ does not change if some of the $x_i$'s are replaced by $-x_i$.
This fact will be important below.

In order to obtain the continuum statistics formulas which we are interested in we need to let $t_1,\dotsc,t_n$ be a fine mesh of the our interval $[a,b]$ and then take $n\to\infty$.
However, note that the Fredholm determinants in \eqref{eq:beextkernel} are being computed in the Hilbert space $L^2(\{t_1,\ldots,t_n\}\times\rr)$, which makes it very hard to make sense of the $n\to\infty$ limit.
To get around this, the idea is to use \cite[Thm. 3.3]{bcr} to turn this Fredholm determinant into the Fredholm determinant of a certain ``path-integral'' kernel computed on $L^2(\rr)$.
As we mentioned, this was done in the BB case in \cite{bcr}.
We describe the result next.

\subsection{Non-intersecting Brownian bridges}

Let ${\sf D}$ denote the differential operator
\begin{equation}\label{eq:defD}
  {\sf D}=-\tfrac{1}{2}(\Delta-x^2+1)
\end{equation}
($\Delta$ is the Laplacian on $\rr$). 
Using the recursion satisfied by the Hermite polynomials one can check that ${\sf D} \varphi_n=n \varphi_n$.
In particular, the Hermite kernel $\KGN$ defined in \eqref{eq:defKN} is nothing but the projection operator onto the space span$\{\varphi_0,\dotsc,\varphi_{N-1}\}$ associated to the first $N$ eigenvalues of ${\sf D}$.
In particular, even though $e^{t{\sf D}}$ is well-defined in general only for $t\leq 0$, $e^{t{\sf D}}\KGN$ is well defined for all $t$, and its integral kernel is given by
\begin{equation}\label{eq:etDKGN}
e^{t{\sf D}}\KGN(x,y)=\sum_{n=0}^{N-1}e^{tn}\varphi_n(x) \varphi_n(y).
\end{equation}
Furthermore, the extended kernel ${\sf K}^{\rm bb}_{{\rm ext},N}$ defined in \eqref{eq:defKbbext} satisfies, for each $s,t$,
\begin{equation}\label{eq:extBB}
  {\sf K}^{\rm bb}_{{\rm ext},N}(s,\cdot;t,\cdot)=-e^{(s-t){\sf D}}\uno{s<t}+e^{(s-t){\sf D}}\KGN.
\end{equation}
This means that the extended kernel has the structure of the kernels considered in \cite[Sec. 3]{bcr}.
One can check, moreover, that the hypotheses of \cite[Thm. 3.3]{bcr} are satisfied, and ultimately lead to the continuum statistics formula for the top line of BB which follows.
For fixed $\ell_1<\ell_2$, consider a function $g\in H^1([\ell_1,\ell_2])$ (i.e. both $g$ and its derivative are in $L^2([\ell_1,\ell_2])$) and introduce an operator $\Theta_{[\ell_1,\ell_2]}^{g,{\rm bb}}$ acting on $L^2(\rr)$ as follows: $\Theta_{[\ell_1,\ell_2]}^{g,{\rm bb}}f(x)=u(\ell_2,x)$, where $u(\ell_2,\cdot)$ is the solution at time $\ell_2$ of the boundary value problem
\begin{equation}\label{eq:bound-pde}
\begin{aligned}
\partial_t u+{\sf D}u&=0\quad\text{for}\ x<g(t),\ t\in(\ell_1,\ell_2)\\
u(\ell_1,x)&=f(x)\mathbf{1}_{x<g(\ell_1)}\\
u(t,x)&=0\quad \text{for}\ x\geq g(t),\ t\in[\ell_1,\ell_2].
\end{aligned}
\end{equation}

\begin{prop}[{\cite[Cor. 4.5]{bcr}}]\label{prop:dbmcont}
For any $\ell_1<\ell_2$ and $g\in H^1([\ell_1,\ell_2])$ we have
\begin{multline}\label{eq:dbmcont}
\pp\left(\sqrt{2}B_N\big(\tfrac{e^{2s}}{1+e^{2s}}\big)<g(s)\sech(s)~\forall\,s\in[\ell_1,\ell_2]\right)\\
=\det\!\left({\sf I}-\KGN+\Theta_{[\ell_1,\ell_2]}^{g,{\rm bb}} e^{(\ell_2-\ell_1){\sf D}}\KGN\right).
\end{multline}
\end{prop}

This formula was derived in \cite[Sec. 4.1]{bcr} for the top line $\lambda_N(t)$ of the stationary GUE Dyson Brownian motion.
It reads
\begin{equation}\label{eq:dbmcontdbm}
\pp\left(\lambda_N(s)<g(s)~\forall\,s\in[\ell_1,\ell_2]\right)
=\det\!\left({\sf I}-\KGN+\Theta_{[\ell_1,\ell_2]}^{g,{\rm bb}} e^{(\ell_2-\ell_1){\sf D}}\KGN\right).
\end{equation}
Since $\lambda_N$ satisfies 
\begin{equation}
\tfrac1{\sqrt{2}}\lambda_N(s)\sech(s)\stackrel{\rm (d)}{=}B_N\tsm\big(\tfrac{e^{2s}}{1+e^{2s}}\big),\label{eq:dbm-bb}
\end{equation}
this formula leads directly to \eqref{eq:dbmcont}.
See \cite[Sec. 2]{nibm-loe} for more details.

It is shown in \cite[Prop. 2.2]{nibm-loe} that the integral kernel of $\Theta_{[\ell_1,\ell_2]}^{g,{\rm bb}}$ can be expressed explicitly in terms of certain hitting probabilities for a Brownian bridge (which is also consistent with our use of the superscript bb).
Remarkably (see also \eqref{eq:bekernel} and \eqref{eq:defLambda} below in the case of BE/RBB), the case we are interested in, which is $g(t)=r\ttsm\cosh(t)$, leads to hitting probabilities of a Brownian bridges to a straight line, which can be computed explicitly by the reflection principle, and lead to the following explicit formula for $\Theta_{[\ell_1,\ell_2]}^{g,{\rm bb}}$ (see \cite[(2.6)]{nibm-loe}):
\begin{equation}\label{eq:thetarRpre}
\Theta^{(r),{\rm bb}}_{[\ell_1,\ell_2]}:=\Theta_{[\ell_1,\ell_2]}^{g(t)=r\ttsm\cosh(t),{\rm bb}}=\bar {\sf P}_{r\ttsm\cosh(\ell_1)}\Big[e^{-(\ell_2-\ell_1){\sf D}}-{\sf R}^{(r),{\rm bb}}_{[\ell_1,\ell_2]}\Big]\bar {\sf P}_{r\ttsm\cosh(\ell_2)},
\end{equation}
with 
\[\bar{\sf P}_mf(x)=({\sf I}-{\sf P}_m)f(x)=f(x)\uno{x\leq m}\]
and where the reflection operator ${\sf R}^{(r),{\rm bb}}_{[\ell_1,\ell_2]}$ is given by
\begin{equation}\label{eq:reflOperBB}
  {\sf R}^{(r),{\rm bb}}_{[\ell_1,\ell_2]}(x,y)=e^{\frac12(y^2-x^2)+\ell_2}\,\tfrac{1}{\sqrt{4\pi(\beta-\alpha)}}e^{-r(e^{\ell_2}y-e^{\ell_1}x)+r^2(\beta-\alpha)-(e^{\ell_1}x+e^{\ell_2}y-2r(\alpha+\beta)-r)^2/(4(\beta-\alpha))},
\end{equation}
with 
\begin{equation}
\alpha=\tfrac14e^{2\ell_1}\qqand\beta=\tfrac14e^{2\ell_2}.\label{eq:alphabeta}
\end{equation}
It is worth mentioning that \eqref{eq:dbmcont}--\eqref{eq:alphabeta} has recently been used in \cite{ferrariVetoHardEdge} to study non-intersecting Brownian bridges conditioned to stay below a given threshold (very shortly after \cite{liechtyWang} provided an alternative treatment based on a Riemann-Hilbert analysis).

\subsection{Non-intersecting Brownian excursions and reflected Brownian bridges}
We turn now to the case of BE/RBB.
To proceed as in the case of Brownian bridges, we need to express the extended kernels ${\sf K}^{{\rm be}/{\rm bb}}_{{\rm ext},N}$ similarly to \eqref{eq:extBB}.
A crucial fact which is implicit in \eqref{eq:extBB} is that, as $N\to\infty$, $\KGN$ becomes the identity (this is because $\big(\varphi_n)_{n\geq0}$ is a complete orthonormal basis of $L^2(\rr)$).
However, this is not the case for $\oeKGN$ (defined in \eqref{eq:defoddKN}/\eqref{eq:defevenKN}), because due to the parity property of the Hermite functions mentioned above, $\oKGN$ converges to the projection onto the subspace $L^2_{\rm odd}(\rr)$ of $L^2(\rr)$ consisting of odd functions, and similarly $\eKGN$ converges to the projection onto the subspace $L^2_{\rm even}(\rr)$ of $L^2(\rr)$ consisting of even functions.
The solution is to replace the Hilbert space $L^2(\{t_1,\ldots,t_n\}\times\rr)$, which is isomorphic to $\bigoplus_{i=1}^nL^2(\rr)$, with $\bigoplus_{i=1}^nL^2_{{\rm odd}/{\rm even}}(\rr)$.
Note that, after performing the replacement explained after \eqref{eq:defKbeext}, this does not change the value of the determinant because $\bar{\rm f}$ maps odd/even functions to odd/even functions and ${\sf K}^{{\rm be}/{\rm rbb}}_{{\rm ext},N}$ maps $\bigoplus_{i=1}^nL^2(\rr)$ to $\bigoplus_{i=1}^nL^2_{{\rm odd}/{\rm even}}(\rr)$.

We end up then with $\det\!\left({\sf I}-\frac12\bar{\rm f}{\sf K}^{{\rm be}/{\rm rbb}}_{{\rm ext},N}\bar{\rm f}\right)_{\bigoplus_{i=1}^nL^2_{{\rm odd}/{\rm even}}(\rr)}$ replacing the right hand side of \eqref{eq:beextkernel}.
From the fact that ${\sf D} \varphi_n=n \varphi_n$ one can check directly that for each $s,t$ we have
\begin{equation}\label{eq:extBE}
  \tfrac12{\sf K}^{{\rm be}/{\rm rbb}}_{{\rm ext},N}(s,\cdot;t,\cdot)=-e^{(s-t){\sf D}}\uno{s<t}+e^{(s-t){\sf D}}\oeKGN
\end{equation}
(with $\oeKGN$ defined in \eqref{eq:defoddKN}/\eqref{eq:defevenKN}) as an operator acting on $L^2_{{\rm odd}/{\rm even}}(\rr)$, and moreover that $\oeKGN$ satisfy
\begin{equation}\label{eq:etDoKGN}
e^{t{\sf D}}\oKGN(x,y)=\sum_{n=0}^{N-1}e^{(2n+1)t}\varphi_{2n+1}(x) \varphi_{2n+1}(y)\qand e^{t{\sf D}}\eKGN(x,y)=\sum_{n=0}^{N-1}e^{2nt}\varphi_{2n}(x) \varphi_{2n}(y)
\end{equation}
for all $t\in\rr$ (note, in particular, that $\tfrac12{\sf K}^{{\rm be}/{\rm rbb}}_{{\rm ext},N}(t,\cdot;t,\cdot)=\oeKGN$ for all $t$).
An additional difficulty in applying \cite[Thm. 3.3]{bcr} in these cases is that in that result it is assumed that the Fredholm determinant acts on a space of the form $L^2(\{t_1,\ldots,t_n\}\times X)$ with $(X,\Sigma,\mu)$ some measure space\footnote{The case of reflected Brownian bridges is slightly simpler, because the space $L^2_{\rm even}(\rr)$ can be identified with $L^2([0,\infty))$, and thus we may regard our extended kernel in that case as acting on $L^2(\{t_1,\ldots,t_n\}\times[0,\infty))$.}.
But that hypothesis was made in \cite{bcr} only for convenience, in order to handle the general setting addressed there, and it is not hard to check that the result carries through to our case without difficulties.
One can check easily, once again, that the hypotheses of that result are satisfied in our present case, which allows us to deduce that the right hand side of \eqref{eq:beextkernel} equals (in the case of BE/RBB)
\begin{equation}\label{eq:pathintbebb}
\det\!\left({\sf I}-\oeKGN+{\bar{\sf Q}}_{r_1}e^{(t_1-t_2){\sf D}}{\bar{\sf Q}}_{r_2}\dotsc{\bar{\sf Q}}_{r_n}e^{(t_n-t_1){\sf D}}\oeKGN\right)_{L^2_{{\rm odd}/{\rm even}}(\rr)},
\end{equation}
where ${\bar{\sf Q}}_{r}f(x)=f(x)\uno{|x|<r}$.
Note that at this point we may change the Hilbert space on which the Fredholm determinant is being computed to $L^2(\rr)$, because $\oeKGN$ are, respectively, the projections onto span$\{ \varphi_1,\varphi_3,\dotsc,\varphi_{2N-1}\}$ (which is a subspace of $L^2_{\rm odd}(\rr)$), and onto  span$\{ \varphi_0,\varphi_2,\dotsc,\varphi_{2N}\}$ (which is a subspace of $L^2_{\rm even}(\rr)$).

The same argument as in the case of Dyson Brownian motion (see \cite[Sec. 4.1.1]{bcr}) now allows us to take a limit of the last formula as the size of a mesh in $t$ goes to $0$.
The result is analogous to Proposition \ref{prop:dbmcont}.
Given a function $g\in H^1([\ell_1,\ell_2])$, define an operator $\Theta_{[\ell_1,\ell_2]}^{g,{\rm be}/{\rm rbb}}$ acting on $L^2(\rr)$ as follows: $\Theta_{[\ell_1,\ell_2]}^{g,{\rm be}/{\rm rbb}} f(x)=u(\ell_2,x)$, where $u(\ell_2,\cdot)$ is the solution at time $\ell_2$ of the boundary value problem\footnote{There is a minor detail missing in the derivation in \cite{bcr}. In view of the order in which the points $r_i$ appear in \eqref{eq:pathintbebb}, the boundary value PDE appearing in the continuum limit (given in the present case by \eqref{eq:bound-pde-be}) should be defined using $\hat g(t)=g(\ell_1+\ell_2-t)$ instead of $g$ itself. However, by the symmetry of $\oeKGN$ one may then take the adjoint of the resulting operator inside the kernel and use the cyclic property of the Fredholm determinant and the identity $(\Theta_{[\ell_1,\ell_2]}^{\hat g,{\rm be}/{\rm rbb}})^*=\Theta_{[\ell_1,\ell_2]}^{g,{\rm be}/{\rm rbb}}$ to obtain \eqref{eq:becont}.}
\begin{equation}\label{eq:bound-pde-be}
\begin{aligned}
\partial_t u+{\sf D}u&=0\quad\text{for}\ |x|<g(t),\ t\in(\ell_1,\ell_2)\\
u(\ell_1,x)&=f(x)\uno{|x|<g(\ell_1)}\\
u(t,x)&=0\quad \text{for}\ |x|\geq g(t).
\end{aligned}
\end{equation}

\begin{prop}\label{prop:becont}
  Given $g\in H^1([\ell_1,\ell_2])$, we have
  \begin{multline}\label{eq:becont}
    \pp\left(\sqrt{2}B^{{\rm be}/{\rm rbb}}_N\bigl(\tfrac{e^{2t}}{1+e^{2t}}\bigr)\leq g(t)\sech(t)~\forall\,t\in[\ell_1,\ell_2]\right)\\
    =\det\!\left({\sf I}-\oeKGN+\Theta_{[\ell_1,\ell_2]}^{g,{\rm be}/{\rm rbb}} e^{(\ell_2-\ell_1){\sf D}}\oeKGN\right).
  \end{multline}
\end{prop}

The PDE appearing in \eqref{eq:bound-pde-be} can be turned into a standard heat equation (with a modified boundary condition) by a suitable change of variables (see the proof of \cite[Prop. 2.2]{nibm-loe}, the computation in this case is exactly the same), and this leads to the following formula for the integral kernel of $\Theta_{[\ell_1,\ell_2]}^{g,{\rm be}/{\rm rbb}}$:
\begin{multline}\label{eq:bekernel}
  \Theta_{[\ell_1,\ell_2]}^{g,{\rm be}/{\rm rbb}}(x,y) = e^{\frac12(y^2-x^2)+\ell_2}\,
  \frac{e^{-(e^{\ell_1}x-e^{\ell_2}y)^2/(4(\beta-\alpha))}}{\sqrt{4\pi(\beta-\alpha)}}\\
  \times\mathbb{P}_{\subalign{\hat{b}(\alpha)&=e^{\ell_1}x,\\\hat{b}(\beta)&=e^{\ell_2}y}}\!\left(\bigl|\hat{b}(t)\bigr|\leq\sqrt{4t}\,g\bigl(\tfrac12\log(4t)\bigr)~\forall\, t\in[\alpha,\beta]\right),
\end{multline}
where the probability is computed with respect to a Brownian bridge\footnote{Note that the definition \eqref{eq:bekernel} is the same for the two cases BE/RBB; the probability on the right hand is always computed using the reflected Brownian bridge $|\hat{b}(t)|$. Our choice of superscript here, be/rbb, is intended to be consistent with the fact that this is the operator appearing in both cases in \eqref{eq:becont}.} $\hat{b}(t)$ from $e^{\ell_1}x$ at time $\alpha$ to $e^{\ell_2}y$ at time $\beta$ and with diffusion coefficient 2, and where $\alpha$ and $\beta$ are as in \eqref{eq:alphabeta}.

\subsection{Proof of Theorem \ref{thm:be-formula}}
Our interest is the case $g(t)=r\ttsm\cosh(t)$.
With this choice, the probability in the last formula reduces to the probability that a reflected Brownian bridge stays below the linear barrier $2rt+\tfrac{1}{2}r$.
Taking $-\ell_1=\ell_2=L$, we get
\begin{equation}\label{eq:limit-kernel-be}
\pp\left(\cm^{{\rm be}/{\rm rbb}}_N\leq r\right)
=\lim_{L\to \infty}\det\!\left({\sf I}-\oeKGN+\Theta^{(r),{\rm be}/{\rm rbb}}_{[-L,L]} e^{2L{\sf D}}\oeKGN\right)
\end{equation}
with
\begin{multline}\label{eq:defLambda}
	\Theta_{[\ell_1,\ell_2]}^{(r),{\rm be}/{\rm rbb}}(x,y):= \Theta_{[\ell_1,\ell_2]}^{g(t)=r\ttsm\cosh(t),{\rm be}/{\rm rbb}}(x,y)\\
	\quad=e^{\frac12(y^2-x^2)+\ell_2}\,\frac{e^{-(e^{\ell_1}x-e^{\ell_2}y)^2/(4(\beta-\alpha))}}{\sqrt{4\pi(\beta-\alpha)}}
	\mathbb{P}_{\subalign{\hat{b}(\alpha)&=e^{\ell_1}x,\\\hat{b}(\beta)&=e^{\ell_2}y}}\!\left(\bigl|\hat{b}(t)\bigr|\leq2rt+\tfrac12r~\forall\, t\in[\alpha,\beta]\right).
\end{multline}
Note that taking $r\to\infty$ in this formula corresponds to the solution of \eqref{eq:bound-pde-be} with $g=\infty$, which is just $e^{-(\ell_2-\ell_1){\sf D}}$, and on the other hand it corresponds to simply replacing the probability by 1.
(In particular, this implies that for any $\ell_1<\ell_2$ the kernel of the operator $e^{-(\ell_2-\ell_1){\sf D}}$ can be written as
\begin{equation}
	e^{-(\ell_2-\ell_1){\sf D}}(x,y)=e^{\frac12(y^2-x^2)+\ell_2}\,
  \frac{e^{-(e^{\ell_1}x-e^{\ell_2}y)^2/(4(\beta-\alpha))}}{\sqrt{4\pi(\beta-\alpha)}},\label{eq:etDkernel}
\end{equation}
which we will use in Section \ref{sec:bounds}).
As a consequence, we may rewrite our operator as follows:
\begin{equation}\label{eq:LambdaL}
\Theta_{[\ell_1,\ell_2]}^{(r),{\rm be}/{\rm rbb}}=\bar{\sf Q}_{r\ttsm\cosh(\ell_1)}\left(e^{-(\ell_2-\ell_1){\sf D}}-{\sf R}^{(r),{\rm be}/{\rm rbb}}_{[\ell_1,\ell_2]}\right)\bar{\sf Q}_{r\ttsm\cosh(\ell_2)},
\end{equation}
where ${\sf R}^{(r),{\rm be}/{\rm rbb}}_{[\ell_1,\ell_2]}$ is the reflection term
\begin{multline}\label{eq:Reflprob}
	{\sf R}^{(r),{\rm be}/{\rm rbb}}_{[\ell_1,\ell_2]}(x,y)\\
	=e^{\frac12(y^2-x^2)+\ell_2}\,
  	\frac{e^{-(e^{\ell_1}x-e^{\ell_2}y)^2/(4(\beta-\alpha))}}{\sqrt{4\pi(\beta-\alpha)}}\,\mathbb{P}_{\subalign{\hat{b}(\alpha)&=e^{\ell_1}x,\\\hat{b}(\beta)&=e^{\ell_2}y}}\!\left(\max_{t\in[\alpha,\beta]}\,\bigl|\hat{b}(t)\bigr|>2rt+\tfrac{1}{2}r\right).
\end{multline}
%\hphantom{^{-}}
The last probability can be computed explicitly using the reflection principle, and equals (see \cite{doob})
\begin{multline}
\sum_{k=1}^\infty\Bigl[e^{-2((2k-1)^2ac+bd)+2(2k-1)(ad+bc)}+e^{-2((2k-1)^2ac+bd)-2(2k-1)(ad+bc)}\\
-e^{-8k^2ac+4k(bc-ad)}-e^{-8k^2ac-4k(bc-ad)}\Bigr]
\end{multline}
with $a=\frac{r(4\alpha+1)}{2\sqrt{2}},\;b=\frac{e^{\ell_1}x}{\sqrt{2}},\; c=\sqrt{2}r+\frac{r(4\alpha+1)}{2\sqrt{2}(\beta-\alpha)}$ and $d=\frac{e^{\ell_2}y}{\sqrt{2}(\beta-\alpha)}$.
Through this formula ${\sf R}^{(r),{\rm be}/{\rm rbb}}_{[\ell_1,\ell_2]}$ gets written as $\sum_{k=1}^\infty\big[{\sf R}_{1,k}+{\sf R}_{2,k}-{\sf R}_{3,k}-{\sf R}_{4,k}\big]$ in the obvious way.
Focusing on ${\sf R}_{1,k}$, accounting for the prefactors on the right hand side of \eqref{eq:Reflprob}, and then comparing with the formula of the reflection operator ${\sf R}^{(r),{\rm bb}}_{[\ell_1,\ell_2]}$ in \eqref{eq:reflOperBB}, we have
\begin{align}
% e&^{\frac12(y^2-x^2)+\ell_2}\,\tfrac{e^{-(x-y)^2/(4(\beta-\alpha))}}{\sqrt{4\pi(\beta-\alpha)}}\\
% &\hspace{0.4in}\times e^{-2\bigl[(2k-1)^2\frac{r(4\alpha+1)}{2\sqrt{2}}\bigl(\sqrt{2}r+\frac{r(4\alpha+1)}{2\sqrt{2}(\beta-\alpha)}\bigr)+\frac{xy}{2(\beta-\alpha)}\bigr]+2(2k-1)\bigl[\frac{r(4\alpha+1)y}{4(\beta-\alpha)}+\frac{x}{\sqrt{2}}\bigl(\sqrt{2}r+\frac{r(4\alpha+1)}{2\sqrt{2}(\beta-\alpha)}\bigr)\bigr]}\\
{\sf R}_{1,k}(x,y)&=\tfrac{e^{\frac12(y^2-x^2)+\ell_2}}{\sqrt{4\pi(\beta-\alpha)}}e^{-(2k-1)r(e^{\ell_2}y-e^{\ell_1}x)+((2k-1)r)^2(\beta-\alpha)-(e^{\ell_1}x+e^{\ell_2}y-2(2k-1)r(\alpha+\beta)-(2k-1)r)^2/(4(\beta-\alpha))}\\
&={\sf R}^{((2k-1)r),{\rm bb}}_{[\ell_1,\ell_2]}(x,y).
\end{align}
One can obtain formulas for ${\sf R}_{2,k}$, ${\sf R}_{3,k}$ and ${\sf R}_{4,k}$, and putting everything together leads to
\begin{multline}\label{eq:infiniteRefl}
{\sf R}^{(r),{\rm be}/{\rm rbb}}_{[\ell_1,\ell_2]}(x,y)=\sum_{k=1}^\infty\Bigl[{\sf R}^{((2k-1)r),{\rm bb}}_{[\ell_1,\ell_2]}(x,y)+{\sf R}^{((1-2k)r),{\rm bb}}_{[\ell_1,\ell_2]}(x,y)\\
-{\sf R}^{(2kr),{\rm bb}}_{[\ell_1,\ell_2]}(x,-y)-{\sf R}^{(-2kr),{\rm bb}}_{[\ell_1,\ell_2]}(x,-y)\Bigr].
\end{multline}

The orthonormality of the $\varphi_n$'s together with \eqref{eq:etDoKGN} imply the identities $e^{2L{\sf D}}\oeKGN=(e^{L{\sf D}}\oeKGN)^2$ and $e^{-L{\sf D}}\oeKGN e^{L{\sf D}}\oeKGN=e^{L{\sf D}}\oeKGN e^{-L{\sf D}}\oeKGN=\oeKGN$.
Using this and the cyclic property of the Fredholm determinant, \eqref{eq:becont} and \eqref{eq:LambdaL} yield
\begin{equation}\label{eq:detRewr}
\pp\!\left(\cm^{{\rm be}/{\rm rbb}}_N\leq r\right)
=\lim_{L\to \infty}\det\!\left({\sf I}-\oeKGN+e^{L{\sf D}}\oeKGN\Theta^{(r),{{\rm be}/{\rm rbb}}}_{[-L,L]} e^{L{\sf D}}\oeKGN\right).
\end{equation}
We claim now that, in trace norm,
\begin{equation}\label{eq:errorRepl}
  \lim_{L\to\infty}e^{L{\sf D}}\oeKGN\Theta^{(r),{{\rm be}/{\rm rbb}}}_{[-L,L]} e^{L{\sf D}}\oeKGN=\lim_{L\to\infty}e^{L{\sf D}}\oeKGN(e^{-2L{\sf D}}-{\sf R}^{(r),{{\rm be}/{\rm rbb}}}_{[-L,L]})e^{L{\sf D}}\oeKGN,
\end{equation}
which just corresponds to removing the operators $\bar{\sf Q}_{r\ttsm\cosh(L)}$ in \eqref{eq:LambdaL}.
The proof of this is very similar to that of \cite[Lem. 2.3]{nibm-loe}.
There is an additional complication here related with the fact that ${\sf R}^{(r)}_L$ involves an infinite sum (see \eqref{eq:infiniteRefl}), but it is not hard to address this because the summands decay very rapidly in $k$.
Since we will deal with this issue later on in (see the discussion after \eqref{eq:limdelUps}), we will skip the proof of \eqref{eq:errorRepl}.

From \eqref{eq:detRewr} and \eqref{eq:errorRepl} we deduce that
\begin{equation}
\pp\!\left(\cm^{{\rm be}/{\rm rbb}}_N\leq r\right)
=\lim_{L\to \infty}\det\!\left({\sf I}-e^{L{\sf D}}\oeKGN {\sf R}^{(r),{\rm be}/{\rm rbb}}_{[-L,L]} e^{L{\sf D}}\oeKGN\right).
\end{equation}
The key point is then to compute the kernel $e^{L{\sf D}}\oeKGN {\sf R}^{(r),{\rm be}/{\rm rbb}}_{[-L,L]} e^{L{\sf D}}\oeKGN$. But, as we will see next, the result of this computation does not depend on $L$ (analogously to what happens for Dyson Brownian motion in \cite{nibm-loe} and for the Airy process in \cite{cqr}).

Define ${\sf S}^{{\rm be}/{\rm rbb}}_{i,k}:=e^{L{\sf D}}\oeKGN {\sf R}_{i,k}\ts e^{L{\sf D}}\oeKGN$ for $i\in\{1,\ldots,4\}$ (see \eqref{eq:infiniteRefl}), so that
\[e^{L{\sf D}}\oeKGN {\sf R}^{(r),{\rm be}/{\rm rbb}}_{[-L,L]} e^{L{\sf D}}\oeKGN=\sum_{k=1}^\infty\bigl[{\sf S}^{{\rm be}/{\rm rbb}}_{1,k}+{\sf S}^{{\rm be}/{\rm rbb}}_{2,k}-{\sf S}^{{\rm be}/{\rm rbb}}_{3,k}-{\sf S}^{{\rm be}/{\rm rbb}}_{4,k}\bigr].\]
In order to compute these products we will use the following formula:
\begin{equation}
e^{L{\sf D}}\oeKGN {\sf R}^{(r),{\rm bb}}_{[-L,L]} e^{L{\sf D}}\oeKGN=\oeKGN\varrho_r\oeKGN,\quad \forall r\in\rr,\ L\in\nn.
\end{equation}
Note that the operator on the right hand side does not depend on $L$.
This formula was proved in \cite[Lem. 2.4]{nibm-loe} in the BB case (that is, with $\oeKGN$ replaced by $\KGN$), but it is straightforward to check that the result still holds in the present case.
Using this for the first two operators, ${\sf S}^{{\rm be}/{\rm rbb}}_{1,k}$ and ${\sf S}^{{\rm be}/{\rm rbb}}_{2,k}$, leads directly to
\begin{equation}
{\sf S}^{{\rm be}/{\rm rbb}}_{1,k}=\oeKGN\varrho_{(2k-1)r}\oeKGN\quad\text{and}\quad{\sf S}^{{\rm be}/{\rm rbb}}_{2,k}=\oeKGN\varrho_{(1-2k)r}\oeKGN.
\end{equation}
For the last two terms we can write ${\sf S}^{{\rm be}/{\rm rbb}}_{3,k}=e^{L{\sf D}}\oeKGN {\sf R}^{(2kr),{\rm bb}}_{[-L,L]}\varrho_0e^{L{\sf D}}\oeKGN$ and ${\sf S}^{{\rm be}/{\rm rbb}}_{4,k}=e^{L{\sf D}}\oeKGN {\sf R}^{(-2kr),{\rm bb}}_{[-L,L]}\varrho_0e^{L{\sf D}}\oeKGN$, where we recall (see \eqref{eq:refl}) that $\varrho_0f(x)=f(-x)$. Using the parity properties of the Hermite functions, we have $\varrho_0e^{L{\sf D}}\oKGN=-e^{L{\sf D}}\oKGN$ and $\varrho_0e^{L{\sf D}}\eKGN=e^{L{\sf D}}\eKGN$ which implies
\begin{align}
{\sf S}^{{\rm be}}_{3,k}&=-\oKGN\varrho_{2kr}\oKGN,&{\sf S}^{{\rm be}}_{4,k}&=-\oKGN\varrho_{-2kr}\oKGN,\\
{\sf S}^{{\rm rbb}}_{3,k}&=\eKGN\varrho_{2kr}\eKGN,&{\sf S}^{{\rm rbb}}_{4,k}&=\eKGN\varrho_{-2kr}\eKGN.
\end{align}
To finish the proof observe that ${\sf S}^{{\rm be}/{\rm rbb}}_{1,k}={\sf S}^{{\rm be}/{\rm rbb}}_{2,k}$ and ${\sf S}^{{\rm be}/{\rm rbb}}_{3,k}={\sf S}^{{\rm be}/{\rm rbb}}_{4,k}$, which follows from the fact that $\int_\rr dx\,\varphi_n(x)\varphi_m(2r-x)=\int_\rr dx\,\varphi_n(x)\varphi_m(-2r-x)$ for all $r\in\rr$ and all $n,m\in\nn$ which are both either odd or even.

\section{Joint distribution of the max and the argmax}\label{sec:derivation}

\subsection{Proof of Theorems \ref{thm:densBB} and \ref{thm:half-line-argmax}}

We will write a single proof for the two results.
We will keep using the superscripts $\star$ and be/rbb (as in the previous section) when we write formulas which are valid, respectively, for the three models and for BE/RBB.

Throughout this and the next sections we will denote by $\|\cdot\|_1$ and $\|\cdot\|_2$ the trace class and Hilbert-Schmidt norms of operators on $L^2(\rr)$ ($\|\cdot\|_2$ will also be used to denote the $L^2(\rr)$ norm).
We recall that
\begin{equation}\label{eq:inenorms}
\|AB\|_1\leq\|A\|_2\|B\|_2,\quad\|AB\|_2\leq\|A\|_2\|B\|_2\qqand \|A\|_2^2=\int_{-\infty}^{\infty}\!dx\,dy\,A(x,y)^2
\end{equation}
if $A$ has integral kernel $A(x,y)$.
We will also use the bound
\begin{equation}\label{eq:inebound}
|\!\det(I+A)-\det(I+B)|\leq\|A-B\|_1 e^{\|A\|_1+\|B\|_1+1}\leq\|A-B\|_1 e^{\|A-B\|_1+2\|B\|_1+1}
\end{equation}
for any two trace class operators $A$ and $B$; for more details see \cite[Sec. 2]{quastelRem-review} or \cite{simon}.

The argument is based on the continuum statistics formulas in Propositions \ref{prop:dbmcont} and \ref{prop:becont}. It will be more convenient for us to work instead with
\begin{equation}
  \widehat\cm^\star_N=\max_{t\in\rr}\sqrt{2}B^\star_N\bigl(\tfrac{e^{2t}}{1+e^{2t}}\bigr)\qqand\widehat\ct^\star_N=\argmax_{t\in\rr}\sqrt{2}B^\star_N\bigl(\tfrac{e^{2t}}{1+e^{2t}}\bigr).
\end{equation}
To that end we introduce, for $r\geq0$ and $t\in\rr$, the functions
\begin{equation}\label{eq:whpsibb}
\begin{split}
 \widehat\psi^{\rm bb}_{r,t}(n)&=\sqrt{2\cosh(t)}\,e^{-nt}\bigl[\varphi_n'(r\ttsm\cosh(t))+r\ttsm\sinh(t)\varphi_n(r\ttsm\cosh(t))\bigr],\\
 \widehat\psi^{\rm be}_{r,t}(n)&=2\widehat\psi^{\rm bb}_{r,t}(n)+2\sqrt{2\cosh(t)}\,e^{-nt}\sum_{k=1}^{\infty}e^{k(k+1)r^2\sinh(2t)}\\
&\hspace{0.5in}\times\bigl[\varphi_n'\bigl((2k+1)r\ttsm\cosh(t)\bigr)+(2k+1)r\ttsm\sinh(t)\varphi_n\bigl((2k+1)r\ttsm\cosh(t)\bigr)\bigr],\\
 \widehat\psi^{\rm rbb}_{r,t}(n)&=2\widehat\psi^{\rm bb}_{r,t}(n)+2\sqrt{2\cosh(t)}\,e^{-nt}\sum_{k=1}^{\infty}(-1)^ke^{k(k+1)r^2\sinh(2t)}\\
&\hspace{0.5in}\times\bigl[\varphi_n'\bigl((2k+1)r\ttsm\cosh(t)\bigr)+(2k+1)r\ttsm\sinh(t)\varphi_n\bigl((2k+1)r\ttsm\cosh(t)\bigr)\bigr],
\end{split}
\end{equation}
and the rank one kernels
\begin{equation}\label{eq:whPsibb}
\begin{split}
  \widehat\Psi^{\rm bb}_{r,t}(x,y)&=\Bigl(\sum_{n=0}^{N-1}\varphi_{n}(x)\widehat\psi^{\rm bb}_{r,t}(n)\Bigr)\Bigl(\sum_{m=0}^{N-1}\varphi_{m}(y)\widehat\psi^{\rm bb}_{r,-t}(m)\Bigr),\\
	\widehat\Psi^{\rm be}_{r,t}(x,y)&=\Bigl(\sum_{n=0}^{N-1}\varphi_{2n+1}(x)\widehat\psi^{\rm be}_{r,t}(2n+1)\Bigr)\Bigl(\sum_{m=0}^{N-1}\varphi_{2m+1}(y)\widehat\psi^{\rm be}_{r,-t}(2m+1)\Bigr),\\
	\widehat\Psi^{\rm rbb}_{r,t}(x,y)&=\Bigl(\sum_{n=0}^{N-1}\varphi_{2n}(x)\widehat\psi^{\rm rbb}_{r,t}(2n)\Bigr)\Bigl(\sum_{m=0}^{N-1}\varphi_{2m}(y)\widehat\psi^{\rm rbb}_{r,-t}(2m)\Bigr).
\end{split}
\end{equation}

\begin{thm}\label{thm:whargmax}
  Let ${\widehat f}^\star_N(r,t)$ be the joint density of $\widehat\cm^\star_N$ and $\widehat\ct^\star_N$. Then for all $r>0$ and all $t\in\rr$, and with $\star$ standing for either $\rm bb$, $\rm be$ or $\rm rbb$, we have
	\begin{align}
	\widehat f^\star_N(r,t)&=\tr\!\left[({\sf I}-\sKGN\varrho^{\star}_r\sKGN)^{-1}\widehat\Psi^{\star}_{r,t}\right]\det\!\left({\sf I}-\sKGN \varrho^{\star}_r \sKGN\right)\\
	&=\det\!\left({\sf I}-\sKGN \varrho^{\star}_r \sKGN+\widehat\Psi^{\star}_{r,t}\right)-\det\!\left({\sf I}-\sKGN \varrho^{\star}_r \sKGN\right).
	\end{align}
\end{thm}
To recover Theorems \ref{thm:densBB} and \ref{thm:half-line-argmax} from this result it is enough to use a change of variables
\begin{equation}
f_N^\star(r,t)=\frac{1}{\sqrt{2}t(1-t)}\widehat f_N^\star\Bigl(\sqrt{2}r,\frac12 \log\Bigl(\frac{t}{1-t}\Bigr)\Bigr)\quad\text{for}\ r>0,\,t\in(0,1),
\end{equation}

\begin{proof}[Proof of Theorem \ref{thm:whargmax}]
We will proceed as in \cite{mqr}.
Let $(\widehat\cm^\star_{N,L},\widehat\ct^\star_{N,L})$ denote the maximum and the location of the maximum of $\sqrt{2}B^\star_N\bigl(\tfrac{e^{2t}}{1+e^{2t}}\bigr)$ restricted to $t\in[-L,L]$, that is,
\begin{equation}
\widehat\cm^\star_{N,L}=\max_{t\in[-L,L]}\sqrt{2}B^\star_N\bigl(\tfrac{e^{2t}}{1+e^{2t}}\bigr)\quad\text{and}\quad\widehat\ct^\star_{N,L}=\argmax_{t\in[-L,L]}\sqrt{2}B^\star_N\bigl(\tfrac{e^{2t}}{1+e^{2t}}\bigr),
\end{equation}
and let $\widehat f^\star_{N,L}(r,t)$ be their joint density.
Note that
\begin{equation}\label{eq:limLf}
\widehat f^\star_N(r,t)=\lim_{L\to\infty}\widehat f^\star_{N,L}(r,t).
\end{equation}
By definition
\begin{equation}
\widehat f^\star_{N,L}(r,t)=\lim_{\delta\to 0}\lim_{\ep\to 0}\frac{1}{\delta\ep}\pp\bigl(\cm^\star_{N,L}\in[r,r+\ep],\ct^\star_{N,L}\in[t,t+\delta]\bigr),
\end{equation}
provided that the limit exists.
If we denote by $\underline{D}_{\ep,\delta}^\star$ and $\overline{D}_{\ep,\delta}^\star$ the sets
\begin{multline}
\underline{D}_{\ep,\delta}^\star=\Bigl\{\sqrt{2}B^\star_N\bigl(\tfrac{e^{2s}}{1+e^{2s}}\bigr)\leq r,\,s\in[t,t+\delta]^\mathsf{c};\;\sqrt{2}B^\star_N\bigl(\tfrac{e^{2s}}{1+e^{2s}}\bigr)\leq r+\ep,\,s\in[t,t+\delta];\\
\sqrt{2}B^\star_N\bigl(\tfrac{e^{2s}}{1+e^{2s}}\bigr)\in[r,r+\ep]\ \text{for some}\ s\in[t,t+\delta]\Bigr\},
\end{multline}
and
\begin{equation}
\overline{D}_{\ep,\delta}^\star=\Bigl\{\sqrt{2}B^\star_N\bigl(\tfrac{e^{2s}}{1+e^{2s}}\bigr)\leq r+\ep,\,s\in[-L,L];\;\sqrt{2}B^\star_N\bigl(\tfrac{e^{2s}}{1+e^{2s}}\bigr)\in[r,r+\ep]\ \text{for some}\ s\in[t,t+\delta]\Bigr\},
\end{equation}
then
\begin{equation}
\underline{D}_{\ep,\delta}^\star\subseteq\bigl\{\cm^\star_{N,L}\in[r,r+\ep],\ct^\star_{N,L}\in[t,t+\delta]\bigr\}\subseteq\overline{D}_{\ep,\delta}^\star.
\end{equation}
Letting $\underline{f}^\star_{N,L}(r,t)=\lim_{\delta\to 0}\lim_{\ep\to 0}\frac{1}{\delta\ep}\pp(\underline{D}_{\ep,\delta}^\star)$ and defining $\overline{f}^\star_{N,L}(r,t)$ analogously 
we deduce that $\underline{f}^\star_{N,L}(r,t)\leq \widehat f^\star_{N,L}(r,t)\leq\overline{f}^\star_{N,L}(r,t)$.
We will only compute $\underline{f}^\star_{N,L}(r,t)$.
As in \cite{mqr}, it will be clear from the argument that for $\overline{f}^\star_{N,L}(r,t)$ we get the same limit, and thus
\begin{equation}
\widehat f^\star_{N,L}(r,t)=\lim_{\delta\to 0}\lim_{\ep\to 0}\frac{1}{\delta\ep}\pp(\underline{D}_{\ep,\delta}^\star).
\end{equation}

We rewrite the last equation as
\begin{multline}
\widehat f^\star_{N,L}(r,t)=\lim_{\delta\to 0}\lim_{\ep\to 0}\frac{1}{\delta\ep}\Bigl[\pp\bigl(\sqrt{2}B_N^\star\bigl(\tfrac{e^{2s}}{1+e^{2s}}\bigr)\leq h_{\ep,\delta}(s)\sech(s),\,s\in[-L,L]\bigr)\\
-\pp\bigl(\sqrt{2}B_N^\star\bigl(\tfrac{e^{2s}}{1+e^{2s}}\bigr)\leq h_{0,\delta}(s)\sech(s),\,s\in[-L,L]\bigr)\Bigr],
\end{multline}
where
\begin{equation}
h_{\ep,\delta}(s)=\cosh(s)(r+\ep\uno{s\in[t,t+\delta]}).
\end{equation}
These two probabilities have explicit Fredholm determinant formulas by Propositions \ref{prop:dbmcont} and \ref{prop:becont}.
We get, using the cyclic property of the determinants,
\begin{multline}
\widehat f^\star_{N,L}(r,t)=\lim_{\delta\to 0}\lim_{\ep\to 0}\frac{1}{\ep\delta}\Bigl[\det\!\left({\sf I}-\sKGN+e^{L{\sf D}}\sKGN\Theta_{[-L,L]}^{h_{\ep,\delta},\star}e^{L{\sf D}}\sKGN\right)\\
-\det\!\left({\sf I}-\sKGN+e^{L{\sf D}}\sKGN\Theta_{[-L,L]}^{h_{0,\delta},\star}e^{L{\sf D}}\sKGN\right)\Bigr],
\end{multline}
where $\Theta_{[-L,L]}^{h_{\ep,\delta},\star}$ means $\Theta_{[-L,L]}^{h_{\ep,\delta},{\rm bb}}$ in the case of BB and $\Theta_{[-L,L]}^{h_{\ep,\delta},{\rm be}/{\rm rbb}}$ in the case of BE/RBB.
The limit in $\ep$ becomes a derivative
\begin{equation}
\widehat f^\star_{N,L}(r,t)=\lim_{\delta\to 0}\frac{1}{\delta}\p_\beta\det\!\left({\sf I}-\sKGN+e^{L{\sf D}}\sKGN\Theta_{[-L,L]}^{h_{\beta,\delta},\star}e^{L{\sf D}}\sKGN\right)\Big|_{\beta=0},
\end{equation}
which in turn gives a trace (see e.g. \cite[Lem. A.2]{mqr}),
\begin{multline}
\widehat f^\star_{N,L}(r,t)=\det\!\left({\sf I}-\sKGN+e^{L{\sf D}}\sKGN\Theta_{[-L,L]}^{h_{0,\delta},\star}e^{L{\sf D}}\sKGN\right)\\
\times\lim_{\delta\to 0}\frac{1}{\delta}\tr\!\biggl[({\sf I}-\sKGN+e^{L{\sf D}}\sKGN\Theta_{[-L,L]}^{h_{0,\delta},\star}e^{L{\sf D}}\sKGN)^{-1}\\
\times e^{L{\sf D}}\sKGN\left[\p_\beta\Theta_{[-L,L]}^{h_{\beta,\delta},\star}\right]_{\beta=0}e^{L{\sf D}}\sKGN\biggr].
\end{multline}
Note that $h_{0,\delta}(s)=r\cosh(s)$, so in particular the determinant and the first factor inside the trace do not depend on $\delta$.
We know moreover, from \cite[Sec. 2.2]{nibm-loe} in the BB case and from Section \ref{sec:bm_halfspace} above in the BE/RBB case, that
\begin{equation}
\lim_{L\to\infty}\left({\sf I}-\sKGN+e^{L{\sf D}}\sKGN\Theta_{[-L,L]}^{h_{0,\delta},\star}e^{L{\sf D}}\sKGN\right)={\sf I}-\sKGN \varrho^\star_r \sKGN
\end{equation}
in trace norm (which implies that the same holds for the inverse of these operators).
Since the trace is linear and continuous under the trace norm topology, we deduce that
\begin{multline}\label{eq:providedTrNormLim}
\lim_{L\to\infty}\widehat f^\star_{N,L}(r,t)=\det\!\left({\sf I}-\sKGN \varrho^\star_r \sKGN\right)\\
\times\tr\!\left[\left({\sf I}-\sKGN \varrho^\star_r \sKGN\right)^{-1}\lim_{L\to\infty}\lim_{\delta\to 0}\frac{1}{\delta}e^{L{\sf D}}\sKGN\left[\p_\beta\Theta_{[-L,L]}^{h_{\beta,\delta},\star}\right]_{\beta=0}e^{L{\sf D}}\sKGN\right],
\end{multline}
provided that the limit inside the trace exists in operator norm (here we are using the fact that $\|AB\|_{1}\leq\|A\|_1\|B\|_{\rm op}$, where $\|\cdot\|_{\rm op}$ is the operator norm).

The next step is to compute the limit $\lim_{L\to\infty}e^{L{\sf D}}\sKGN\left[\lim_{\delta\to 0}\frac{1}{\delta}\p_\beta\Theta_{[-L,L]}^{h_{\beta,\delta},\star}\right]_{\beta=0}e^{L{\sf D}}\sKGN$.
To this end we need some additional notation.
Let
\begin{equation}\label{eq:wtthetarRpre}
  \wt\Theta^{(r),\star}_{[\ell_1,\ell_2]}=e^{-(\ell_2-\ell_1){\sf D}}-{\sf R}^{(r),\star}_{[\ell_1,\ell_2]}.
\end{equation}
Comparing with \eqref{eq:thetarRpre} and \eqref{eq:LambdaL}, $\wt\Theta^{(r),\star}_{[\ell_1,\ell_2]}$ is simply $\Theta^{(r),\star}_{[\ell_1,\ell_2]}$ with the projections on the two sides removed.
Next we introduce the kernels
\begin{equation}\label{eq:Thetakernels}
\begin{aligned}
\varTheta_1^{\rm bb}(x,z_1)&=e^{-z_1^2/2}\wt{\Theta}^{(r),{\rm bb}}_{[-L,t]}(x,z_1)\uno{x\leq r\ttsm\cosh(L)},\\
\varTheta_2^{\rm bb}(z_2,y)&=e^{z_2^2/2}\wt{\Theta}^{(r),{\rm bb}}_{[t,L]}(z_2,y)\uno{y\leq r\ttsm\cosh(L)},\\
\varTheta_1^{{\rm be}/{\rm rbb}}(x,z_1)&=e^{-z_1^2/2}\wt{\Theta}^{(r),{\rm be}/{\rm rbb}}_{[-L,t]}(x,z_1)\uno{|x|\leq r\ttsm\cosh(L)},\\
\varTheta_2^{{\rm be}/{\rm rbb}}(z_2,y)&=e^{z_2^2/2}\wt{\Theta}^{(r),{\rm be}/{\rm rbb}}_{[t,L]}(z_2,y)\uno{|y|\leq r\ttsm\cosh(L)}.
\end{aligned}
\end{equation}
and define further, for $a\in\rr$,
\begin{equation}\label{eq:psirankone}
\begin{aligned}
\widehat{\varPhi}^{1,\star}_{L,a}(x)&=\sqrt{\tfrac{\cosh(t)}{2}}\,\bigl.\p_w\bigl(e^{L{\sf D}}\sKGN\varTheta_1^{\star}(x,w)\bigr)\bigr|_{w=a},\\
\widehat{\varPhi}^{2,\star}_{L,a}(y)&=\sqrt{\tfrac{\cosh(t)}{2}}\,\bigl.\p_w\bigl(\varTheta_2^{\star}e^{L{\sf D}}\sKGN(w,y)\bigr)\bigr|_{w=a}.
\end{aligned}
\end{equation}

\begin{lem}\label{lem:limdelta}
The following limits hold in the operator norm topology:
\begin{equation}\label{eq:nibmlimdelta}
\lim_{\delta\to 0}\frac{1}{\delta}\left[\p_\beta\Theta_{[-L,L]}^{h_{\beta,\delta},{\rm bb}}\right]_{\beta=0}(x,y)=\frac{\cosh(t)}{2}\left(\p_w\varTheta_1^{\rm bb}(x,w)\p_w\varTheta_2^{\rm bb}(w,y)\right)\bigg|_{w=r\ttsm\cosh(t)}
\end{equation}
and
\begin{multline}\label{eq:belimdelta}
\lim_{\delta\to 0}\frac{1}{\delta}\left[\p_\beta\Theta_{[-L,L]}^{h_{\beta,\delta},{{\rm be}/{\rm rbb}}}\right]_{\beta=0}(x,y)
=\frac{\cosh(t)}{2}\left(\p_w\varTheta_1^{{\rm be}/{\rm rbb}}(x,w)\p_w\varTheta_2^{{\rm be}/{\rm rbb}}(w,y)\right)\bigg|_{w=r\ttsm\cosh(t)}\\
+\frac{\cosh(t)}{2}\left(\p_w\varTheta_1^{{\rm be}/{\rm rbb}}(x,w)\p_w\varTheta_2^{{\rm be}/{\rm rbb}}(w,y)\right)\bigg|_{w=-r\ttsm\cosh(t)}.
\end{multline}
\end{lem}

It follows based on the lemma that
\begin{equation}\label{eq:basedonlemma}
\lim_{\delta\to 0}\frac{1}{\delta}e^{L{\sf D}}\sKGN\left[\p_\beta\Theta_{[-L,L]}^{h_{\beta,\delta},\star}\right]_{\beta=0}e^{L{\sf D}}\sKGN=\widehat{\Psi}^\star_L,
\end{equation}
where $\widehat{\Psi}^\star_L$ has kernel
\begin{equation}\label{eq:hatPsis}
\begin{aligned}
\widehat{\Psi}^{\rm bb}_L(x,y)&=\widehat{\varPhi}^{1,{\rm bb}}_{L,r\ttsm\cosh(t)}(x)\widehat{\varPhi}^{2,{\rm bb}}_{L,r\ttsm\cosh(t)}(y),\\
\widehat{\Psi}^{{\rm be}/{\rm rbb}}_L(x,y)&=\widehat{\varPhi}^{1,{{\rm be}/{\rm rbb}}}_{L,r\ttsm\cosh(t)}(x)\widehat{\varPhi}^{2,{{\rm be}/{\rm rbb}}}_{L,r\ttsm\cosh(t)}(y)+\widehat{\varPhi}^{1,{{\rm be}/{\rm rbb}}}_{L,-r\ttsm\cosh(t)}(x)\widehat{\varPhi}^{2,{{\rm be}/{\rm rbb}}}_{L,-r\ttsm\cosh(t)}(y).
\end{aligned}
\end{equation}
We thus need to compute the $L\to\infty$ limit of $\widehat{\Psi}^{\star}_L$.
This is the content of the next result.

\begin{lem}\label{lem:limLPsi}
Let $\widehat\Psi^{\star}_{r,t}$ be the kernel given by \eqref{eq:whPsibb}. Then the following limit holds in the Hilbert-Schmidt norm:
\begin{equation}\label{eq:LlimitPsi}
\widehat{\Psi}^{\star}_L(x,y)\xrightarrow[L\to\infty]{}\widehat\Psi^{\star}_{r,t}(x,y).
\end{equation}
\end{lem}

In view of \eqref{eq:limLf}, \eqref{eq:providedTrNormLim} and \eqref{eq:basedonlemma}, this lemma completes the proof of Theorem \ref{thm:whargmax}.
\end{proof}

We provide now the proof of Lemma \ref{lem:limLPsi}, the proof of Lemma \ref{lem:limdelta} comes afterwards.

\begin{proof}[Proof of Lemma \ref{lem:limLPsi}]
Let us focus for a moment now on the BB case.
We begin by using the formula for $\wt{\Theta}^{(r),{\rm bb}}_{[-L,t]}$ (see \eqref{eq:wtthetarRpre}, \eqref{eq:reflOperBB} and \eqref{eq:etDkernel}) to compute
\begin{align}
\theta(x)&:=\bigl.\p_w\bigl(e^{-w^2/2}\wt{\Theta}^{(r),{\rm bb}}_{[-L,t]}(x,w)\bigr)\bigr|_{w=r\ttsm\cosh(t)}\\
&=\tfrac{1}{\sqrt{\pi(e^{2t}-e^{-2L})}}\bigl.\p_w\bigl(e^{-x^2/2+t-(e^{-L}x-e^{t}w)^2/(e^{2t}-e^{-2L})}\\
&\hspace{0.2in}-e^{-x^2/2+t-r(e^{t}w-e^{-L}x)+r^2(e^{2t}-e^{-2L})/4-(e^{-L}x+e^{t}w-r(e^{2t}+e^{-2L})/2-r)^2/(e^{2t}-e^{-2L})}\bigr)\bigr|_{w=r\ttsm\cosh(t)}\\
&=\tfrac{4(x-r\ttsm\cosh(L))}{\sqrt{\pi}(e^{2t}-e^{-2L})^{3/2}}\,e^{-\frac{2x^2(e^{2t}+e^{-2L})-4rxe^{-L}(1+e^{2t})+r^2(e^{2t}+1)^2}{4(e^{2t}-e^{-2L})}-L+2t}.
\end{align}
From \eqref{eq:psirankone} we then have
\begin{equation}
\widehat{\varPhi}^{1,{\rm bb}}_{L,r\ttsm\cosh(t)}(x)=\sqrt{\tfrac{\cosh(t)}{2}}\,e^{L{\sf D}}\KGN\bar{\sf P}_{r\ttsm\cosh(L)}\theta(x)=\sqrt{\tfrac{\cosh(t)}{2}}\left(e^{L{\sf D}}\KGN\theta(x)-e^{L{\sf D}}\KGN{\sf P}_{r\ttsm\cosh(L)}\theta(x)\right).
\end{equation}
$\theta(x)$ is essentially a Gaussian, and thus we have the same estimate as in \cite[App. B]{nibm-loe}: for some constants $c_1,c_2>0$,
\begin{equation}
\|e^{L{\sf D}}\KGN{\sf P}_{r\ttsm\cosh(L)}\theta\|_{L^2(\rr)}\leq c_1\tts e^{NL-c_2e^{2L}}\xrightarrow[L\to\infty]{}0.
\end{equation}
This implies that, in computing the limit of $e^{L{\sf D}}\KGN\bar{\sf P}_{r\ttsm\cosh(L)}\theta$, we may erase the projection in the middle and work instead with $e^{L{\sf D}}\KGN\theta$.
As we will see next, this last kernel does not depend on $L$.
We start by writing it
\begin{equation}
e^{L{\sf D}}\KGN\theta(x)=\sum_{n=0}^{N-1}e^{Ln}\varphi_n(x)\int_{-\infty}^\infty dz\,\varphi_n(z)\theta(z)
\end{equation}
and then use the contour representation of the Hermite polynomials
\begin{equation}\label{eq:integhermite}
\varphi_n(z)=(2^nn!\sqrt{\pi})^{-1/2}e^{-z^2/2}\frac{n!}{2\pi\I}\oint du \frac{e^{2uz-u^2}}{u^{n+1}},
\end{equation}
(where the contour of integration encircles the origin) to compute the $z$ integral, which is just a Gaussian integral:
\begin{align}
\int_{-\infty}^\infty &dz\,\varphi_n(z)\theta(z)\\
&=(2^nn!\sqrt{\pi})^{-1/2}\frac{n!}{2\pi\I}\oint du \frac{1}{u^{n+1}\sqrt{\pi}(e^{2t}-e^{-2L})^{3/2}}\\
&\hspace{0.3in}\times\int_{-\infty}^\infty dz\,4(z-r\ttsm\cosh(L))e^{-z^2/2+2uz-u^2-\frac{2z^2(e^{2t}+e^{-2L})-4rze^{-L}(1+e^{2t})+r^2(e^{2t}+1)^2}{4(e^{2t}-e^{-2L})}-L+2t}\\
&=(2^nn!\sqrt{\pi})^{-1/2}\frac{n!}{2\pi\I}\oint du \frac{e^{-u^2e^{-2L-2t}+2ue^{-L-t}r\ttsm\cosh(t)-r^2\cosh(t)^2-L-t}}{u^{n+1}}(4u-2re^L).
\end{align}
Now we perform the change of variable $u\mapsto ue^{L+t}$ to deduce from \eqref{eq:integhermite} that
\begin{align}
\int_{-\infty}^\infty dz\,\varphi_n(z)\theta(z)&=(2^nn!\sqrt{\pi})^{-1/2}\frac{n!}{2\pi\I}\oint du \frac{e^{-u^2+2ur\ttsm\cosh(t)-r^2\cosh(t)^2}}{u^{n+1}e^{(n+1)(L+t)}}(4ue^{L+t}-2re^L)\\
&=2\tts e^{-\tfrac{r^2\cosh(t)^2}{2}-n(L+t)}\bigl[\varphi_n'(r\ttsm\cosh(t))+r\ttsm\sinh(t)\varphi_n(r\ttsm\cosh(t))\bigr].
\end{align}
Therefore
\begin{equation}
e^{L{\sf D}}\KGN\theta(x)=\sum_{n=0}^{N-1}2\tts e^{-\tfrac{r^2\cosh(t)^2}{2}-nt}\varphi_n(x)\bigl[\varphi_n'(r\ttsm\cosh(t))+r\ttsm\sinh(t)\varphi_n(r\ttsm\cosh(t))\bigr].
\end{equation}
We have proved that
\begin{equation}
\widehat{\varPhi}^{1,{\rm bb}}_{L,r\ttsm\cosh(t)}(x)\xrightarrow[L\to\infty]{}e^{-\tfrac{r^2\cosh(t)^2}{2}}\sum_{n=0}^{N-1}\varphi_n(x)\widehat\psi^{\rm bb}_{r,t}(n)
\end{equation}
in $L^2(\rr)$, where
\begin{equation}
\widehat\psi^{\rm bb}_{r,t}(n)=\sqrt{2\cosh(t)}\,e^{-nt}\bigl[\varphi_n'(r\ttsm\cosh(t))+r\ttsm\sinh(t)\varphi_n(r\ttsm\cosh(t))\bigr].
\end{equation}
The exact same computations leads also to
\begin{equation}
\widehat{\varPhi}^{2,{\rm bb}}_{L,r\ttsm\cosh(t)}(y)\xrightarrow[L\to\infty]{}e^{\tfrac{r^2\cosh(t)^2}{2}}\sum_{n=0}^{N-1}\varphi_n(y)\widehat\psi^{\rm bb}_{r,-t}(n)
\end{equation}
in $L^2(\rr)$.
Putting two limits together and recalling the definition of $\widehat\Psi^{\rm bb}_{r,t}$ in \eqref{eq:whPsibb}, we complete the proof for the case of BB.

In the case of BE/RBB, the same arguments lead to the following formula
(note that the two terms in the definition of $\widehat{\Psi}^{\rm be/rbb}_L$ in \eqref{eq:hatPsis} become just one in this formula; this is because, thanks to the parity properties of the Hermite functions, the evaluation at $w=r\cosh(t)$ and $w=-r\cosh(t)$ give the same answer):
\begin{align}
\widehat{\Psi}^{\rm be}_L(x,y)\xrightarrow[L\to\infty]{}\widehat\Psi_{r,t}^{\rm be}(x,y)&=\Bigl(\sum_{n=0}^{N-1}\varphi_{2n+1}(x)\widehat\psi^{\rm be}_{r,t}(2n+1)\Bigr)\Bigl(\sum_{m=0}^{N-1}\varphi_{2m+1}(y)\widehat\psi^{\rm be}_{r,-t}(2m+1)\Bigr),\\
\widehat{\Psi}^{\rm rbb}_L(x,y)\xrightarrow[L\to\infty]{}\widehat\Psi_{r,t}^{\rm rbb}(x,y)&=\Bigl(\sum_{n=0}^{N-1}\varphi_{2n}(x)\widehat\psi^{\rm rbb}_{r,t}(2n)\Bigr)\Bigl(\sum_{m=0}^{N-1}\varphi_{2m}(y)\widehat\psi^{\rm rbb}_{r,-t}(2m)\Bigr),
\end{align}
where
\begin{multline}
\widehat\psi^{\rm be}_{r,t}(n)=2\sqrt{2\cosh(t)}\,e^{-nt}\biggl[\varphi_n'(r\ttsm\cosh(t))+r\ttsm\sinh(t)\varphi_n(r\ttsm\cosh(t))\\
+\sum_{k=1}^{\infty}e^{k(k+1)r^2\sinh(2t)}\bigl[\varphi_n'\bigl((2k+1)r\ttsm\cosh(t)\bigr)+(2k+1)r\ttsm\sinh(t)\varphi_n\bigl((2k+1)r\ttsm\cosh(t)\bigr)\bigr]\biggr],
\end{multline}
and
\begin{multline}
\widehat\psi^{\rm rbb}_{r,t}(n)=2\sqrt{2\cosh(t)}\,e^{-nt}\biggl[\varphi_n'(r\ttsm\cosh(t))+r\ttsm\sinh(t)\varphi_n(r\ttsm\cosh(t))\\
+\sum_{k=1}^{\infty}(-1)^ke^{k(k+1)r^2\sinh(2t)}\bigl[\varphi_n'\bigl((2k+1)r\ttsm\cosh(t)\bigr)+(2k+1)r\ttsm\sinh(t)\varphi_n\bigl((2k+1)r\ttsm\cosh(t)\bigr)\bigr]\biggr],
\end{multline}
and this leads to the desired formula.
\end{proof}

\begin{proof}[Proof of Lemma \ref{lem:limdelta}]
We will focus on the BB case, and then explain the main differences for BE/RBB.
Recalling that $h_{\ep,\delta}(s)=\cosh(s)(r+\ep\uno{s\in[t,t+\delta]})$ we have, by the semigroup property,
\begin{equation}
\lim_{\epsilon\to 0}\frac{1}{\epsilon}\left[\Theta_{[-L,L]}^{h_{\ep,\delta},{\rm bb}}-\Theta_{[-L,L]}^{h_{0,\delta},{\rm bb}}\right]=\Theta_{[-L,t]}^{(r),{\rm bb}}\left[\lim_{\epsilon\to 0}\frac{1}{\epsilon}\left[\Theta_{[t,t+\delta]}^{(r+\ep),{\rm bb}}-\Theta_{[t,t+\delta]}^{(r),{\rm bb}}\right]\right]\Theta_{[t+\delta,L]}^{(r),{\rm bb}}.
\end{equation}
Using the kernel $\wt{\Theta}_{[\ell_1,\ell_2]}^{(r),{\rm bb}}$ (see its definition in \eqref{eq:wtthetarRpre}) we can write
\begin{multline}\label{eq:nibmdevbeta}
\left[\p_\beta\Theta_{[-L,L]}^{h_{\beta,\delta},{\rm bb}}(x,y)\right]_{\beta=0}=\int_{-\infty}^{r\ttsm\cosh(t)}\!dz_1\int_{-\infty}^{r\ttsm\cosh(t+\delta)}\!dz_2\,\wt{\Theta}^{(r),{\rm bb}}_{[-L,t]}(x,z_1)\uno{x\leq r\ttsm\cosh(L)}\\
\times\left(\bigl[\p_\ep\wt{\Theta}_{[t,t+\delta]}^{(r+\ep),{\rm bb}}(z_1,z_2)\bigr]_{\ep=0}\right)\wt{\Theta}^{(r),{\rm bb}}_{[t+\delta,L]}(z_2,y)\uno{y\leq r\ttsm\cosh(L)}.
\end{multline}
For convenience, we let $\alpha=\frac{1}{4}e^{2t}$ and $\beta=\frac{1}{4}e^{2(t+\delta)}$ and introduce the kernel
\begin{equation}\label{eq:Upsilonbb}
\Upsilon^{\rm bb}(z_1,z_2)=e^{(z_1^2-z_2^2)/2}\bigl[\p_\ep\wt{\Theta}_{[t,t+\delta]}^{(r+\ep),{\rm bb}}(z_1,z_2)\bigr]_{\ep=0}.
\end{equation}
We perform the change of variables $z_1\mapsto e^{-t}\sqrt{\beta-\alpha}\,z_1+r\ttsm\cosh(t)$ and $z_2\mapsto e^{-t-\delta}\sqrt{\beta-\alpha}\,z_2+r\ttsm\cosh(t+\delta)$ above and use the kernels $\varTheta_1^{\rm bb}$ and $\varTheta_2^{\rm bb}$ defined in \eqref{eq:Thetakernels} to write
%\begin{align}
%&\tfrac1\delta\tsm\left[\p_\beta\Theta_{[-L,L]}^{h_{\beta,\delta},{\rm bb}}(x,y)\right]_{\beta=0}=\int_{-\infty}^{0}\!dz_1\int_{-\infty}^{0}\!dz_2\,\frac1\delta e^{-2t-\delta}(\beta-\alpha)\\
%&\hspace{0.4in}\times\tfrac1{\sqrt{\delta}}\varTheta_1^{\rm bb}\bigl(x,e^{-t}\sqrt{\beta-\alpha}\,z_1+r\ttsm\cosh(t)\bigr)\tfrac1{\sqrt{\delta}}\varTheta_2^{\rm bb}\bigl(e^{-t-\delta}\sqrt{\beta-\alpha}\,z_2+r\ttsm\cosh(t+\delta),y\bigr)\\
%&\hspace{1.3in}\times \delta\Upsilon^{\rm bb}\bigl(e^{-t}\sqrt{\beta-\alpha}\,z_1+r\ttsm\cosh(t),e^{-t-\delta}\sqrt{\beta-\alpha}\,z_2+r\ttsm\cosh(t+\delta)\bigr).
%\end{align}
% \begin{multline}
% \left[\p_\beta\Theta_{[-L,L]}^{h_{\beta,\delta},{\rm bb}}(x,y)\right]_{\beta=0}=\int_{-\infty}^{0}\!dz_1\int_{-\infty}^{0}\!dz_2\,e^{-2t-\delta}(\beta-\alpha)\varTheta_1^{\rm bb}\bigl(x,e^{-t}\sqrt{\beta-\alpha}\,z_1+r\ttsm\cosh(t)\bigr)\\
% \times\varTheta_2^{\rm bb}\bigl(e^{-t-\delta}\sqrt{\beta-\alpha}\,z_2+r\ttsm\cosh(t+\delta),y\bigr)\\
% \times\Upsilon^{\rm bb}\bigl(e^{-t}\sqrt{\beta-\alpha}\,z_1+r\ttsm\cosh(t),e^{-t-\delta}\sqrt{\beta-\alpha}\,z_2+r\ttsm\cosh(t+\delta)\bigr),
% \end{multline}
\begin{multline}
\tfrac1\delta\tsm\left[\p_\beta\Theta_{[-L,L]}^{h_{\beta,\delta},{\rm bb}}(x,y)\right]_{\beta=0}=\int_{-\infty}^{0}\!dz_1\int_{-\infty}^{0}\!dz_2\,\tfrac1{\sqrt{\delta}}\varTheta_1^{\rm bb}\bigl(x,e^{-t}\sqrt{\beta-\alpha}\,z_1+r\ttsm\cosh(t)\bigr)\\
\times\tfrac1{\sqrt{\delta}}\varTheta_2^{\rm bb}\bigl(e^{-t-\delta}\sqrt{\beta-\alpha}\,z_2+r\ttsm\cosh(t+\delta),y\bigr)\\
\times e^{-2t-\delta}(\beta-\alpha)\Upsilon^{\rm bb}\bigl(e^{-t}\sqrt{\beta-\alpha}\,z_1+r\ttsm\cosh(t),e^{-t-\delta}\sqrt{\beta-\alpha}\,z_2+r\ttsm\cosh(t+\delta)\bigr).
\end{multline}
Now we need to take $\delta\to0$.
Note that $\beta\longrightarrow\alpha$ in this limit and $\varTheta_i^{\rm bb}\bigl(x,r\ttsm\cosh(t)\bigr)=0$ for $i=1,2$, so from the first two lines we get the product of two derivatives. The limit of the last line above can also be computed explicitly, by using (see \eqref{eq:Upsilonbb})
\begin{multline}
\Upsilon^{\rm bb}\bigl(e^{-t}\sqrt{\beta-\alpha}\,z_1+r\ttsm\cosh(t),e^{-t-\delta}\sqrt{\beta-\alpha}\,z_2+r\ttsm\cosh(t+\delta)\bigr)\\
=-\frac{\beta z_1+\alpha z_2+(z_1+z_2)/4}{\sqrt{\pi}(\beta-\alpha)}\,e^{t+\delta+r\sqrt{\beta-\alpha}(z_1-z_2)-r^2(\beta-\alpha)-(z_1+z_2)^2/4},
\end{multline}
and yields $-\frac{(z_1+z_2)\ttsm\cosh(t)}{2\sqrt{\pi}}e^{-(z_1+z_2)^2/4}$. Therefore
\begin{multline}
\lim_{\delta\to 0}\frac{1}{\delta}\left[\p_\beta\Theta_{[-L,L]}^{h_{\beta,\delta},{\rm bb}}\right]_{\beta=0}(x,y)=\int_{-\infty}^{0}\!dz_1\int_{-\infty}^{0}\!dz_2\,\tfrac{z_1z_2(-z_1-z_2)\ttsm\cosh(t)}{4\sqrt{\pi}}e^{-(z_1+z_2)^2/4}\\
\times\bigl.\p_w\varTheta_1^{\rm bb}(x,w)\bigr|_{w=r\ttsm\cosh(t)}\bigl.\p_w\varTheta_2^{\rm bb}(w,y)\bigr|_{w=r\ttsm\cosh(t)}.
\end{multline}
In order to obtain \eqref{eq:nibmlimdelta}, we evaluate the integral in $z_1$ and $z_2$ which gives $2\sqrt{\pi}$. This completes the proof for the case of BB.

In the case of BE/RBB we have
\begin{multline}\label{eq:bedevbeta}
\left[\p_\beta\Theta_{[-L,L]}^{h_{\beta,\delta},{{\rm be}/{\rm rbb}}}(x,y)\right]_{\beta=0}=\int_{-r\ttsm\cosh(t)}^{r\ttsm\cosh(t)}\!dz_1\int_{-r\ttsm\cosh(t+\delta)}^{r\ttsm\cosh(t+\delta)}\!dz_2\,\wt{\Theta}^{(r),{{\rm be}/{\rm rbb}}}_{[-L,t]}(x,z_1)\uno{|x|\leq r\ttsm\cosh(L)}\\
\times\left(\bigl[\p_\ep\wt{\Theta}_{[t,t+\delta]}^{(r+\ep),{{\rm be}/{\rm rbb}}}(z_1,z_2)\bigr]_{\ep=0}\right)\wt{\Theta}^{(r),{{\rm be}/{\rm rbb}}}_{[t+\delta,L]}(z_2,y)\uno{|y|\leq r\ttsm\cosh(L)}.
\end{multline}
We may use the formulas in \eqref{eq:LambdaL} and \eqref{eq:infiniteRefl} to write the derivative in $\ep$ as
\begin{equation}
\sum_{k\in\zz\setminus\{0\}}\left[(-1)^k\p_\ep{\sf R}_{[t,t+\delta]}^{((r+\ep)k),{\rm bb}}\bigl(z_1,(-1)^{k+1}z_2\bigr)\right]_{\ep=0}
\end{equation}
and then proceed as above.
We introduce the kernel
\begin{equation}
\Upsilon^{{\rm be}/{\rm rbb}}_k(z_1,z_2)=e^{(z_1^2-z_2^2)/2}\left[(-1)^k\p_\ep{\sf R}_{[t,t+\delta]}^{((r+\ep)k),{\rm bb}}\bigl(z_1,(-1)^{k+1}z_2\bigr)\right]_{\ep=0},
\end{equation} 
and then perform the change of variables $z_1\mapsto e^{-t}\sqrt{\beta-\alpha}\,z_1+kr\ttsm\cosh(t)$ and $z_2\mapsto e^{-t-\delta}\sqrt{\beta-\alpha}\,z_2+(-1)^{k+1}kr\ttsm\cosh(t+\delta)$ (here we still use the notations $\alpha=\frac{1}{4}e^{2t}$, $\beta=\frac{1}{4}e^{2(t+\delta)}$ and the kernels $\varTheta_1^{{\rm be}/{\rm rbb}}$, $\varTheta_2^{{\rm be}/{\rm rbb}}$ defined in \eqref{eq:Thetakernels}) to get
\begin{multline}\label{eq:ksum}
\left[\p_\beta\Theta_{[-L,L]}^{h_{\beta,\delta},{{\rm be}/{\rm rbb}}}(x,y)\right]_{\beta=0}\\
=\sum_{k\in\zz\setminus\{0\}}\int_{(-k-1)r\ttsm\cosh(t)e^{t}/\sqrt{\beta-\alpha}}^{(-k+1)r\ttsm\cosh(t)e^{t}/\sqrt{\beta-\alpha}}\!dz_1\int_{((-1)^{k}k-1)r\cosh(t+\delta)e^{t+\delta}/\sqrt{\beta-\alpha}}^{((-1)^{k}k+1)r\cosh(t+\delta)e^{t+\delta}/\sqrt{\beta-\alpha}}\!dz_2\,e^{-2t-\delta}(\beta-\alpha)\\
\times\varTheta_1^{{\rm be}/{\rm rbb}}\bigl(x,e^{-t}\sqrt{\beta-\alpha}\,z_1+kr\ttsm\cosh(t)\bigr)\varTheta_2^{{\rm be}/{\rm rbb}}\bigl(e^{-t-\delta}\sqrt{\beta-\alpha}\,z_2+(-1)^{k+1}kr\ttsm\cosh(t+\delta),y\bigr)\\
\times\Upsilon^{{\rm be}/{\rm rbb}}_k\bigl(e^{-t}\sqrt{\beta-\alpha}\,z_1+kr\ttsm\cosh(t),e^{-t-\delta}\sqrt{\beta-\alpha}\,z_2+(-1)^{k+1}kr\ttsm\cosh(t+\delta)\bigr),
\end{multline}
where the kernel $\Upsilon^{{\rm be}/{\rm rbb}}_k$ can be computed explicitly and satisfies the following limit: writing $\gamma=\beta- \alpha$
\begin{multline}\label{eq:limdelUps}
\lim_{\delta\to 0}e^{-2t-\delta}\gamma\Upsilon^{{\rm be}/{\rm rbb}}_k\bigl(e^{-t}\sqrt{\gamma}\,z_1+kr\ttsm\cosh(t),e^{-t-\delta}\sqrt{\gamma}\,z_2+(-1)^{k+1}kr\ttsm\cosh(t+\delta)\bigr)\\
=-\frac{k\ttsm\cosh(t)((-1)^{k+1}z_1+z_2)}{2\sqrt{\pi}}e^{-((-1)^{k+1}z_1+z_2)^2/4}.
\end{multline}
We split the $k$ sum in \eqref{eq:ksum} into two regions, $\zz\setminus\{-1,0,1\}$ and $\{-1,1\}$. For each $k$ in the first region, since the kernels have a Gaussian form and since note that $1/\sqrt{\beta-\alpha}\to\infty$ as $\delta\to0$, it is not hard to see that the double integral can be bounded by $c_1\tts e^{-c_2k^2/(\beta-\alpha)}$ for some contants $c_1,c_2>0$ independent of $k$ and $\delta$, hence the whole sum can be bounded by $2\sum_{k\geq2}e^{-c_2k^2/(\beta-\alpha)}\leq c_1'e^{-4c_2/(\beta-\alpha)}\longrightarrow0$ as $\delta\to0$. On the second region, it is straightforward to see that when $\delta\to0$ the double integral becomes $\int_{-\infty}^{0}\!dz_1\int_{-\infty}^{0}\!dz_2$ and $\int_{0}^{\infty}\!dz_1\int_{0}^{\infty}\!dz_2$, respectively, when $k=1$ and $k=-1$. The same Gaussian bounds which we just used allow us now to replace the original limits in the integrals by the ones we just indicated and take the $\delta\to0$ limit inside. These facts, together with \eqref{eq:limdelUps} and the fact that $\varTheta_1^{{\rm be}/{\rm rbb}}\bigl(x,\pm r\ttsm\cosh(t)\bigr)=\varTheta_2^{{\rm be}/{\rm rbb}}\bigl(\pm r\ttsm\cosh(t+\delta),y\bigr)=0$, imply that
\begin{multline}
\lim_{\delta\to 0}\frac{1}{\delta}\left[\p_\beta\Theta_{[-L,L]}^{h_{\beta,\delta},{{\rm be}/{\rm rbb}}}\right]_{\beta=0}(x,y)\\
=\int_{-\infty}^{0}\!dz_1\int_{-\infty}^{0}\!dz_2\,\tfrac{z_1z_2(-z_1-z_2)\ttsm\cosh(t)}{4\sqrt{\pi}}e^{-(z_1+z_2)^2/4}\bigl[\p_w\varTheta_1^{{\rm be}/{\rm rbb}}(x,w)\p_w\varTheta_2^{{\rm be}/{\rm rbb}}(w,y)\bigr]_{w=r\ttsm\cosh(t)}\\
+\int_{0}^{\infty}\!dz_1\int_{0}^{\infty}\!dz_2\,\tfrac{z_1z_2(z_1+z_2)\ttsm\cosh(t)}{4\sqrt{\pi}}e^{-(z_1+z_2)^2/4}\bigl.[\p_w\varTheta_1^{{\rm be}/{\rm rbb}}(x,w)\p_w\varTheta_2^{{\rm be}/{\rm rbb}}(w,y)\bigr]_{w=-r\ttsm\cosh(t)}.
\end{multline}
Again the integral in $z_1$ and $z_2$ evaluates to $2\sqrt{\pi}$. This gives \eqref{eq:belimdelta} and completes the proof for the case of BE/RBB.
\end{proof}

\subsection{Proof of Corollaries \ref{cor:bbCvgce} and \ref{cor:berbbCvgce}}

An explicit expression for the joint density of $\cm$ and $\ct$ (see \eqref{eq:Tdef} and \eqref{eq:Mdef}), which we will denote as $f(r,t)$, was obtained in \cite{mqr}.
To state the formula we need to introduce the function
\begin{equation}
\psi_{r,t}(x)=2e^{xt}\left[t\Ai(x+r+t^2)+\Ai'(x+r+t^2)\right],
\end{equation}
the Airy kernel
\begin{equation}
\K(x,y)=\int_0^\infty d\lambda\Ai(x+\lambda)\Ai(y+\lambda),
\end{equation}
and the rank one kernel $\Psi_{r,t}(x,y)=\varPhi_{r,t}(x)\varPhi_{r,-t}(y)$, where
\begin{equation}
\varPhi_{r,t}(x)=\int_0^\infty d\lambda\Ai(x+\lambda)\psi_{r,t}(\lambda).
\end{equation}
Then the joint density $f(r,t)$ obtained in \cite{mqr} can be rewritten as
\begin{align}
f(r,t)&=\tr\!\left[({\sf I}-\K\varrho_r\K)^{-1}\Psi_{r,t}\right]\det\!\left({\sf I}-\K\varrho_r\K\right)\\
&=\det\!\left({\sf I}-\K\varrho_r\K+\Psi_{r,t}\right)-\det\!\left({\sf I}-\K\varrho_r\K\right).\label{eq:densairy2}
\end{align}
We are going to show that, under suitable scaling, the joint densities of $\cm_N^\star$ and $\ct_N^\star$ for the three models converge to $f(r,t)$, i.e.
\begin{equation}
\frac{1}{4\sqrt{N}}f_N^{\rm bb}\Bigl(\sqrt{N}+\frac{r}{2N^{1/6}},\frac12+\frac{t}{2N^{1/3}}\Bigr)\xrightarrow[N\to\infty]{}f(r,t)
\end{equation}
and
\begin{equation}
\frac{1}{4\sqrt{2N}}f_N^{{\rm be}/{\rm rbb}}\Bigl(\sqrt{2N}+\frac{r}{2^{7/6}N^{1/6}},\frac12+\frac{t}{2^{4/3}N^{1/3}}\Bigr)\xrightarrow[N\to\infty]{}f(r,t),
\end{equation}
which will prove the two corollaries.

We start with the BB case. For $r\geq 0$ and $t\in\rr$ let 
\[{\wt r}_N=\sqrt{N}+\tfrac{r}{2N^{1/6}},\qquad {\wt t}_N=\tfrac12+\tfrac{t}{2N^{1/3}}\]
 and recall that $g(t)=\frac{1}{\sqrt{2t(1-t)}}$. A simple scaling argument on the right hand side of \eqref{eq:densBB} leads to
\begin{equation}
\frac{1}{4\sqrt{N}}f_N^{\rm bb}\bigl({\wt r}_N,{\wt t}_N\bigr)=\det\!\left({\sf I}-\wtKGN \varrho_r \wtKGN+{\wt \Psi}^{\rm bb}_{N,r,t}\right)-\det\!\left({\sf I}-\wtKGN \varrho_r \wtKGN\right),
\end{equation}
with $\wtKGN(x,y)=\kappa_N\KGN(\tilde x_N,\tilde y_N)$ and ${\wt \Psi}^{\rm bb}_{N,r,t}(x,y)=2^{-5/2}N^{-2/3}\Psi^{\rm bb}_{{\wt r}_N,{\wt t}_N}(\tilde x_N,\tilde y_N)$, where $\kappa_N=2^{-1/2}N^{-1/6}$, $\tilde x_N=\sqrt{2N}+\kappa_N x$, and $\tilde y_N=\sqrt{2N}+\kappa_N y$. 
On the other hand, it is a basic fact in random matrix theory that $\wtKGN$ converges (in trace norm) to $\K$ as $N\to\infty$ (see e.g. \cite{andGuioZeit}).
In view of \eqref{eq:densairy2} and \eqref{eq:inebound}, it remains to show that
\begin{equation}\label{eq:Psibbconv}
{\wt \Psi}^{\rm bb}_{N,r,t}\xrightarrow[N\to\infty]{}\Psi_{r,t}
\end{equation}
in trace norm. 
Recall that ${\wt\Psi}^{\rm bb}_{N,r,t}$ is a rank one kernel which can be written as (see \eqref{eq:Psibb}) ${\wt\Psi}^{\rm bb}_{N,r,t}={\wt\varPhi}^{\rm bb}_{N,r,t}\otimes{\wt\varPhi}^{\rm bb}_{N,r,-t}$ with
\begin{equation}\label{eq:sumestconv}
{\wt \varPhi}^{\rm bb}_{N,r,t}(x):=2^{-5/4}N^{-1/3}\sum_{n=0}^{N-1}\varphi_n(\sqrt{2N}+\kappa_N x)\psi^{\rm bb}_{{\wt r}_N,{\wt t}_N}(n),
\end{equation}
where $\psi^{\rm bb}_{r,t}$ was defined in \eqref{eq:psibb}.
We then estimate (see \eqref{eq:inenorms})
\begin{align}
\|{\wt\Psi}^{\rm bb}_{N,r,t}-\Psi_{r,t}\|_1&\leq \|{\wt\varPhi}^{\rm bb}_{N,r,t}{\wt\varPhi}^{\rm bb}_{N,r,-t}-\varPhi_{r,t}{\wt\varPhi}^{\rm bb}_{N,r,-t}\|_1+\|\varPhi_{r,t}{\wt\varPhi}^{\rm bb}_{N,r,-t}-\varPhi_{r,t}\varPhi_{r,-t}\|_1\\
&=\|{\wt\varPhi}^{\rm bb}_{N,r,t}-\varPhi_{r,t}\|_2\|{\wt\varPhi}^{\rm bb}_{N,r,-t}\|_2+\|\varPhi_{r,t}\|_2\|{\wt\varPhi}^{\rm bb}_{N,r,-t}-\varPhi_{r,-t}\|_2.
\end{align}
We want to show
\begin{lem}\label{lem:L2Phi}
For $t\in\rr$, the following limit holds in $L^2(\rr)$.
\begin{equation}
{\wt\varPhi}^{\rm bb}_{N,r,t}\xrightarrow[N\to\infty]{}\varPhi_{r,t}.
\end{equation}
\end{lem}
As a consequence we also get that $\|{\wt\varPhi}^{\rm bb}_{N,r,-t}\|_2$ is uniformly bounded in $N$. 
The lemma thus yields \eqref{eq:Psibbconv}, which completes the proof of Corollary \ref{cor:bbCvgce}.

\begin{proof}[Proof of Lemma \ref{lem:L2Phi}]
Setting $\gamma(t)=-\log\left((1+tN^{-1/3})/(1-tN^{-1/3})\right)/2$ and using the definition \eqref{eq:psibb} and Lemma \ref{lem:formulaK}, we can write the sum on the right hand side of \eqref{eq:sumestconv} as
\begin{multline}
2^{-3/4}N^{-1/3}g({\wt t}_N)^{3/2}\left.\Big[\p_y+\frac{t\tts g({\wt t}_N){\wt r}_N}{N^{1/3}}\Big]e^{\gamma(t){\sf D}}\KGN(x,y)\right|_{\subalign{x&=\sqrt{2N}+\kappa_N x\\y&={\wt r}_Ng({\wt t}_N)}}\\
=2^{-7/4}g({\wt t}_N)^{3/2}e^{\gamma(t)(N-\frac{1}{2})}\int_0^\infty dz\,e^{-s(\gamma(t))\left((\sqrt{2N}+\kappa_N x)\kappa_N z+{\wt r}_Ng({\wt t}_N)\kappa_N z+c(\gamma(t))(\kappa_N z)^2\right)}\\
\times\biggl[\varphi_N\bigl(\tau_N^{(1)}(x,z)\bigr)\varphi_{N-1}'\bigl(\tau_N^{(2)}(r,t,z)\bigr)+\varphi_{N-1}\bigl(\tau_N^{(1)}(x,z)\bigr)\varphi_N'\bigl(\tau_N^{(2)}(r,t,z)\bigr)\\
+\bigl(-s(\gamma(t))\kappa_N z+(2{\wt t}_N-1){\wt r}_Ng({\wt t}_N)\bigr)\Bigl[\varphi_N\bigl(\tau_N^{(1)}(x,z)\bigr)\varphi_{N-1}\bigl(\tau_N^{(2)}(r,t,z)\bigr)\\
+\varphi_{N-1}\bigl(\tau_N^{(1)}(x,z)\bigr)\varphi_N\bigl(\tau_N^{(2)}(r,t,z)\bigr)\Bigr]\biggr],
\end{multline}
where we have used $\tau_N^{(1)}(x,z)=\sqrt{2N}+\kappa_N x+c(\gamma(t))\kappa_N z$, $\tau_N^{(2)}(r,t,z)={\wt r}_Ng({\wt t}_N)+c(\gamma(t))\kappa_N z$ with $s(t)=\sinh(t/2)$ and $c(t)=\cosh(t/2)$.
Now we can check that this expression converges to $\varPhi_{r,t}(x)$ in $L^2(\rr)$ by using the known asymptotics $\varphi_N(\sqrt{2N}+\kappa_N x)=2^{1/4}N^{-1/12}\left(\Ai(x)+\mathcal{O}(N^{-2/3})\right)$ and $\varphi_N'(\sqrt{2N}+\kappa_N x)=2^{3/4}N^{1/12}\left(\Ai'(x)+\mathcal{O}(N^{-2/3})\right)$, together with the fact that $\tau_N^{(1)}(x,z)=\sqrt{2N}+\kappa_N(x+z)+\mathcal{O}(N^{-5/6})$, $\tau_N^{(2)}(r,t,z)=\sqrt{2N}+\kappa_N(t^2+r+z)+\mathcal{O}(N^{-1/2})$ and $g({\wt t}_N)=\sqrt{2}+t^2N^{-2/3}/\sqrt{2}+\mathcal{O}(N^{-4/3})$.\\
\end{proof}

We will prove next the convergence in the case of BE, the proof for RBB is very similar. Proceeding analogously to the proof in the BB case, we write ${\wt\Psi}^{\rm be}_{N,r,t}={\wt\varPhi}^{\rm be}_{N,r,t}\otimes{\wt\varPhi}^{\rm be}_{N,r,-t}$ with
\begin{multline}\label{eq:tvarPhibe}
{\wt \varPhi}^{\rm be}_{N,r,t}(x):=2^{-7/12}N^{-1/3}\sum_{n=0}^{N-1}\varphi_{2n+1}(\sqrt{4N}+\kappa_{2N} x)\psi^{\rm bb}_{{\wt r}_{2N},{\wt t}_{2N}}(2n+1)\\
+2^{-1/12}N^{-1/3}\tts g({\wt t}_{2N})^{3/2}\sum_{n=0}^{N-1}e^{\gamma(t)(2n+1)}\varphi_{2n+1}(\sqrt{4N}+\kappa_{2N} x)\sum_{k=1}^{\infty}\psi^{\rm be}_{{\wt r}_{2N},{\wt t}_{2N},k}(2n+1),
\end{multline}
where $\psi^{\rm be}_{r,t,k}(n)$ is the $k^{th}$ summand in the infinite sum in \eqref{eq:psibe}, that is
\begin{multline}
\psi^{\rm be}_{{\wt r}_{2N},{\wt t}_{2N},k}(2n+1)=e^{2k(k+1)(2{\wt t}_{2N}-1){\wt r}_{2N}^2\tts g({\wt t}_{2N})^2}\\
\times\Bigl[\varphi_{2n+1}'\bigl((2k+1){\wt r}_{2N}\tts g({\wt t}_{2N})\bigr)+(2k+1)(2{\wt t}_{2N}-1){\wt r}_{2N}\tts g({\wt t}_{2N})\varphi_{2n+1}\bigl((2k+1){\wt r}_{2N}\tts g({\wt t}_{2N})\bigr)\Bigr].
\end{multline}
We have to show that the function ${\wt \varPhi}^{\rm be}_{N,r,t}$ converges to $\varPhi_{r,t}$ in $L^2(\rr)$. Note that the first sum in \eqref{eq:tvarPhibe} actually converges to $\varPhi_{r,t}$ as in the case of BB. On the other hand, one can check that the second term goes to $0$ by using the asymptotics of the Hermite functions in \eqref{eq:asympHermite} and the fact that ${\wt r}_{2N}\tts g({\wt t}_{2N})=\sqrt{4N}+\kappa_{2N}(t^2+r)+\mathcal{O}(N^{-5/6})$ to estimate that $\varphi_{2n+1}'\bigl((2k+1){\wt r}_{2N}\tts g({\wt t}_{2N})\bigr)\leq c_1\tts e^{-c_2 N k^2}$ and $\varphi_{2n+1}\bigl((2k+1){\wt r}_{2N}\tts g({\wt t}_{2N})\bigr)\leq c_1\tts e^{-c_2 N k^2}$ when $k>0,n\in\{0,\ldots,N-1\}$ for some $c_1,c_2>0$. This together with the uniform bound of Hermite functions $\sup_{x\in\rr}\varphi_n(x)<c$ for any $n\geq 1$ (see e.g. \cite{NIST:DLMF}), imply that $\sum_{n=0}^{N-1}e^{\gamma(t)(2n+1)}\varphi_{2n+1}(\sqrt{4N}+\kappa_{2N} x)\psi^{\rm be}_{{\wt r}_{2N},{\wt t}_{2N},k}(2n+1)\leq c'_1\tts e^{-c'_2 N k^2}$ for $N$ large enough and hence completes the proof of Corollaries \ref{cor:berbbCvgce}.

\section{Small deviations for the argmax for non-intersecting Brownian bridges}\label{sec:bounds}

The proofs in this section follow \cite{quastelRemTails}, where the tails of $\ct$ were studied.
Throughout the section we will use $c_1,c_2,\ldots$ to denote positive constants whose value may change from line to line.

The Hermite kernel has an integral representation given as follows \cite[Sec. 4]{Aubrun2005}:
\begin{equation}\label{eq:integrHermKer}
\sum_{n=0}^{N-1}\varphi_n(x)\varphi_n(y)=\sqrt{\frac{N}{2}}\int_0^{\infty}dz\bigl(\varphi_N(x+z)\varphi_{N-1}(y+z)+\varphi_{N-1}(x+z)\varphi_N(y+z)\bigr).
\end{equation}
In the following lemma we will use this formula to derive an integral representation for the kernel $e^{t{\sf D}}\KGN$, which will be used repeatedly throughout this section:

\begin{lem}\label{lem:formulaK}
For all $t\in\rr$ we have
\begin{multline}
e^{t{\sf D}}\KGN(x,y)=\sqrt{\tfrac{N}{2}}e^{t(N-\frac12)}\int_0^{\infty}dz\biggl[e^{-s(t)((x+y)z+c(t)z^2)}\\
\times\Bigl(\varphi_N\bigl(x+c(t)z\bigr)\varphi_{N-1}\bigl(y+c(t)z\bigr)+\varphi_{N-1}\bigl(x+c(t)z\bigr)\varphi_N\bigl(y+c(t)z\bigr)\Bigr)\biggr],
\end{multline}
where $s(t)=\sinh(t/2)$ and $c(t)=\cosh(t/2)$.
\end{lem}

The proof depends on the following result:

\begin{lem}\label{lem:etDvarphi}
Given $t,z\in\rr$, define the shifted Hermite function $\varphi_{n,z}(x)=\varphi_n(x+z)$ and the function $\theta_{n,t,z}(x)=e^{tn-\sinh(t)(xz+\cosh(t)z^2/2)}\varphi_{n,\cosh(t)z}(x)$.
Then for all $s<0, t\in\rr$ we have
\begin{equation}\label{eq:etDtheta}
e^{s{\sf D}}\theta_{n,t,z}(x)=\theta_{n,s+t,z}(x).
\end{equation}
In particular, $e^{t{\sf D}}\varphi_{n,z}(x)=\theta_{n,t,z}(x)$ and $e^{t{\sf D}}\theta_{n,-t,z}(x)=\varphi_{n,z}(x)$ for all $t<0$. As a consequence, $e^{t{\sf D}}\varphi_{n,z}$ is well defined for all $t\in\rr$ via the formula
\begin{equation}\label{eq:etD}
e^{t{\sf D}}\varphi_{n,z}(x)=e^{tn-\sinh(t)(xz+\cosh(t)z^2/2)}\varphi_{n,\cosh(t)z}(x),
\end{equation}
and it satisfies the semigroup property in the sense that $e^{(s+t){\sf D}}\varphi_{n,z}(x)=e^{s{\sf D}}e^{t{\sf D}}\varphi_{n,z}(x)$ for all $s,t\in\rr$.
\end{lem}

\begin{proof}[Proof of Lemma \ref{lem:etDvarphi}]
The operator $e^{s{\sf D}}$ for $s<0$ is well defined and its integral kernel is given in \eqref{eq:etDkernel} (where we take $-\ell_1=\ell_2=s/2$):
\begin{equation}
e^{s{\sf D}}(x,y)=\frac{1}{\sqrt{\pi(e^{-s}-e^{s})}}\,e^{\frac12(y^2-x^2)-s/2-(e^{s/2}x-e^{-s/2}y)^2/(e^{-s}-e^{s})}.
\end{equation}
This formula together with the contour integral representation of the shifted Hermite function $\varphi_{n,\cosh(t)z}(x)$ (see \eqref{eq:integhermite}),
\begin{equation}
\varphi_{n,\cosh(t)z}(x)=(2^nn!\sqrt{\pi})^{-1/2}e^{-(x+\cosh(t)z)^2/2}\frac{n!}{2\pi\I}\oint dw \frac{e^{2w(x+\cosh(t)z)-w^2}}{w^{n+1}}
\end{equation}
(where the contour of integration encircles the origin), gives us
\begin{multline}
e^{s{\sf D}}\theta_{n,t,z}(x)=\bigl(2^nn!\sqrt{\pi}\bigr)^{-1/2}\frac{n!}{2\pi\I}
\int_{-\infty}^{\infty}dy\oint dw \frac{e^{tn-\sinh(2t)z^2/4}}{w^{n+1}\sqrt{\pi(e^{-s}-e^{s})}}\\
\times e^{\frac12(y^2-x^2)-s/2-(e^{s/2}x-e^{-s/2}y)^2/(e^{-s}-e^{s})-\sinh(t)yz-(y+\cosh(t)z)^2/2+2w(y+\cosh(t)z)-w^2}.
\end{multline}
We can compute the $y$ integral first, which is just a Gaussian integral, to obtain
\begin{equation}
(2^nn!\sqrt{\pi})^{-1/2}\frac{n!}{2\pi\I}e^{tn-\sinh(s+t)(xz+\cosh(s+t)z^2/2)-(x+\cosh(s+t)z)^2/2}\oint dw \frac{e^{2e^{s}w(x+\cosh(s+t)z)-e^{2s}w^2}}{w^{n+1}}.
\end{equation}
By changing $w\mapsto we^{-s}$, we see that the last integral is nothing but $\varphi_{n,\cosh(s+t)z}(x)$, which prove \eqref{eq:etDtheta}. The remaining statements in the lemma follow directly from this identity.
\end{proof}

\begin{proof}[Proof of Lemma \ref{lem:formulaK}]
Since $e^{t{\sf D}}$ and $\KGN$ commute, we have $e^{t{\sf D}}\KGN=e^{\frac12t{\sf D}}\KGN e^{\frac12t{\sf D}}$.
The formula now follows directly from the integral representation of $\KGN$ given in \eqref{eq:integrHermKer} and \eqref{eq:etD}.
\end{proof}

In order to estimate the tails $\ct^{\rm bb}_N$ it will be more convenient for us to work instead with
\begin{equation}
  \widehat\ct^{\rm dbm}_N=\argmax_{t\in\rr}\frac{\lambda_N(t)}{\cosh(t)}\qqand\widehat\cm^{\rm dbm}_N=\max_{t\in\rr}\frac{\lambda_N(t)}{\cosh(t)},
\end{equation}
where we recall that $\lambda_N(t)$ is the top line of the stationary GUE Dyson Brownian motion (see \eqref{eq:dbm-bb}).
More precisely, we will prove the following:

\begin{thm}\label{thm:hatTNtails}
  There are constants $c_1,c_2,c_3,n_0>0$ and $t_0>1/3$ such that
  \begin{equation}
    c_1\tts e^{-c_2Nt^3}\leq\pp\big(\widehat\ct_N^{\rm dbm}>t\big)\leq c_3\tts e^{-\frac43Nt^3+\mathcal{O}(N^{2/3})},
  \end{equation}
  with the upper bound holding uniformly in $N\in\nn$ and $t\in(0,1)$ satisfying $Nt^3\geq n_0$ and the lower bound holding uniformly in $N\in\nn$ and $t\in(0,t_0)$ satisfying $Nt^3\geq n_0$.
\end{thm}

To recover Theorem \ref{thm:TNtails} from this result, observe first that, by symmetry and \eqref{eq:dbm-bb},
\[\pp\big(|\ct_N^{\rm bb}-\tfrac12|>\ep\big)=2\ts\pp\big(\ct_N^{\rm bb}>\tfrac12+\ep\big)
=2\ts\pp\!\left(\widehat\ct_N^{\rm dbm}>\tfrac12\log\!\big(\tfrac{1+2\ep}{1-2\ep}\big)\right).\]
Hence we may apply the above theorem to get
\begin{equation}
    c_1\tts e^{-c_2Nt(\ep)^3}\leq\pp\big(\big|\ct_N^{\rm bb}-\tfrac12\big|>\ep\big)\leq c_3\ts e^{-\frac43Nt(\ep)^3+\mathcal{O}(N^{2/3})}
\end{equation}
with $t(\ep)=\frac12\log(\frac{1+2\ep}{1-2\ep})$ for $N$ large enough so that $N\log(\frac{1+2\ep}{1-2\ep})^3\geq8n_0$ and provided that $\ep$ is small enough so that $t$ is in the right regime for each bound.
Since $t(\ep)\geq2\ep$ for $\ep\in[0,1)$ and $t(\ep)\in(0,1)$ for $\ep\in[0,\ep_2)$, the upper bound in Theorem \ref{thm:TNtails} follows for all $\ep\in[0,\ep_2)$.
Similarly, $\ep\in[0,\ep_1)$ ensures that $t(\ep)<1/3<t_0$, and we also have that $t(\ep)\leq c\ts\ep$ for some $c>0$ and $\ep$ in this range, so the lower bound follows in the same way.

The rest of this section is devoted to the proof of Theorem \ref{thm:hatTNtails}.
We will assume throughout that $t\in(0,1)$ and that $N\in\nn$ is large enough so that $Nt^3\geq n_0$ where $n_0>0$ is a large parameter which will be chosen in order to make all estimates work.
Throughout the proof we will make extensive use of Laplace's method for estimating integrals, see for instance \cite[Sec. 2.4]{erdelyi}.

\subsection{Upper bound}\label{sec:upper}

We start by writing, for any $t\in(0,1)$ and $N\in\nn$,
\begin{multline}\label{eq:iniUpperBound}
\pp\!\left(\widehat\ct^{\rm dbm}_N> t\right)\leq\pp\!\left(\widehat\ct^{\rm dbm}_N> t,\,\widehat\cm^{\rm dbm}_N> \sqrt{2N}(1-N^{-1/3}\alpha t)\right)\\
+\pp\!\left(\widehat\cm^{\rm dbm}_N\leq \sqrt{2N}(1-N^{-1/3}\alpha t)\right),
\end{multline}
where $\alpha>0$ is a parameter which will be chosen shortly.
By Theorem \ref{thm:nibm-loe}, the second probability on the right hand side equals $F_{{\rm LOE},N}\left(4N(1-N^{-1/3}\alpha t)^2\right)$.
By \cite[Thm. 2]{ledouxRider}, there are constants $c_1,c_2>0$ such that for any $\delta\in(0,1]$ we have 
\begin{equation}
F_{{\rm LOE},N}(4N(1-\delta))\leq c_1\ts e^{-c_2N^2\delta^3}.\label{eq:riderledoux}
\end{equation}
Choosing $\alpha$ large enough so that $c_2\alpha^3>\frac43$, and then $N$ large enough so that $\alpha t<N^{1/3}$, we get
\begin{equation}\label{eq:smalldevLOE}
F_{{\rm LOE},N}\Bigl(4N(1-N^{-1/3}\alpha t)^2\Bigr)\leq F_{{\rm LOE},N}\Bigl(4N(1-N^{-1/3}\alpha t)\Bigr)
\leq c_1 e^{-c_2\alpha^3 Nt^3}\leq c_1 e^{-\frac43Nt^3}
\end{equation}
as desired.
We are thus left with obtaining the same bound for the first probability on the right hand side of \eqref{eq:iniUpperBound}.

We express this last probability as 
\begin{equation}
\pp\!\left(\widehat\ct^{\rm dbm}_N> t,\widehat\cm^{\rm dbm}_N> \sqrt{2N}(1-N^{-1/3}\alpha t)\right)=\int_t^{\infty}ds\int_{\sqrt{2N}(1-N^{-1/3}\alpha t)}^{\infty}dm\,\widehat f_N^{\rm bb}(m,s).
\end{equation}
Using the second identity in Theorem \ref{thm:whargmax} 
together with \eqref{eq:inebound} we see that the last integral is bounded by
\begin{equation}\label{eq:Psibbbound}
\int_t^{\infty}ds\int_{\sqrt{2N}(1-N^{-1/3}\alpha t)}^{\infty}dm\,\|\widehat\Psi^{\rm bb}_{m,s}\|_1 e^{1+2\|\KGN \varrho_m \KGN\|_1+\|\widehat\Psi^{\rm bb}_{m,s}\|_1}.
\end{equation}
Thus we need to estimate the two trace norms appearing above.
We will estimate first $\|\widehat\Psi^{\rm bb}_{m,s}\|_1$.
Since $\widehat\Psi^{\rm bb}_{m,s}$ is rank one, this norm can be written as
\begin{equation}\label{eq:rank1Psibb}
\|\widehat\Psi^{\rm bb}_{m,s}\|_1=\Bigl\|\sum_{n=0}^{N-1}\varphi_n(\cdot)\widehat\psi^{\rm bb}_{m,s}(n)\Bigr\|_2\Bigl\|\sum_{n=0}^{N-1}\varphi_n(\cdot)\widehat\psi^{\rm bb}_{m,-s}(n)\Bigr\|_2,
\end{equation}
where we have used \eqref{eq:whPsibb}.
We have, for all $s\in\rr$,
\begin{align}
&\Bigl\|\sum_{n=0}^{N-1}\varphi_n(\cdot)\widehat\psi^{\rm bb}_{m,s}(n)\Bigr\|^2_2=\int_{-\infty}^{\infty}dx\sum_{n,k=0}^{N-1}\varphi_n(x)\varphi_k(x)\widehat\psi^{\rm bb}_{m,s}(n)\widehat\psi^{\rm bb}_{m,s}(k)=\sum_{n=0}^{N-1}\widehat\psi^{\rm bb}_{m,s}(n)^2\\
&\quad=2\cosh(s)\sum_{n=0}^{N-1}\Bigl[e^{-2ns}\varphi_n'(m\cosh(s))^2+2m\sinh(s)e^{-2ns}\varphi_n'(m\cosh(s))\varphi_n(m\cosh(s))\\
&\hspace{3.1in}+m^2\sinh(s)^2e^{-2ns}\varphi_n(m\cosh(s))^2\Bigr]\\
&\quad=2\cosh(s)\left.\Big[\p_x\p_y+m\sinh(s)(\p_x+\p_y)+m^2\sinh(s)^2\Big]e^{-2s{\sf D}}\KGN(x,y)\right|_{x=y=m\cosh(s)}\\
&\quad=2\sqrt{2N}\cosh(s)e^{-s(2N-1)-\sinh(2s)m^2/2}\int_{0}^{\infty}dz\,e^{\sinh(2s)(m+z)^2/2}H_N\bigl(s,(m+z)\cosh(s)\bigr)
\end{align}
with
\begin{multline}\label{eq:deffuncH}
  H_N(s,x):=x^2\tanh(s)^2\varphi_N(x)\varphi_{N-1}(x)\\
  +x\tanh(s)\bigl(\varphi_N'(x)\varphi_{N-1}(x)+\varphi_N(x)\varphi_{N-1}'(x)\bigr)+\varphi_N'(x)\varphi_{N-1}'(x).
\end{multline}
where we have used the orthogonality of the family $(\varphi_n)_{n\in\nn}$, the definition of $\widehat\psi^{\rm bb}_{m,s}$ in \eqref{eq:whpsibb} and Lemma \ref{lem:formulaK}.
Using this identity for both norms on the right hand side of \eqref{eq:rank1Psibb} (note that the exponential factors in front of each integral cancel) we obtain
\begin{multline}
\|\widehat\Psi^{\rm bb}_{m,s}\|_1=2\sqrt{2N}\cosh(s)\left[\int_{0}^{\infty}dz\,e^{\sinh(2s)(m+z)^2/2}H_N\bigl(s,(m+z)\cosh(s)\bigr)\right]^{1/2}\\
\times\left[\int_{0}^{\infty}dz\,e^{-\sinh(2s)(m+z)^2/2}H_N\bigl(-s,(m+z)\cosh(s)\bigr)\right]^{1/2}.
\end{multline}
We now focus on the integral $\int_{\sqrt{2N}(1-N^{-1/3}\alpha t)}^{\infty}dm\,\|\widehat\Psi^{\rm bb}_{m,s}\|_1$. By using Cauchy-Schwarz inequality, we first have
\begin{align}
\int_{\sqrt{2N}(1-N^{-1/3}\alpha t)}^{\infty}&dm\,\|\widehat\Psi^{\rm bb}_{m,s}\|_1\leq 2\sqrt{2N}\cosh(s)\\
&\times\left[\int_0^{\infty}dm\left|\int_{0}^{\infty}dz\,e^{\sinh(2s)(m+z)^2/2}H_N\bigl(s,(m+z)\cosh(s)\bigr)\right|\right]^{1/2}\\
&\times\left[\int_0^{\infty}dm\left|\int_{0}^{\infty}dz\,e^{-\sinh(2s)(m+z)^2/2}H_N\bigl(-s,(m+z)\cosh(s)\bigr)\right|\right]^{1/2},
\end{align}
then performing the change of variables $m\longmapsto\sqrt{2N}\wt{m}$ with $\wt{m}:=1-N^{-1/3}\alpha t+m$ and $z\longmapsto \sqrt{2N}z$ to write the right hand side above as
\begin{multline}\label{eq:intnormPsi}
2(2N)^{3/2}\cosh(s)\left[\int_0^{\infty}dm\left|\int_{0}^{\infty}dz\,e^{N\sinh(2s)(\wt{m}+z)^2}H_N\bigl(s,\sqrt{2N}(\wt{m}+z)\cosh(s)\bigr)\right|\right]^{1/2}\\
\times\left[\int_0^{\infty}dm\left|\int_{0}^{\infty}dz\,e^{-N\sinh(2s)(\wt{m}+z)^2}H_N\bigl(-s,\sqrt{2N}(\wt{m}+z)\cosh(s)\bigr)\right|\right]^{1/2}.
\end{multline}
At this stage we need the following asymptotic approximations of the Hermite functions $\varphi_N(\sqrt{2N}x)$ and their derivatives $\varphi_N'(\sqrt{2N}x)$ when $x\in(1,\infty)$ (see \cite[Sec. 4]{Skov1959}):
\begin{equation}\label{eq:asympHermite}
\begin{aligned}
\varphi_N(\sqrt{2N}x)&=\frac{1}{\sqrt{2\pi}(2N)^{1/4}(x^2-1)^{1/4}}e^{-Nh(x)+1/2\log(x+\sqrt{x^2-1})}\left[1+\mathcal{O}\!\left(\tfrac{1}{N(x-1)^{3/2}}\right)\right],\\
\varphi_N'(\sqrt{2N}x)&=-\frac{(2N)^{1/4}(x^2-1)^{1/4}}{\sqrt{2\pi}}e^{-Nh(x)+1/2\log(x+\sqrt{x^2-1})}\left[1+\mathcal{O}\!\left(\tfrac{1}{N(x-1)^{3/2}}\right)\right],
\end{aligned}
\end{equation}
where $h(x)=x\sqrt{x^2-1}-\log(x+\sqrt{x^2-1})>0$ for $x>1$ and the error terms are uniform in $x\in(1,\infty)$ as $N^{2/3}(x-1)\to\infty$.
The same asymptotics hold for $\varphi_{N-1}(\sqrt{2N}x)$ and $\varphi_{N-1}'(\sqrt{2N}x)$.
Now observe that, since $\cosh(s)\geq 1+s^2/2$ and $\tilde m=1-N^{-1/3}\alpha t+m$, we have, for $m,z\geq 0$, $s\geq t$ with $t\in(0,1)$,
\begin{equation}
(\wt{m}+z)\cosh(s)
\geq 1+t^2\bigl(\tfrac{1}{2}-\tfrac{\alpha}{tN^{1/3}}-\tfrac{\alpha}{2N^{1/3}}\bigr)
\geq 1+\tfrac{n_0^{2/3}}{N^{2/3}}\Bigl(\tfrac{1}{2}-\tfrac{3\alpha}{2n_0^{1/3}}\Bigr),
\end{equation}
where the last bound follows from $N\geq Nt^3\geq n_0$.
Choosing $n_0$ large enough the right hand side is larger than 1, and thus (in view of \eqref{eq:deffuncH}) we may use the Hermite function asymptotics in \eqref{eq:intnormPsi} to get
\begin{multline}
H_N\bigl(s,\sqrt{2N}(\wt{m}+z)\cosh(s)\bigr)\leq c_1\sqrt{N}f_s(\wt{m}+z)\\
\times e^{-2N\left(\cosh(s)(\wt{m}+z)\sqrt{\cosh(s)^2(\wt{m}+z)^2-1}-\log\bigl(\cosh(s)(\wt{m}+z)+\sqrt{\cosh(s)^2(\wt{m}+z)^2-1}\bigr)\right)}
\end{multline}
for some $c_1>0$, where
\begin{equation}
f_s(x)=\frac{\sinh(s)^2x^2}{(\cosh(s)^2x^2-1)^{1/2}}+\sinh(s)x+(\cosh(s)^2x^2-1)^{1/2}.
\end{equation}
Applying now the general identity $\int_0^{\infty}dx\int_0^{\infty}dy\,g(x+y)=\int_0^{\infty}dx\,xg(x)$ (observing that $\tilde m+z$ is a function of $m+z$) together with the above upper bound we obtain
\begin{multline}
\int_0^{\infty}dm\int_{0}^{\infty}dz\,e^{N\sinh(2s)(\wt{m}+z)^2}H_N\bigl(s,\sqrt{2N}(\wt{m}+z)\cosh(s)\bigr)\\
\leq c_1\sqrt{N}\int_0^{\infty}dz\,zf_s(1-N^{-1/3}\alpha t+z)e^{-Ng_s(1-N^{-1/3}\alpha t+z)},\label{eq:bd11}
\end{multline}
where
\begin{equation}
g_s(z)=-\sinh(2s)z^2+2\cosh(s)z\sqrt{\cosh(s)^2z^2-1}-2\log\bigl(\cosh(s)z+\sqrt{\cosh(s)^2z^2-1}\bigr).
\end{equation}
Since $g_s'(z)=-4\cosh(s)(\sinh(s)z-\sqrt{\cosh(s)^2z^2-1})$, $g_s'(1)=0$ and $g_s''(1)=4\tanh(s)^{-1}>0$, $g_s(z)$ attains its minimum for $z\geq0$ at $z=1$ with $g_s(1)=-2s$, and thus it follows from a simple application of Laplace's method that the major contribution to the $z$ integral comes from the neighborhood of the point $z=N^{-1/3}\alpha t$ and hence the right hand side of \eqref{eq:bd11} is bounded by
\begin{equation}
\frac{\alpha t\sinh(s)^{3/2}}{N^{1/3}\cosh(s)^{1/2}}e^{2Ns}\left(c_1+\mathcal{O}\!\left[\tfrac{1}{t^3N^{3/2}}\right]\right),
\end{equation}
where the error term is uniform in $s\in(t,\infty)$. Hence when $Nt^3\geq n_0$, we have
\begin{equation}
\int_0^{\infty}dm\int_{0}^{\infty}dz\,e^{N\sinh(2s)(\wt{m}+z)^2}H_N\bigl(s,\sqrt{2N}(\wt{m}+z)\cosh(s)\bigr)
\leq c_1\frac{\alpha t\sinh(s)^{3/2}}{N^{1/3}\cosh(s)^{1/2}}e^{2Ns}.
\end{equation}
This yields a bound for the first integral in \eqref{eq:intnormPsi}.
The second integral can be estimated in the same way, leading to 
\begin{equation}
\int_0^{\infty}dm\int_{0}^{\infty}dz\,e^{-N\sinh(2s)(\wt{m}+z)^2}H_N\bigl(-s,\sqrt{2N}(\wt{m}+z)\cosh(s)\bigr)
\leq c_1\frac{\alpha t\sinh(s)^{3/2}}{N^{1/3}\cosh(s)^{1/2}}e^{2N(s-\sinh(2s))}.
\end{equation}
We deduce that
\begin{align}
\int_t^{\infty}ds\int_{\sqrt{2N}(1-N^{-1/3}\alpha t)}^{\infty}dm\,\|\widehat\Psi^{\rm bb}_{m,s}\|_1&\leq c_1 N^{7/6}\tts t\int_t^{\infty}ds\sinh(s)^{3/2}\cosh(s)^{1/2}e^{(2s-\sinh(2s))N}\\
&\leq c_1 N^{1/6}e^{2t}e^{N\left(2t-\sinh(2t)\right)},
\end{align}
where the last estimate can be obtained using Laplace's method and the Taylor series expansion of the hyperbolic sine and cosine.
This also implies that $\|\widehat\Psi^{\rm bb}_{m,s}\|_1$ is bounded above by a constant uniformly for $s\geq t$ and $m\geq \sqrt{2N}(1-N^{-1/3}\alpha t)$.
Since $2t-\sinh(2t)\leq-\frac43t^3$ for $t\geq0$, we deduce then from \eqref{eq:Psibbbound} and Lemma \ref{lem:KrhoKbound} below that, for $Nt^3$ large enough, that
\begin{equation}
\pp\!\left(\widehat\ct^{\rm dbm}_N> t,\widehat\cm^{\rm dbm}_N> \sqrt{2N}(1-N^{-1/3}\alpha t)\right)\leq c_1e^{-\frac{4}{3}Nt^3+\mathcal{O}(N^{2/3})}.
\end{equation}
This estimate together \eqref{eq:smalldevLOE} yield the desired upper bound.

\begin{lem}\label{lem:KrhoKbound}
There are constants $c,n_0>0$ such that for all $N\geq n_0$, $t\in(0,1)$ and $m\geq \sqrt{2N}(1-t/N^{1/3})$,
\begin{equation}
\|\KGN \varrho_m \KGN\|_1\leq c\ts N^{2/3}.
\end{equation}
\end{lem}
\begin{proof}[Proof of Lemma \ref{lem:KrhoKbound}]
For any $a>0$ define the multiplication operator $(e^{\xi a}f)(x)=e^{ax}f(x)$ and write
\begin{equation}
\|\KGN \varrho_m \KGN\|_1\leq\|\KGN e^{\xi a}\|_2\|e^{-\xi a}\varrho_m \KGN\|_2.
\end{equation}
We have
\begin{equation}
\|\KGN e^{\xi a}\|_2^2=\int_{\rr^2}dx\,dy\left(\sum_{n=0}^{N-1}\varphi_n(x)\varphi_n(y)e^{ya}\right)^2=\sum_{n=0}^{N-1}\int_{\rr}dy\,e^{2ya}\varphi_n(y)^2,
\end{equation}
and
\begin{equation}
\|e^{-\xi a}\varrho_m \KGN\|_2^2=\sum_{n=0}^{N-1}\int_{\rr}dx\,e^{-2xa}\varphi_n(2m-x)^2=e^{-4am}\sum_{n=0}^{N-1}\int_{\rr}dx\,e^{2xa}\varphi_n(x)^2,
\end{equation}
(in both cases we have used the orthogonality of the family $(\varphi_n)_{n\in\nn}$), which yields the bound
\begin{equation}
\|\KGN \varrho_m \KGN\|_1\leq e^{-2am}\sum_{n=0}^{N-1}\int_{\rr}dx\,e^{2xa}\varphi_n(x)^2.
\end{equation}
We split the $x$ integral into two regions, $({-\infty},{\sqrt{2N}}]$ and $({\sqrt{2N}},{\infty})$. On the first one we use the upper bound $\sum_{n=0}^{N-1}\varphi_n(x)^2\leq\sqrt{N/2}$ for $x\in\rr$ (see \cite{Levin1992}) to estimate the integral by
\begin{equation}\label{eq:intxregion1}
\sum_{n=0}^{N-1}\int_{-\infty}^{\sqrt{2N}}dx\,e^{2xa}\varphi_n(x)^2\leq\frac{\sqrt{N}}{2\sqrt{2}a}e^{2\sqrt{2N}a}.
\end{equation}
On the second one we use \eqref{eq:integrHermKer} (with the change of variable $x\mapsto\sqrt{2N}(1+x)$) and the bound 
\begin{equation}
n^{1/12}\varphi_n\bigl(\sqrt{2n}+x/(\sqrt{2}n^{1/6})\bigr)\leq c_1\ts e^{-c_2x^{3/2}}\quad\forall x>0,n\in\nn
\end{equation}
for the Hermite function (see \cite{Aubrun2005}) to write
\begin{align}\label{eq:bd12}
\sum_{n=0}^{N-1}\int_{\sqrt{2N}}^{\infty}dx\,e^{2xa}\varphi_n(x)^2&=\int_0^{\infty}dx\int_0^{\infty}dz\,(2N)^{3/2}e^{2\sqrt{2N}a(1+x)}\\
&\hspace{0.3in}\times\varphi_{N}\bigl(\sqrt{2N}(1+x+z)\bigr)\varphi_{N-1}\bigl(\sqrt{2N}(1+x+z)\bigr)\\
&\leq c_1 N^{4/3}e^{2\sqrt{2N}a}\int_0^{\infty}dx\int_0^{\infty}dz\,e^{2\sqrt{2N}ax-c_2(x+z)^{3/2}N}\\
&\leq c_1 N^{4/3}e^{2\sqrt{2N}a}\int_0^{\infty}dx\,e^{2\sqrt{2N}ax-c_2x^{3/2}N}\int_0^{\infty}dz\,e^{-c_2z^{3/2}N}.
\end{align}
The $z$ integral can be computed explicitly, and equals $c\ts N^{-2/3}$ for some $c>0$, while we can bound the $x$ integral by $c_1e^{c_2a^3/\sqrt{N}}$ using Laplace's method and hence the right hand side of \eqref{eq:bd12} can be bounded by $c_1N^{2/3}e^{2\sqrt{2N}a+c_2a^3/\sqrt{N}}$ for large enough $N$.
Putting this bound together with \eqref{eq:intxregion1} gives
\begin{equation}
\|\KGN \varrho_m \KGN\|_1\leq e^{-2am}\left(\frac{\sqrt{N}}{2\sqrt{2}a}e^{2\sqrt{2N}a}+c_1N^{2/3}e^{2\sqrt{2N}a+c_2a^3/\sqrt{N}}\right).
\end{equation}
Since $m\geq \sqrt{2N}(1-t/N^{1/3})$ and $t\in(0,1)$ deduce that
\begin{equation}
\|\KGN \varrho_m \KGN\|_1\leq e^{2\sqrt{2N}\frac{a}{N^{1/3}}}\left(\frac{\sqrt{N}}{2\sqrt{2}a}+ c_1N^{2/3}e^{c_2a^3/\sqrt{N}}\right),
\end{equation}
which finishes the proof by choosing $a=N^{-1/6}$.
\end{proof}

\subsection{Lower bound}\label{sec:lower}

Proceeding analogously to the proof of the upper bound in \cite{quastelRemTails}, we start by writing, for $N\in\nn$, $t\in(0,1)$, and two parameters $\beta>0$ and $s\in(0,1)$ to be chosen later on,
\begin{equation}\label{eq:inelbound1}
\pp\!\left(\widehat\ct^{\rm dbm}_N> t\right)\geq\pp\!\left(\tfrac{\lambda_N(x)}{\cosh(x)}\leq \sqrt{2N}\!\cosh(\beta t)\ \forall x\leq t;\ \tfrac{\lambda_N(t+s)}{\cosh(t+s)}> \sqrt{2N}\!\cosh(\beta t)\right).%\\
\end{equation}
The basic idea of the proof in \cite{quastelRemTails} is the following.
Suppose for a moment that the two events in the probability on the right hand side were independent.
The first event has a probability which can be bounded away from zero (see below), so the lower bound is controlled by $\pp\bigl(\tfrac{\lambda_N(t+s)}{\cosh(t+s)}> \sqrt{2N}\!\cosh(\beta t)\bigr)$. 
This last probability has the desired tail decay if we choose $s=\alpha t$ for some $\alpha\in(0,1)$.
The proof thus boil down to estimating the correction coming from the correlation between the two events.
To do this, we rewrite \eqref{eq:inelbound1} as
\begin{multline}\label{eq:inelbound2}
\pp\!\left(\widehat\ct^{\rm dbm}_N> t\right)\geq\pp\!\left(\tfrac{\lambda_N(x)}{\cosh(x)}\leq \sqrt{2N}\!\cosh(\beta t)\ \forall x\leq t\right)\\
-\pp\!\left(\tfrac{\lambda_N(x)}{\cosh(x)}\leq \sqrt{2N}\!\cosh(\beta t)\ \forall x\leq t;\ \tfrac{\lambda_N(t+s)}{\cosh(t+s)}\leq \sqrt{2N}\!\cosh(\beta t)\right).
\end{multline}
The correlation between the two events in the last probability is controled by the following estimate:

\begin{lem}\label{lem:mainuppbound}
Let $\beta>3$. There are $\alpha_0,n_0>0$ such that if $\alpha\in(0,\alpha_0)$ and $s=\alpha t$, then for all $N\in\nn$, $t\in(0,1)$ satisfying $Nt^3>n_0$, we have
\begin{multline}\label{eq:mainuppbd}
\pp\!\left(\tfrac{\lambda_N(x)}{\cosh(x)}\leq\sqrt{2N}\!\cosh(\beta t)\ \forall x\leq t;\ \tfrac{\lambda_N(t+s)}{\cosh(t+s)}\leq \sqrt{2N}\!\cosh(\beta t)\right)\\
\leq \pp\!\left(\tfrac{\lambda_N(x)}{\cosh(x)}\leq \sqrt{2N}\!\cosh(\beta t)\ \forall x\leq t\right)\pp\!\left(\tfrac{\lambda_N(t+s)}{\cosh(t+s)}\leq \sqrt{2N}\!\cosh(\beta t)\right)\\
\times\left[1+\tfrac{a_1}{2Nt^3}e^{-\frac{4}{3}N\left((\beta^2+(1+\alpha)^2)^{3/2}t^3+\mathcal{O}(t^{5})\right)}\right],
\end{multline}
where $a_1$ is defined implicitly in \eqref{eq:ubfiniteGUE}.
\end{lem}

To see how the lower bound follows from the lemma, we let $\beta>3$, choose $\alpha$ as in the lemma, let $s=\alpha t$ and then use the estimate and \eqref{eq:inelbound2} to get
\begin{multline}
\pp\!\left(\widehat\ct^{\rm dbm}_N> t\right)\geq\pp\!\left(\tfrac{\lambda_N(x)}{\cosh(x)}\leq \sqrt{2N}\!\cosh(\beta t)\ \forall x\leq t\right)\\
\times\left[1-\pp\bigl(\tfrac{\lambda_N(t+s)}{\cosh(t+s)}\leq \sqrt{2N}\!\cosh(\beta t)\bigr)\Bigl(1+\tfrac{a_1}{2Nt^3}e^{-\frac{4}{3}N\left((\beta^2+(1+\alpha)^2)^{3/2}t^3+\mathcal{O}(t^{5})\right)}\Bigr)\right].
\end{multline}
For the first probability on the right hand side we have that there is a $p_0>0$
 such that
 \begin{equation}
\pp\!\left(\tfrac{\lambda_N(x)}{\cosh(x)}\leq \sqrt{2N}\!\cosh(\beta t)\ \forall x\leq t\right)\geq\pp\!\left(\tfrac{\lambda_N(x)}{\cosh(x)}\leq \sqrt{2N}\ \forall x\in\rr\right)=F_{{\rm LOE},N}(4N)\geq p_0
\end{equation}
uniformly in $N$.
On the other hand, since $\cosh(\beta t)\cosh(t+s)=1+\frac{\beta^2+(1+\alpha)^2}{2}t^2+\mathcal{O}(t^4)$ for $t\in(0,1)$ and since $\lambda_N(t)$ has the distribution $F_{{\rm GUE},N}$ of the largest eigenvalue of a $N\times N$ GUE random matrix (with scaling chosen as in \cite{nibm-loe}), Lemma \ref{lem:tailGUEN} implies that
\begin{align}
\pp\!\left(\tfrac{\lambda_N(t+s)}{\cosh(t+s)}\leq \sqrt{2N}\!\cosh(\beta t)\right)&=F_{{\rm GUE},N}(\sqrt{2N}\!\cosh(\beta t)\cosh(t+s))\\
&\leq 1-\tfrac{a_1}{Nt^3}e^{-\frac{4}{3}N\left((\beta^2+(1+\alpha)^2)^{3/2}t^3+\mathcal{O}(t^{5})\right)}.\label{eq:ubfiniteGUE}
\end{align}
We deduce that
\begin{align}
\pp\!\left(\widehat\ct^{\rm dbm}_N> t\right)&\geq p_0\!\left[\tfrac{a_1}{2Nt^3}\ts e^{-\frac{4}{3}N\left((\beta^2+(1+\alpha)^2)^{3/2}t^3+\mathcal{O}(t^{5})\right)}+\tfrac{a_1^2}{2N^2t^6}\ts e^{-\frac{8}{3}N\left((\beta^2+(1+\alpha)^2)^{3/2}t^3+\mathcal{O}(t^{5})\right)}\right]\\
&\geq\tfrac{c_1}{Nt^3}e^{-\frac{4}{3}N\left((\beta^2+(1+\alpha)^2)^{3/2}t^3+\mathcal{O}(t^{5})\right)},
\end{align}
which yields the lower bound.

Our goal then is prove Lemma \ref{lem:mainuppbound}. For this we need an expression for the probability of the form $\pp\left(\lambda_N(x)\leq a\cosh(x)\ \forall x\leq t\right)$. To state the extension of that formula, define, for $a,t\in\rr$, the operator (acting on $L^2(\rr)$)
\begin{equation}\label{eq:notaQM}
{\sf Q}={\sf P}_{a\cosh(t)}({\sf I}+{\sf M}_{a,t}\varrho_{a,t}),
\end{equation}
where
\begin{equation}
\varrho_{a,t}f(x)=f(2a\cosh(t)-x)\quad\text{and}\quad{\sf M}_{a,t}f(x)=e^{2a\sinh(t)(x-a\cosh(t))}f(x).
\end{equation}

\begin{prop}\label{prop:22kernel}
With the above definitions, and for any $a,b\in\rr$ and $s>0$,
\begin{align}
\pp&\bigl(\lambda_N(x)\leq a\cosh(x)\ \forall x\leq t;\ \lambda_N(t+s)\leq b\cosh(t+s)\bigr)\\
&=\det\!\left({\sf I}-\KGN+\KGN({\sf I}-{\sf Q})e^{-s{\sf D}}\left({\sf I}-{\sf P}_{b\cosh(t+s)}\right)e^{s{\sf D}}\KGN\right)_{L^2(\rr)}\label{eq:22kernel1}\\
&=\det\!\left({\sf I}-{\sf\Gamma}\left[\begin{matrix}
	{\sf Q}\KGN{\sf P}_{a\cosh(t)} & {\sf Q}e^{-s{\sf D}}(\KGN-{\sf I}){\sf P}_{b\cosh(t+s)}\\
  {\sf P}_{b\cosh(t+s)}e^{s{\sf D}}\KGN {\sf P}_{a\cosh(t)} & {\sf P}_{b\cosh(t+s)}\KGN{\sf P}_{b\cosh(t+s)}
	\end{matrix}\right]{\sf\Gamma}^{-1}\right)_{\!\!L^2(\rr)^2}\label{eq:22kernel2}
\end{align}
where
\begin{equation}\label{eq:notaG}
{\sf\Gamma}=\left[\begin{matrix}
	{\sf G} & 0\\
	0 & {\sf G}
	\end{matrix}\right]\quad\text{with}\quad {\sf G}f(x)=e^{-2a\sinh(t)x}f(x).
\end{equation}
\end{prop}

\begin{proof}[Proof of Proposition \ref{prop:22kernel}]
We will only prove \eqref{eq:22kernel1}.
The proof of \eqref{eq:22kernel2} follows from the same argument as that in the proof of \cite[Prop. 3.3]{quastelRemTails}, and is basically a version of the argument in \cite{bcr} (see also \cite{prolhacSpohn,quastelRemAiry1}).

Given $L>0$, it is straightforward to adapt the proof given in \cite[Cor. 4.5]{bcr} of the continuum statistics formula \eqref{eq:dbmcont} to deduce that
\begin{multline}\label{eq:prove22kernel1}
\pp\bigl(\lambda_N(x)\leq a\cosh(x)\ \forall x\in[-L,t];\ \lambda_N(t+s)\leq b\cosh(t+s)\bigr)\\
=\det\!\left({\sf I}-\KGN+\Theta_{[-L,t]}^{(a),{\rm bb}}e^{-s{\sf D}}\bar{\sf P}_{b\cosh(t+s)}e^{(L+t+s){\sf D}}\KGN\right),
\end{multline}
where $\Theta_{[-L,t]}^{(a),{\rm bb}}$ is defined as $\Theta_{[-L,t]}^{g,{\rm bb}}$ (see \eqref{eq:thetarRpre}) for $g=a\cosh(t)$. Since $e^{(t+s){\sf D}}\KGN=e^{t{\sf D}}\KGN e^{s{\sf D}}\KGN$ for all $t,s\in\rr$, we can use the cyclic property of the determinant to turn the last determinant into
\begin{equation}\label{eq:prove22kernel2}
\det\!\left({\sf I}-\KGN+e^{(L+t){\sf D}}\KGN\Theta_{[-L,t]}^{(a),{\rm bb}}e^{-s{\sf D}}\bar{\sf P}_{b\cosh(t+s)}e^{s{\sf D}}\KGN\right).
\end{equation}
We will show below that
\begin{equation}\label{eq:limitThetaa}
e^{(L+t){\sf D}}\KGN\Theta_{[-L,t]}^{(a),{\rm bb}}\xrightarrow[L\to\infty]{}\KGN({\sf I}-{\sf M}_{a,t}\varrho_{a,t})\bar{\sf P}_{a\cosh(t)}
\end{equation}
in trace norm. This together with \eqref{eq:prove22kernel1} and \eqref{eq:prove22kernel2} yields that the probability in the Proposition equals
\begin{equation}
\det\!\left({\sf I}-\KGN+\KGN({\sf I}-{\sf M}_{a,t}\varrho_{a,t})\bar{\sf P}_{a\cosh(t)}e^{-s{\sf D}}\bar{\sf P}_{b\cosh(t+s)}e^{s{\sf D}}\KGN\right).
\end{equation}
Now formula \eqref{eq:22kernel1} readily follows by observing that ${\sf M}_{a,t}$ and ${\sf P}_{a\cosh(t)}$ commute and $\varrho_{a,t}\bar{\sf P}_{a\cosh(t)}={\sf P}_{a\cosh(t)}\varrho_{a,t}$.

All that remains is to prove \eqref{eq:limitThetaa}. The proof follows from the same arguments as that in the proofs of \cite[Lem. 2.3 and Lem. 2.4]{nibm-loe} (in which case we were taking $t=0$). We decompose $\Theta_{[-L,t]}^{(a),{\rm bb}}$ as
\begin{equation}
\Theta_{[-L,t]}^{(a),{\rm bb}}=\Big[e^{-(t+L){\sf D}}-{\sf R}^{(a),{\rm bb}}_{[-L,t]}\Big]\bar {\sf P}_{a\ttsm\cosh(t)}-\Omega_{t,L}^{(a)},
\end{equation}
where $\Omega_{t,L}^{(a)}={\sf P}_{a\ttsm\cosh(L)}\Big[e^{-(t+L){\sf D}}-{\sf R}^{(a),{\rm bb}}_{[-L,t]}\Big]\bar {\sf P}_{a\ttsm\cosh(t)}$. The first term leads to 
\[e^{(L+t){\sf D}}\KGN\big[e^{-(t+L){\sf D}}-{\sf R}^{(a),{\rm bb}}_{[-L,t]}\big]=\KGN({\sf I}-{\sf M}_{a,t}\varrho_{a,t})\]
(see \cite[Lem. 2.4]{nibm-loe}) while the remaining term $e^{(L+t){\sf D}}\KGN\Omega_{t,L}^{(a)}$ converges to $0$ in trace norm (see \cite[Appx. B]{nibm-loe}).
\end{proof}

\begin{proof}[Proof of Lemma \ref{lem:mainuppbound}]
We start by using \eqref{eq:22kernel2} with $a=b=\sqrt{2N}\cosh(\beta t)$.
To simplify notation, write ${\sf P}_1={\sf P}_{a\cosh(t)},{\sf P}_2={\sf P}_{a\cosh(t+s)}$.
The idea of the proof (which comes from \cite{widomAiry2}) is to factor out the two diagonal terms in the determinant and then estimate the remainder.
More precisely, we write
\begin{multline}
{\sf I}-{\sf\Gamma}\left[\begin{matrix}
	{\sf Q}\KGN{\sf P}_1 & {\sf Q}e^{-s{\sf D}}(\KGN-{\sf I}){\sf P}_2\\
  {\sf P}_2e^{s{\sf D}}\KGN{\sf P}_1 & {\sf P}_2\KGN{\sf P}_2
	\end{matrix}\right]{\sf\Gamma}^{-1}=\left({\sf I}-{\sf\Gamma}\left[\begin{matrix}
	{\sf Q}\KGN{\sf P}_1 & 0\\
  0 & {\sf P}_2\KGN{\sf P}_2
	\end{matrix}\right]{\sf\Gamma}^{-1}\right)\\
\times\left({\sf I}-{\sf\Gamma}\left[\begin{matrix}
	0 & ({\sf I}-{\sf Q}\KGN{\sf P}_1)^{-1}{\sf Q}e^{-s{\sf D}}(\KGN-{\sf I}){\sf P}_2\\
  ({\sf I}-{\sf P}_2\KGN{\sf P}_2)^{-1}{\sf P}_2e^{s{\sf D}}\KGN{\sf P}_1 & 0
	\end{matrix}\right]{\sf\Gamma}^{-1}\right).
\end{multline}
The determinant of the first factor on the right hand side equals
\begin{equation}
\det\!\left({\sf I}-{\sf G}{\sf Q}\KGN{\sf P}_1{\sf G}^{-1}\right)\det\!\left({\sf I}-{\sf G}{\sf P}_2\KGN{\sf P}_2{\sf G}^{-1}\right).
\end{equation}
The second determinant equals $F_{{\rm GUE},N}(a\cosh(t+s))=\pp\!\left(\lambda_N(t+s)\leq a\cosh(t+s)\right)$.
For the first one we have, by the cyclic property of determinants and the facts that ${\sf P}_1{\sf Q}={\sf Q}$ and $\KGN=(\KGN)^2$,
\begin{equation}
\det\!\left({\sf I}-{\sf G}{\sf Q}\KGN{\sf P}_1{\sf G}^{-1}\right)=\det\!\left({\sf I}-\KGN{\sf Q}\KGN\right)=\pp\!\left(\lambda_N(x)\leq a\cosh(x)\ \forall x\leq t\right)
\end{equation}
by \eqref{eq:22kernel1} where we take $s=0$ and $b=a$.
This yields the first two factors on the right hand side of \eqref{eq:mainuppbd}.

We are left with estimating
\begin{multline}
\det\!\left({\sf I}-{\sf\Gamma}\left[\begin{matrix}
	0 & ({\sf I}-{\sf Q}\KGN{\sf P}_1)^{-1}{\sf Q}e^{-s{\sf D}}(\KGN-{\sf I}){\sf P}_2\\
  ({\sf I}-{\sf P}_2\KGN{\sf P}_2)^{-1}{\sf P}_2e^{s{\sf D}}\KGN{\sf P}_1 & 0
	\end{matrix}\right]{\sf\Gamma}^{-1}\right)_{L^2(\rr)^2}\\
  =\det({\sf I}-\wt{\sf K}),
\end{multline}
with $\wt{\sf K}={\sf R}_{1,1}{\sf R}_{1,2}{\sf R}_{2,2}{\sf R}_{2,1}$ and
\begin{equation}\label{eq:4R}
\begin{aligned}
{\sf R}_{1,1}&={\sf G}({\sf I}-{\sf Q}\KGN{\sf P}_1)^{-1}{\sf G}^{-1},&{\sf R}_{1,2}&={\sf G}{\sf Q}e^{-s{\sf D}}(\KGN-{\sf I}){\sf P}_2,\\
{\sf R}_{2,2}&=({\sf I}-{\sf P}_2\KGN{\sf P}_2)^{-1},&{\sf R}_{2,1}&={\sf P}_2e^{s{\sf D}}\KGN{\sf P}_1{\sf G}^{-1}.
\end{aligned}
\end{equation}
Since $|\!\det({\sf I}-\wt{\sf K})-\det({\sf I})|\leq\|\wt{\sf K}\|_1 e^{1+\|\wt{\sf K}\|_1}$, the proof will be complete once we show that, for $Nt^3$ large enough,
\begin{equation}\label{eq:Rupneeded}
\|\wt{\sf K}\|_1\leq\frac{a_1}{2e^2Nt^3}e^{-\frac{4}{3}N\left((\beta^2+(\alpha+1)^2)^{3/2}t^3+\mathcal{O}(t^{5})\right)}.
\end{equation}
To get this estimate write (see \eqref{eq:inenorms})
$\|\wt{\sf K}\|_1\leq \|{\sf R}_{1,1}\|_2\|{\sf R}_{1,2}\|_2\|{\sf R}_{2,2}\|_2\|{\sf R}_{2,1}\|_2$,
and then use Lemma \ref{lem:4Rbounds}, which gives
\begin{equation}
\|\wt{\sf K}\|_1\leq c_1N^{-5/4}t^{-15/4}e^{-N\left(\frac{4}{3}(\beta^2+(\alpha+1)^2)^{3/2}t^3+h_{\beta}(\alpha)t^3+\mathcal{O}(t^5)\right)},
\end{equation}
where $h_{\beta}(\alpha)=\alpha+2\alpha^2+\frac{2}{3}\alpha^3+\alpha\beta^2+\frac{2}{3}(1+\beta^2)^{3/2}-\frac{2}{3}(\beta^2+(1+\alpha)^2)^{3/2})$.
Since, for fixed $\beta>3$, we have $h_{\beta}(0)=0$ and $h_{\beta}'(0)=1+\beta^2-2\sqrt{1+\beta^2}>0$, we deduce that $h_{\beta}(\alpha)>0$ for small enough $\alpha$ and therefore that \eqref{eq:Rupneeded} holds for small enough $\alpha$ and large enough $Nt^3$ as desired.
\end{proof}

\begin{lem}\label{lem:4Rbounds}
Let ${\sf R}_{1,1}$, ${\sf R}_{1,2}$, ${\sf R}_{2,2}$ and ${\sf R}_{2,1}$ be defined as in \eqref{eq:4R}. There are constants $c_1, n_0>0$ and a constant $t_0>1/3$ such that if $0<t<t_0$, $Nt^3\geq n_0$ and $\beta\geq 3$,
\begin{subequations}
\begin{align}
\|{\sf R}_{1,1}\|_2&\leq 2,\label{eq:4R1}\\ 
\|{\sf R}_{1,2}\|_2&\leq c_1N^{-1/4}t^{-3/4}e^{-N\left((4+\alpha)t+\left((4+\alpha)\beta^2+2\alpha^3/3+2\alpha^2+\alpha+8/3\right)t^3+\mathcal{O}(t^5)\right)},\label{eq:4R2}\\ 
\|{\sf R}_{2,2}\|_2&\leq 2,\label{eq:4R3}\\ 
\|{\sf R}_{2,1}\|_2&\leq c_1N^{-1}t^{-3}e^{-N\left((-\alpha-4)t+\left(2(\beta^2+(1+\alpha)^2)^{3/2}/3+2(1+\beta^2)^{3/2}/3-4\beta^2-8/3\right)t^3+\mathcal{O}(t^5)\right)}.\qquad\label{eq:4R4}
\end{align}
\end{subequations}
\end{lem}

\begin{proof}[Proof of Lemma \ref{lem:4Rbounds}]
Recall the notation introduced in \eqref{eq:notaQM} and \eqref{eq:notaG}. In the present case we have $a=b=\sqrt{2N}\cosh(\beta t)$. For notational simplicity we will write 
\[{\sf P}_1={\sf P}_{a\cosh(t)},\quad {\sf P}_2={\sf P}_{a\cosh(t+s)},\quad {\sf M}={\sf M}_{a,t} \qand \varrho=\varrho_{a,t}.\]
We will use repeatedly the asymptotics for Hermite functions in \eqref{eq:asympHermite} and the decomposition, for $s\in\rr$, (see Lemma \ref{lem:formulaK})
\begin{equation}\label{eq:decompHermKer}
e^{s{\sf D}}\KGN=\sqrt{N/2}\,e^{s(N-1/2)}\cosh(s/2)^{-1}\bigl[{\sf B}_{N,s}{\sf F}_s{\sf P_0}{\sf F}_s{\sf B}_{N-1,s}+{\sf B}_{N-1,s}{\sf F}_s{\sf P_0}{\sf F}_s{\sf B}_{N,s}\bigr],
\end{equation}
where
\[{\sf B}_{N,s}(x,y)=e^{-\tanh(s/2)(xy)}\varphi_N(x+y)\qqand %${\sf F}_s$ is the multiplication operator defined by 
{\sf F}_sf(x)=e^{-\tanh(s/2)x^2/2}f(x).\]
 Note that for the case $s=0$, \eqref{eq:decompHermKer} simply becomes \eqref{eq:integrHermKer}:
\begin{equation}
\KGN=\sqrt{N/2}\left({\sf B}_N{\sf P}_0{\sf B}_{N-1}+{\sf B}_{N-1}{\sf P}_0{\sf B}_N\right),
\end{equation}
with ${\sf B}_N(x,y):={\sf B}_{N,0}(x,y)=\varphi_N(x+y)$.

We start now with the first estimate. Since
\begin{equation}
\|{\sf R}_{1,1}\|_2\leq \sum_{k\geq 0}\|({\sf GQ}\KGN{\sf P}_1{\sf G}^{-1})^{k}\|_2\leq \sum_{k\geq 0}\|{\sf GQ}\KGN{\sf P}_1{\sf G}^{-1}\|_2^{k}<2
\end{equation}
if $\|{\sf GQ}\KGN{\sf P}_1{\sf G}^{-1}\|_2<1/2$, it is enough to show that
\begin{equation}\label{eq:4R1oneterm}
\|{\sf GQ}\KGN{\sf P}_1{\sf G}^{-1}\|_2\leq c_1e^{-c_2Nt^{3}},
\end{equation}
for large enough $Nt^3$.
Let ${\sf N}$ be the multiplication operator defined by ${\sf N}f(x)=\ell_N(x)^{-1}f(x)$ where $\ell_N(x)=((1+t^2+xN^{-1/6})^2-1)^{3/4}$ for $N\in\nn$, $t\in(0,1)$. We have, recalling that ${\sf Q}={\sf P}_1+{\sf P}_1{\sf M}\varrho$,
\begin{multline}\label{eq:estR11}
\|{\sf GQ}\KGN{\sf P}_1{\sf G}^{-1}\|_2\leq \sqrt{N/2}\ts\Bigl(\|{\sf G}{\sf P}_1{\sf B}_N{\sf P}_0\|_2\|{\sf P}_0{\sf B}_{N-1}{\sf P}_1{\sf G}^{-1}\|_2+\|{\sf G}{\sf P}_1{\sf B}_{N-1}{\sf P}_0\|_2\|{\sf P}_0{\sf B}_N{\sf P}_1{\sf G}^{-1}\|_2\\
+\|{\sf G}{\sf P}_1{\sf M}\varrho{\sf B}_N{\sf P}_0{\sf N}\|_2\|{\sf N}^{-1}{\sf P}_0{\sf B}_{N-1}{\sf P}_1{\sf G}^{-1}\|_2+\|{\sf G}{\sf P}_1{\sf M}\varrho{\sf B}_{N-1}{\sf P}_0{\sf N}\|_2\|{\sf N}^{-1}{\sf P}_0{\sf B}_N{\sf P}_1{\sf G}^{-1}\|_2\Bigr).
\end{multline}
We will focus on the first and the third terms in the sum on the right hand side. The bounds for the two remaining terms are very similar. We write first
\begin{equation}\label{eq:gpbp}
\|{\sf G}{\sf P}_1{\sf B}_N{\sf P}_0\|_2^2=\int_{\sqrt{2N}\cosh(\beta t)\cosh(t)}^{\infty}dx\int_0^{\infty}dy\,e^{-4\sqrt{2N}\cosh(\beta t)\sinh(t)x}\varphi_N(x+y)^2,
\end{equation}
and
\begin{equation}\label{eq:pbpg}
\|{\sf P}_0{\sf B}_{N-1}{\sf P}_1{\sf G}^{-1}\|_2^2=\int_0^{\infty}dx\int_{\sqrt{2N}\cosh(\beta t)\cosh(t)}^{\infty}dy\,e^{4\sqrt{2N}\cosh(\beta t)\sinh(t)y}\varphi_{N-1}(x+y)^2.
\end{equation}
In order to deal with both integrals we are going to use the following estimate:

\begin{lem}\label{lem:intHermite}
Let $\alpha>1$. There are constants $c_1,n_0>0$ such that for all $a\in(1,\alpha)$, all $b<2\sqrt{2(a-1)}$, and all $N\in\nn$ satisfying $\min\!\left\{N(a-1)^{3/2},N(2\sqrt{2a-2}-b)^3\right\}\geq n_0$,
\begin{equation}
\int_0^\infty dx\int_{a\sqrt{2N}}^\infty dy\,e^{b\sqrt{2N}y}\varphi_N(x+y)^2\leq\frac{c_1}{N^{7/6}\sqrt{a^2-1}(2\sqrt{2(a-1)}-b)}e^{-N\left(\frac{8\sqrt{2}}{3}(a-1)^{3/2}-2ab\right)}.
\end{equation}
\end{lem}

\begin{proof}[Proof of Lemma \ref{lem:intHermite}]
Changing variables $x\longmapsto\sqrt{2N}x$ and $y\longmapsto\sqrt{2N}(y+a)$ and then using the asymptotics \eqref{eq:asympHermite} we see that, for $N(a-1)^{3/2}\geq n_0$ with $n_0$ large enough, the double integral is bounded by
\begin{equation}
c_1e^{2Nab}N\int_0^\infty dx\int_0^\infty dy\,\frac{(x+y+a)+\left((x+y+a)^2-1\right)^{1/2}}{N^{1/2}\left((x+y+a)^2-1\right)^{1/2}}e^{-2N\left(h(x+y+a)-by\right)},
\end{equation}
where $h(x)=x\sqrt{x^2-1}-\log(x+\sqrt{x^2-1})$.
Since the function $x\mapsto \tfrac{x+\sqrt{x^2-1}}{\sqrt{x^2-1}}$ is decreasing on $(1,\infty)$ and $h(x)\geq \frac{4\sqrt{2}}{3}(x-1)^{3/2}$ for $x\geq 1$, the above integral is bounded by
\begin{multline}\label{eq:estxint}
c_1\frac{a+\sqrt{a^2-1}}{\sqrt{a^2-1}}e^{2Nab}N^{1/2}\int_0^\infty dx\int_0^\infty dy\,e^{-2N\left(\frac{4\sqrt{2}}{3}(x+y+a-1)^{3/2}-by\right)}\\
\leq \frac{c_1}{\sqrt{a^2-1}}e^{2Nab}N^{-1/6}\int_0^\infty dy\,e^{-N\left(\frac{8\sqrt{2}}{3}(a-1+y)^{3/2}-2by\right)},
\end{multline}
where we used the inequality $(x+y)^{3/2}\geq x^{3/2}+y^{3/2}$ for $x,y\geq 0$  and then computed the explicit $x$ integral $\int_0^\infty dx\,e^{-\frac{8\sqrt{2}}{3}Nx^{3/2}}=c\ts N^{-2/3}$ for some constant $c>0$.
For $y\geq0$ the exponent in the $y$ integral is maximized at $y=0$ as long as $b<2\sqrt{2(a-1)}$ so it follows from Laplace's method that
\begin{multline}\label{eq:estyint}
\int_0^\infty dy\,e^{-N\left(\frac{8\sqrt{2}}{3}(a-1+y)^{3/2}-2by\right)}=\frac{c_1}{N(2\sqrt{2a-2}-b)}e^{-N\frac{8\sqrt{2}}{3}(a-1)^{3/2}}\\
\times\left[1+\mathcal{O}\!\left(\tfrac{1}{N(a-1)^{1/2}(2\sqrt{2(a-1)}-b)^{2}}\right)\right],\quad\text{as}\ N\to\infty.
\end{multline}
Since $a$ is bounded and $\min\!\left\{N(a-1)^{3/2},N(2\sqrt{2(a-1)}-b)^3\right\}\geq n_0$, we observe that the estimate holds by taking $n_0$ large enough and moreover the error term is bounded by an arbitrarily small constant $c_2 n_0^{-1}$. Putting \eqref{eq:estyint} and \eqref{eq:estxint} together to complete the proof.
\end{proof}

We will apply this result setting $a_t=\cosh(\beta t)\cosh(t)$ in both \eqref{eq:gpbp} and \eqref{eq:pbpg}, and $b_t=-4\cosh(\beta t)\sinh(t)<0$ for \eqref{eq:gpbp}, $b_t=4\cosh(\beta t)\sinh(t)$ for \eqref{eq:pbpg} (note that $b_t\geq 2\sqrt{2a_t-2}$ for $\beta>3$ and $t\in(0,1/3)$ in this case).
In both cases, we have $a_t-1\geq c_1t^2$ and $2\sqrt{2a_t-2}-b_t\geq c_1t$ with some explicit constant $c_1>0$ for $t\in(0,1)$, so the condition $\min\!\left\{N(a_t-1)^{3/2},N(2\sqrt{2a_t-2}-b_t)^3\right\}\geq n_0$ appearing in the lemma holds if we let $Nt^3\geq n_0$.
Thus, for $Nt^3$ large enough with $t\in(0,1/3)$, we get
\begin{align}
\|{\sf G}{\sf P}_1{\sf B}_N{\sf P}_0\|_2^2&\leq c_1N^{-7/6}t^{-2}e^{-4N\cosh(\beta t)^2\sinh(2t)-\frac{8\sqrt{2}}{3}N(\cosh(\beta t)\cosh(t)-1)^{3/2}},\\
\|{\sf P}_0{\sf B}_{N-1}{\sf P}_1{\sf G}^{-1}\|_2^2&\leq c_1N^{-7/6}t^{-2}e^{4N\cosh(\beta t)^2\sinh(2t)-\frac{8\sqrt{2}}{3}N(\cosh(\beta t)\cosh(t)-1)^{3/2}}.
\end{align}
This yields a bound of $c_1N^{-7/6}t^{-2}e^{-\frac{8\sqrt{2}}{3}N(\cosh(\beta t)\cosh(t)-1)^{3/2}}$ for the first term in the sum in \eqref{eq:estR11}.
Turning now to the third term in the same sum, we have
\begin{align}
&\|{\sf G}{\sf P}_1{\sf M}\varrho{\sf B}_N{\sf P}_0{\sf N}\|_2^2\\
&\hspace{0.3in}=\int_{a\cosh(t)}^{\infty}dx\int_0^{\infty}dy\,e^{-2a^2\sinh(2t)}\varphi_N(2a\cosh(t)-x+y)^2((1+t^2+yN^{-1/6})^2-1)^{-3/2}\\
&\hspace{0.3in}\leq e^{-4N\cosh(\beta t)^2\sinh(2t)}\|\varphi_N\|_2^2\|{\sf P}_0\ell_N^{-1}\|_2^2\leq c_1N^{1/6}t^{-1}e^{-4N\cosh(\beta t)^2\sinh(2t)},
\end{align}
and we can see that, by proceeding analogously to the proof of Lemma \ref{lem:intHermite}, the following estimate holds
\begin{equation}
\|{\sf N}^{-1}{\sf P}_0{\sf B}_{N-1}{\sf P}_1{\sf G}^{-1}\|_2^2\leq c_1N^{-7/6}t\ts e^{4N\cosh(\beta t)^2\sinh(2t)-\frac{8\sqrt{2}}{3}N(\cosh(\beta t)\cosh(t)-1)^{3/2}}.
\end{equation}
Putting this together with the last estimate and the analog bounds for the other two terms in the sum in \eqref{eq:estR11} gives
\begin{equation}
\|{\sf GQ}\KGN{\sf P}_1{\sf G}^{-1}\|_2\leq c_1e^{-\frac{4\sqrt{2}}{3}N(\cosh(\beta t)\cosh(t)-1)^{3/2}},
\end{equation}
for large enough $Nt^3$ which, since $(\cosh(\beta t)\cosh(t)-1)^{3/2}\geq c_2t^3$ for $t\in(0,1)$, gives \eqref{eq:4R1oneterm}.

We turn now to ${\sf R}_{1,2}$, for which we have
\begin{equation}\label{eq:estR12}
\|{\sf R}_{1,2}\|_2\leq\|{\sf G}{\sf P}_1e^{-s{\sf D}}(\KGN-{\sf I}){\sf P}_2\|_2+\|{\sf G}{\sf P}_1{\sf M}\varrho e^{-s{\sf D}}(\KGN-{\sf I}){\sf P}_2\|_2.
\end{equation}
The first term on the right hand side can be estimated as
\begin{align}
\|{\sf G}{\sf P}_1e^{-s{\sf D}}(\KGN-{\sf I}){\sf P}_2\|_2^2&=\int_{a\cosh(t)}^\infty dx\int_{a\cosh(t+s)}^\infty dy\,e^{-4a\sinh(t)x}\Bigl(\sum_{n=N}^\infty e^{-sn}\varphi_n(x)\varphi_n(y)\Bigr)^2\\
&\leq\int_{a\cosh(t)}^\infty dx\int_{a\cosh(t+s)}^\infty dy\,e^{-2a^2\sinh(2t)}\Bigl(\sum_{n=N}^\infty e^{-sn}\varphi_n(x)\varphi_n(y)\Bigr)^2\\
&\leq\int_{a\cosh(t+s)}^\infty dy\,e^{-2a^2\sinh(2t)}\sum_{n=0}^\infty e^{-2sn}\varphi_n(y)^2,\label{eq:est1stR12}
\end{align}
where in the last inequality we have extended the $x$ integral to the whole real line then used the orthogonality of the family $(\varphi_n)_{n\in\nn}$. Note that the sum is nothing but $\left.e^{-2s{\sf D}}(x,y)\right|_{x=y}$ (see \eqref{eq:etDkernel}), hence the last term becomes
\begin{equation}
e^{-2a^2\sinh(2t)}\int_{a\cosh(t+s)}^\infty dy\,\tfrac{1}{\sqrt{2\pi\sinh(2s)}}e^{(s\sinh(2s)-2\sinh(s)^2y^2)/\sinh(2s)}.
\end{equation}
We then use the estimate $\int_t^\infty dx\,e^{-x^2}\leq e^{-t^2}/(2t)$ for $t>0$ to bound the expression above by
\begin{multline}\label{eq:1stR12}
c_1N^{-1/2}t^{-3/2}e^{-2N\left(2\cosh(\beta t)^2\sinh(2t)+\tanh(s)\cosh(\beta t)^2\cosh(t+s)^2\right)}\\
=c_1N^{-1/2}t^{-3/2}e^{-2N\left((4+\alpha)t+\left((4+\alpha)\beta^2+2\alpha^3/3+2\alpha^2+\alpha+8/3\right)t^3+\mathcal{O}(t^5)\right)}.
\end{multline}
For the remaining term on the right hand side of \eqref{eq:estR12}, by writing
\begin{align}
&\|{\sf G}{\sf P}_1{\sf M}\varrho e^{-s{\sf D}}(\KGN-{\sf I}){\sf P}_2\|_2^2\\
&\hspace{0.7in}=\int_{a\cosh(t)}^\infty dx\int_{a\cosh(t+s)}^\infty dy\,e^{-2a^2\sinh(2t)}\Bigl(\sum_{n=N}^\infty e^{-sn}\varphi_n(2a\cosh(t)-x)\varphi_n(y)\Bigr)^2\\
&\hspace{0.7in}=\int_{-\infty}^{a\cosh(t)} dx\int_{a\cosh(t+s)}^\infty dy\,e^{-2a^2\sinh(2t)}\Bigl(\sum_{n=N}^\infty e^{-sn}\varphi_n(x)\varphi_n(y)\Bigr)^2,
\end{align}
then estimating as in \eqref{eq:est1stR12}, we can obtain the same bound as in the first term, which yields \eqref{eq:4R2}.

For ${\sf R}_{2,2}$ we observe that
\begin{equation}
\|{\sf P}_2\KGN{\sf P}_2\|_2\leq \sqrt{N/2}\ts\bigl(\|{\sf P}_2{\sf B}_N{\sf P}_0\|_2\|{\sf P}_0{\sf B}_{N-1}{\sf P}_2\|_2+\|{\sf P}_2{\sf B}_{N-1}{\sf P}_0\|_2\|{\sf P}_0{\sf B}_N{\sf P}_2\|_2\bigr)
\end{equation}
which can be easily seen to be bounded by $1/2$ for large enough $Nt^3$ by bounds similar to those used to prove \eqref{eq:4R1}, and thus we get \eqref{eq:4R3} in exactly the same way.

Finally, for ${\sf R}_{2,1}$ we use a similar decomposition as for ${\sf R}_{1,1}$ (see \eqref{eq:decompHermKer} and \eqref{eq:estR11}): we may write
\begin{multline}\label{eq:estR21}
\|{\sf R}_{2,1}\|_2\leq \sqrt{N/2}\,e^{s(N-1/2)}\cosh(s/2)^{-1}\Bigl(\|{\sf P}_2{\sf B}_{N,s}{\sf F}_s{\sf P_0}\|_2\|{\sf P_0}{\sf F}_s{\sf B}_{N-1,s}{\sf P}_1{\sf G^{-1}}\|_2\\
+\|{\sf P}_2{\sf B}_{N-1,s}{\sf F}_s{\sf P_0}\|_2\|{\sf P_0}{\sf F}_s{\sf B}_{N,s}{\sf P}_1{\sf G^{-1}}\|_2\Bigr).
\end{multline}
We have
\begin{multline}
\|{\sf P}_2{\sf B}_{N,s}{\sf F}_s{\sf P_0}\|_2^2=\int_{a\cosh(t+s)}^{\infty}dx\int_0^{\infty}dy\,e^{-2\tanh(s/2)(xy+y^2/2)}\varphi_N(x+y)^2\\
\leq \int_{a\cosh(t+s)}^{\infty}dx\int_0^{\infty}dy\,\varphi_N(x+y)^2
\leq c_1N^{-7/6}t^{-2}e^{-\frac{8\sqrt{2}}{3}N(\cosh(\beta t)\cosh(t+s)-1)^{3/2}},
\end{multline}
and
\begin{align}
\|{\sf P_0}{\sf F}_s{\sf B}_{N-1,s}{\sf P}_1{\sf G^{-1}}\|_2^2&=\int_0^{\infty}dx\int_{a\cosh(t)}^{\infty}dy\,e^{-2\tanh(s/2)(xy+x^2/2)+4a\sinh(t)y}\varphi_{N-1}(x+y)^2\\
&\leq \int_0^{\infty}dx\int_{a\cosh(t)}^{\infty}dy\,e^{4a\sinh(t)y}\varphi_{N-1}(x+y)^2\\
&\leq c_1N^{-7/6}t^{-2}e^{4N\cosh(\beta t)^2\sinh(2t)-\frac{8\sqrt{2}}{3}N(\cosh(\beta t)\cosh(t)-1)^{3/2}},
\end{align}
where in the two last inequalities we have used Lemma \ref{lem:intHermite}. Putting these bounds together with the analogous ones for the other term on the right hand side of \eqref{eq:estR21} shows that $\|{\sf R}_{2,1}\|_2$ is bounded by
\begin{multline}
c_1N^{-2/3}t^{-2}e^{-N\left(-s-2\cosh(\beta t)^2\sinh(2t)+\frac{4\sqrt{2}}{3}(\cosh(\beta t)\cosh(t)-1)^{3/2}+\frac{4\sqrt{2}}{3}(\cosh(\beta t)\cosh(t+s)-1)^{3/2}\right)}\\
\leq c_1N^{-2/3}t^{-2}e^{-N\left((-\alpha-4)t+\left(2(\beta^2+(1+\alpha)^2)^{3/2}/3+2(1+\beta^2)^{3/2}/3-4\beta^2-8/3\right)t^3+\mathcal{O}(t^5)\right)},
\end{multline}
which gives \eqref{eq:4R4}.
\end{proof}

\appendix

\section{A small deviation estimate for a finite GUE matrix}
\label{sec:estimates}

Let $\lambda_{{\rm GUE},N}$ be the largest eigenvalue of an $N\times N$ GUE matrix $A$, defined as follows: $A$ is a (complex-valued) Hermitian matrix $A$ such that $A_{ij}=\mathcal{N}\tsm(0,1/4)+\I\ts\mathcal{N}\tsm(0,1/4)$ for $i>j$ and $A_{ii}=\mathcal{N}\tsm(0,1/2)$, where all the Gaussian variables are independent (subject to the Hermitian condition).

\begin{lem}\label{lem:tailGUEN}
There are constants $c_1,c_2,n_0>0$ such that for all $t\in(0,1)$ and $N\in\nn$ satisfying $Nt^{3/2}\geq n_0$,
\begin{equation}
\tfrac{c_1}{Nt^{3/2}}e^{-\frac{8\sqrt{2}}{3}N\left(t^{3/2}+\mathcal{O}(t^{5/2})\right)}\leq \pp\!\left(\lambda_{{\rm GUE},N}\geq \sqrt{2N}(1+t)\right)\leq \tfrac{c_2}{Nt^{3/2}}e^{-\frac{8\sqrt{2}}{3}Nt^{3/2}}.
\end{equation}
\end{lem}

This estimate extends to large $t$ the one appearing in \cite[Lem. 7.3]{zeitouni-paquette} (we remark also that in that paper the dependence on $t$ in the prefactor in the lower bound is missing).

\begin{proof}[Proof of Lemma \ref{lem:tailGUEN}]
We begin by recalling that, under the scaling which we are using for the GUE (see \cite[Sec. 2]{nibm-loe}),
\begin{equation}
\pp\!\left(\lambda_{{\rm GUE},N}\leq t\right)=\det\!\left({\sf I}-{\sf P}_t\KGN{\sf P}_t\right)=\exp\!\left(-\sum_{n=1}^\infty\frac{\tr(({\sf P}_t\KGN{\sf P}_t)^n)}{n}\right).
\end{equation}
As in \cite{zeitouni-paquette}, the second equality comes from the fact that, since $\KGN$ is a positive self-adjoint operator, then all its eigenvalues are non-negative, and so are all the eigenvalues of ${\sf P}_t\KGN{\sf P}_t$. We also have that all traces are non-negative and $\tr(({\sf P}_t\KGN{\sf P}_t)^n)\leq \tr({\sf P}_t\KGN{\sf P}_t)^n$. Thus we get the simple bounds
\begin{equation}
1-\tr({\sf P}_t\KGN{\sf P}_t)\leq \pp\!\left(\lambda_{{\rm GUE},N}\leq t\right)\leq e^{-\tr({\sf P}_t\KGN{\sf P}_t)},
\end{equation}
which implies
\begin{equation}
1-e^{-\tr({\sf P}_t\KGN{\sf P}_t)}\leq \pp\!\left(\lambda_{{\rm GUE},N}> t\right)\leq \tr({\sf P}_t\KGN{\sf P}_t).
\end{equation}
Therefore it only remains to give upper and lower bounds for the trace, which is given by
\begin{equation}
\tr\!\left({\sf P}_{\sqrt{2N}(1+t)}\KGN{\sf P}_{\sqrt{2N}(1+t)}\right)=\int_{\sqrt{2N}(1+t)}^\infty dx\,\KGN(x,x).
\end{equation}
Using the integral representation for the kernel $\KGN$ in \eqref{eq:integrHermKer} and changing variables $x\mapsto\sqrt{2N}(1+t+x)$ and $z\mapsto\sqrt{2N}z$ gives
\begin{multline}
(2N)^{3/2}\int_0^\infty dx\int_0^{\infty}dz\,\varphi_N\bigl(\sqrt{2N}(1+t+x+z)\bigr)\varphi_{N-1}\bigl(\sqrt{2N}(1+t+x+z)\bigr)\\
=(2N)^{3/2}\int_0^\infty dx\,x\varphi_N\bigl(\sqrt{2N}(1+t+x)\bigr)\varphi_{N-1}\bigl(\sqrt{2N}(1+t+x)\bigr).
\end{multline}
We use now the asymptotics for the Hermite functions in \eqref{eq:asympHermite} to deduce that, for $t>0$ as $Nt^{3/2}\to\infty$, the above is bounded by
\begin{equation}\label{eq:asymno1}
c_1N\int_0^\infty dx\,\frac{x}{((1+t+x)^2-1)^{1/2}}e^{-2Nh(1+t+x)}\!\left[1+\mathcal{O}\!\left(\tfrac{1}{Nt^{3/2}}\right)\right],
\end{equation}
where $h(t)=t\sqrt{t^2-1}-\log(t+\sqrt{t^2-1})$. The exponent in the $x$ integral is maximized at $x=0$ so it follows from Laplace's method that, as $N\to\infty$,
\begin{equation}\label{eq:asymno2}
\int_0^\infty dx\,\frac{x}{((1+t+x)^2-1)^{1/2}}e^{-2Nh(1+t+x)}=\frac{e^{-2Nh(1+t)}}{N^2(t^2+2t)^{3/2}}\left[c_1+\mathcal{O}\!\left(\tfrac{1}{N}\right)\right],\quad\forall t\in(0,1).
\end{equation}
Since $N\geq Nt^{3/2}$ for $t\in(0,1)$, both error bounds in \eqref{eq:asymno1} and \eqref{eq:asymno2} are arbitrarily small if we let $Nt^{3/2}\geq n_0$ for $n_0$ large enough. Putting these estimates together to obtain the asymptotics for the trace and then using the expansions $h(1+t)=\frac{4\sqrt{2}}{3}t^{3/2}+\mathcal{O}(t^{5/2})$ (with $h(1+t)\geq\frac{4\sqrt{2}}{3}t^{3/2}$) and $(t^2+2t)^{3/2}=2\sqrt{2}t^{3/2}+\mathcal{O}(t^{5/2})$ for $t\in(0,1)$ completes the claimed bounds.
\end{proof}

\vs
\paragraph{\bf Acknowledgements}
The authors would like to thank Brian Rider for calling our attention to \cite{zeitouni-paquette}.
GBN and DR were partially supported by Conicyt Basal-CMM and by Programa Iniciativa Cient\'ifica Milenio grant number NC120062 through Nucleus Millenium Stochastic Models of Complex and Disordered Systems.
DR was also supported by Fondecyt Grant 1160174.

\printbibliography[heading=apa]

\end{document}